\documentclass[final,11pt]{elsarticle}
\usepackage{srcltx}
\usepackage{eurosym}
\usepackage{mathtools}
\usepackage{amsmath}
\usepackage{amsfonts}
\usepackage{amssymb}
\usepackage{amsthm}
\usepackage{graphicx}
\usepackage{mathrsfs}
\usepackage{xcolor}
\usepackage{exscale}
\usepackage{latexsym}

\usepackage[colorlinks,plainpages=true,pdfpagelabels,hypertexnames=true,colorlinks=true,pdfstartview=FitV,linkcolor=blue,citecolor=red,urlcolor=black]{hyperref}
\PassOptionsToPackage{unicode}{hyperref}
\PassOptionsToPackage{naturalnames}{hyperref}
\usepackage{enumerate}
\usepackage[shortlabels]{enumitem}
\usepackage{bookmark}
\usepackage{wasysym}
\usepackage{esint}
\usepackage[ddmmyyyy]{datetime}
\usepackage[margin=2cm]{geometry}
\makeatletter
\g@addto@macro\normalsize{%
  \setlength\abovedisplayskip{4pt}
  \setlength\belowdisplayskip{4pt}
  \setlength\abovedisplayshortskip{4pt}
  \setlength\belowdisplayshortskip{4pt}
}
\numberwithin{equation}{section}
\everymath{\displaystyle}
\usepackage[capitalize,nameinlink]{cleveref}
\crefname{section}{Section}{Sections}
\crefname{subsection}{Subsection}{Subsections}
\crefname{condition}{Condition}{Conditions}
\crefname{hypothesis}{Hypothesis}{Conditions}
\crefname{assumption}{Assumption}{Assumptions}
\crefname{lemma}{Lemma}{Lemmas}
\crefname{definition}{Definition}{Definitions}

\crefformat{equation}{\textup{#2(#1)#3}}
\crefrangeformat{equation}{\textup{#3(#1)#4--#5(#2)#6}}
\crefmultiformat{equation}{\textup{#2(#1)#3}}{ and \textup{#2(#1)#3}}
{, \textup{#2(#1)#3}}{, and \textup{#2(#1)#3}}
\crefrangemultiformat{equation}{\textup{#3(#1)#4--#5(#2)#6}}%
{ and \textup{#3(#1)#4--#5(#2)#6}}{, \textup{#3(#1)#4--#5(#2)#6}}%
{, and \textup{#3(#1)#4--#5(#2)#6}}

\Crefformat{equation}{#2Equation~\textup{(#1)}#3}
\Crefrangeformat{equation}{Equations~\textup{#3(#1)#4--#5(#2)#6}}
\Crefmultiformat{equation}{Equations~\textup{#2(#1)#3}}{ and \textup{#2(#1)#3}}
{, \textup{#2(#1)#3}}{, and \textup{#2(#1)#3}}
\Crefrangemultiformat{equation}{Equations~\textup{#3(#1)#4--#5(#2)#6}}%
{ and \textup{#3(#1)#4--#5(#2)#6}}{, \textup{#3(#1)#4--#5(#2)#6}}%
{, and \textup{#3(#1)#4--#5(#2)#6}}

\crefdefaultlabelformat{#2\textup{#1}#3}
\numberwithin{equation}{section}

\newtheorem{theorem} {Theorem}[section]

\newtheorem{lemma}{Lemma}[section]

\newtheorem{counter example}{Counter Example}[section]
\newtheorem{remark} {Remark}[section]
\newtheorem{definition} {Definition}[section]
\def\CC{{\rm \kern.24em \vrule width.02em height1.4ex depth-.05ex \kern-.26emC}}

\def\TagOnRight

\def\AA{{it I} \hskip-3pt{\tt A}}

\def\QQ{\rlap {\raise 0.4ex \hbox{$\scriptscriptstyle |$}} {\hskip -0.1em Q}}

\makeatletter
\newcommand{\vo}{\vec{o}\@ifnextchar{^}{\,}{}}
\makeatother
\def\YYint#1#2#3{{\setbox0=\hbox{$#1{#2#3}{\iint}$}
    \vcenter{\hbox{$#2#3$}}\kern-.50\wd0}}
\def\XXint#1#2#3{{\setbox0=\hbox{$#1{#2#3}{\int}$}
    \vcenter{\hbox{$#2#3$}}\kern-.50\wd0}}

\makeatletter
\def\namedlabel#1#2{\begingroup
   \def\@currentlabel{#2}%
   \label{#1}\endgroup
}
\makeatother
\makeatletter
\newcommand{\rmh}[1]{\mathpalette{\raisem@th{#1}}}
\newcommand{\raisem@th}[3]{\hspace*{-1pt}\raisebox{#1}{$#2#3$}}
\makeatother

\newcounter{desccount}
\newcommand{\descitem}[2]{\item[#1]\refstepcounter{desccount}\label{#2}}
\newcommand{\descref}[2]{\hyperref[#1]{\textnormal{\textcolor{black}{}\textcolor{blue}{ #2}\textcolor{black}{}}}}
\newcommand{\dref}[2]{\hyperref[#1]{\textcolor{black}{(}\textcolor{blue}{\bf #2}\textcolor{black}{)}}}
\newcommand{\be} {\begin{eqnarray}}
\newcommand{\ee} {\end{eqnarray}}
\newcommand{\Bea} {\begin{eqnarray*}}
\newcommand{\Eea} {\end{eqnarray*}}

 \newcommand{\al} {\alpha}
\newcommand{\rr}{\rightarrow}
\newcommand{\ti}{\tilde}

\newcommand{\T}  {\theta}

\newcommand{\p}  {\prime}
\newcommand{\e}  {\epsilon}

\newcommand{\f}{\infty}
\newcommand{\R}{\mathbb{R}}

\newcommand{\noi} {\noindent}

\newcommand{\ga}{\gamma}
\newcommand{\abs}[1]{\left| #1\right|}

\newcounter{whitney}
\refstepcounter{whitney}

\newcounter{ineqcounter}
\refstepcounter{ineqcounter}
\makeatletter
\def\ps@pprintTitle{%
\let\@oddhead\@empty
\let\@evenhead\@empty
\def\@oddfoot{}%
\let\@evenfoot\@oddfoot}
\makeatother
\usepackage[doublespacing]{setspace}
\usepackage[titletoc,toc,page]{appendix}

\makeatletter
\newcommand{\refcheckize}[1]{%
  \expandafter\let\csname @@\string#1\endcsname#1%
  \expandafter\DeclareRobustCommand\csname relax\string#1\endcsname[1]{%
    \csname @@\string#1\endcsname{##1}\wrtusdrf{##1}}%
  \expandafter\let\expandafter#1\csname relax\string#1\endcsname
}
\makeatother
\refcheckize{\cref}
\refcheckize{\Cref}


\makeatletter
\newcommand{\mainsectionstyle}{%
	\renewcommand{\@secnumfont}{\bfseries}
	\renewcommand\section{\@startsection{section}{2}%
		\z@{.5\linespacing\@plus.7\linespacing}{-.5em}%
		{\normalfont\bfseries}}%
}
\makeatother
\usepackage{pgf,tikz}
\usetikzlibrary{arrows}
\usetikzlibrary{decorations.pathreplacing}
\usepackage{xpatch}
\xpatchcmd{\MaketitleBox}{\hrule}{}{}{}% remove first horizontal rule (above abstract)
\xpatchcmd{\MaketitleBox}{\hrule}{}{}{}% remoce second horizonral rule (below keywords)

\date{}
\begin{document}
	
	\begin{frontmatter}
		
		\title{Exact and optimal controllability for scalar conservation laws with discontinuous flux}
		
		\author[myaddress]{Adimurthi}
		\ead{aditi@math.tifrbng.res.in}
		%\tnotetext[thankssecondauthor]{Supported in part by Inspire Research Grant.}
		
		\author[myaddress]{Shyam Sundar Ghoshal}
		\ead{ghoshal@tifrbng.res.in}
		
%		\author[myaddress1]{Pierangelo Marcati}
%		% \cortext[mycorrespondingauthor]{Corresponding author}
%		\ead{pierangelo.marcati@gssi.infn.it}
		
		\address[myaddress]{Centre for Applicable Mathematics,Tata Institute of Fundamental Research, Post Bag No 6503, Sharadanagar, Bangalore - 560065, India.}
%		\address[myaddress1]{Gran Sasso Science Institute, Viale F. Crispi, 7, 67100 L'Aquila, Italy.}

		\begin{abstract}
			This paper describes the reachable set and resolves  an optimal control problem   for the scalar  conservation laws with discontinuous
			flux. We give a necessary and sufficient criteria for the reachable set. A new backward resolution has been described to obtain the reachable set.
			Regarding the optimal control problem we first prove the existence of a  minimizer and then the  backward algorithm  allows us to compute it. The same method also  applies to compute the initial data control for an exact control problem. 
			Our methodology for the proof relies on the explicit formula for the conservation laws with the discontinuous flux and  finer properties of the 
			characteristics curves. 		
		\end{abstract}
		
	\end{frontmatter}
	\tableofcontents	
	
	\section{Introduction}\label{introduction}
	
	\setcounter{equation}{0}
	The goal of this paper is to study the reachable sets, optimal controllability and exact controllability  of  the following scalar conservation laws  with discontinuous flux 
	\begin{eqnarray}\label{conlaw-equation}
	\left\{ \begin{array}{lll}
	u_t+F(x,u)_x=0, &\mbox{if}& x\in \R,\ t>0,\\
	\ \ \ \ \ \ \ \  u(x,0)=u_0(x), &\mbox{if}& x\in \R,
	\end{array}
	\right.
	\end{eqnarray}
	where the flux $F$ is given by, $F(x,u)=H(x)f(u)+(1-H(x))g(u)$, $H$ is the Heaviside function. 
	Through out this article we assume the fluxes $f,g$ to be $C^1(\mathbb{R})$, 
	strictly convex with  superlinear growth (i.e., $\lim\limits_{|p|\rr \f} \left(\frac{f(p)}{|p|}, \frac{g(p)}{|p|}  \right) =(\f,\f)$) and initial data  $u_0\in L^\infty(\mathbb{R})$.
	We denote by $\theta_f,\theta_g$  the unique minima of the fluxes $f,g$ respectively. In this article, by entropy solution we mean a weak solution to \eqref{conlaw-equation} satisfying interface entropy condition as in \cite{Kyoto}.%Here the solution to (\ref{conlaw-equation}) we mean the entropy solution constructed through the Hamilton Jacobi equation done in \cite{Kyoto}. 
	
	Here we explore three aspects of control theory in conservation laws with discontinuous flux: (i) characterization of reachable set, (ii) exact controlability and (iii) optimal controlability. Above three problems are classical and they are answered for $f=g$ case in \cite{Sco, Sop}. It is an open question for $f\neq g$ case. One may think of clubbing states obtained from two boundary control problems which are separately known from \cite{Sco, Sop}. Unfortunately, this does not work since the equation (\ref{conlaw-equation}) is completely different from solving two different boundary value problems. Furthermore, adopting the method of backward construction \cite{Sco, Sop} to characterize  the reachable set is a big challenge due to the following facts:%. This is because:
	\begin{itemize}
		\item [(1)] Unlike the scalar conservation laws, for \eqref{conlaw-equation}, $L^1$--contraction is still unknown in general setting even if $f,g$ are convex. %There is no general theory as in Kruzkov to obtain an entropy solution for (\ref{conlaw-equation}).
		\item [(2)] Entropy solutions do not admit rarefaction  waves from the interface $\{x=0\}$ (see subsection \ref{NR}).
		\item [(3)] Reflected characteristic curves (see definition \ref{reflected}) from the boundary can occur in the structure of entropy solution to \eqref{conlaw-equation}.
	\end{itemize}
	%One of the main results in this paper  is to 
	We resolve the above difficulties by introducing a new backward construction to characterize the reachable sets. Then we  adopt this to obtain the optimal control result. 
	
	%{\color{red}Note that for discontinuous   The current article uses the explicit formulas for the entropy solutions to (\ref{conlaw-equation}) and finer properties of the characteristics curves.  Note that a generalized notion of entropy solution, called $(A,B)$ entropy solution has been obtained \cite{Jde} with similar explicit formulas as in \cite{Kyoto}. Hence a similar theorems of the present paper can be extended for the $(A,B)$ entropy solution. Basically the characterization of the reachable set is described in section \ref{connection} and the rest of the proof follows as in the case of \cite{Kyoto}.}

	The scalar conservation laws with discontinuous flux of type (\ref{conlaw-equation}) has a huge variety of applications in several fields, namely
	traffic flow modeling,  modeling gravity, modeling continuous sedimentation
	in  clarifier-thickener units, ion etching in  the semiconductor industry and many more.
	In the past two decades the first order model of type (\ref{conlaw-equation}) has been extensively studied from both the theoretical 
	and numerical point of view. Concerning the uniqueness it is worth to mention that the following Kru\v{z}kov
	type entropy inequalities, in both  the two upper quarter-planes are not sufficient to guarantee the uniqueness, 
	
	\begin{equation}\label{Kruzkov2}
	\begin{array}{lll}
	\int\limits_0^\infty \int\limits_0^\infty \left(\phi_1(u)\frac{\partial s}{\partial t}+\psi_1(u)\frac{\partial s}{\partial x}\right)&\geq 
	-\int\limits_0^\infty \psi_1(u(0+,t))s(0,t)dt,\\
	\int\limits_{-\infty}^0 \int\limits_0^\infty \left(\phi_2(u)\frac{\partial s}{\partial t}+\psi_2(u)\frac{\partial s}{\partial x}\right)&\geq 
	\int\limits_0^\infty \psi_2(u(0+,t))s(0,t)dt.
	\end{array}
	\end{equation}
	Here $(\phi_1,\psi_1)$ denote the entropy pair corresponding to the flux $f$, $(\phi_2,\psi_2)$ denote the entropy pair corresponding to the flux $g$,  and $s\in C_0^1(\mathbb{R}\times \mathbb{R}_+)$, a non-negative test function. 
	Consequently
	one need an extra criteria on the interface called ``interface entropy condition" (see \cite{Kyoto}) given by 
	\begin{equation}\label{entropy-condition}
	\mbox{meas}\{t: f^\prime(u(0+,t))>0, g^\prime(u(0-,t))<0\}=0.
	\end{equation}
	Using this extra entropy along with the above Kru\v{z}kov type inequalities the 
	uniqueness result has been obtained in \cite{Kyoto}.
	On the other hand, the existence result has been proved in several ways, namely via Hamilton-Jacobi, 
	convergence of numerical schemes, vanishing viscosity method, for further details we refer the reader to  \cite{Kyoto, Siam, Jhde, Jde,  And1, Burger,
		BurKarRisTow, Diehl5, Gimseresebro, Towers} and the references therein. The present article 
	uses the explicit formula obtained in  \cite{Kyoto}, via  the  Hamilton-Jacobi Cauchy problem.
	%\begin{equation}\label{Hamilton-Jacobi-equation}
	% \left\{  \begin{split}
	%v_t+g(v_x)&=0 \ \mbox{if}\ x<0,\ t>0,\\
	%v_t+f(v_x)&=0 \ \mbox{if}\ x>0,\ t>0,\\ 
	%v(x,0)&=v_0(x),  \ \mbox{if}\ x\in \mathbb{R}.\\
	%\end{split}\right.
	%\end{equation}
	By using this formula it can be shown that if the initial data $v_0$ is uniformly Lipschitz then the viscosity solution $v(\cdot, t)$ is also uniformly Lipschitz, for all $t>0$. Let $u:=\frac{\partial v}{\partial x}$, then $u$ is the unique weak solution (see \cite{Kyoto}, Theorem 2.2) of
	(\ref{conlaw-equation}),  enjoys (\ref{entropy-condition}) near interface and 
	satisfies the following Rankine-Hugoniot condition on the interface.
	\begin{equation}\label{RH-condition}	
	\textrm{meas}\big\{t:f(u(0+,t))\neq g(u(0-,t))\big\}=0.
	\end{equation} 
	Note that in general TV of entropy solution to \eqref{conlaw-equation} can blow up \cite{Cpam,Ghoshal} at finite time even for BV initial data which makes the current article more technical while obtaining the compactness. Regarding  the well-posedness theory to $f=g$ case, we refer the reader to \cite{Da1} for Cauchy problem 
	and for the initial boundary value problem to \cite{Jos}.
	
	%Through out this paper we work with the solution which is obtained from the Hamilton-Jacobi formulation.
	
	Concerning the exact controllability for the scalar convex conservation laws the first work has been done in \cite{An1}, where they considered
	the initial boundary value problem in a quarter plane with $u_0=0$ and by using  one boundary control they investigated  the reachable set. 
	As in \cite{Sco}, they considered  $u_0\in L^\infty$ and three possible cases, namely pure initial value problem with 
	initial data control outside any domain, initial boundary value problem in a quarter plane with one boundary control and initial boundary problem in a strip with
	two boundary controls to get the reachable sets in a complete generality. In both the articles the Lax-Oleinik type formulas has been exploited.
	An alternative approach has been provided in \cite{Hor} by using the return method (see \cite{Coron,Coron-book}). For the viscous Burgers equation any non-zero
	state can be reached in finite time by two boundary controls \cite{Glass}, recently, it has been proved \cite{And3} that there exist many pairs $(C,T)$ so that the state $C$ is not reachable from zero state at time $T$ for the viscosity 1.
	Control theory for the system of conservation
	laws is still largely open. We refer to \cite{SamJEE,An2,Bre1,CoronShyam,Glass-3D,Glass-Euler} and references therein for controllability results on system of hyperbolic conservation laws.% nevertheless
%	in \cite{Bre1}, the authors  constructed an example showing that exact controllability to a constant cannot be reached  in a finite time and proved
%	asymptotic stabilization to a constant  by two boundary controls. Recently, under dissipative boundary conditions the 
%	asymptotic stabilization to $0$ has been proved in \cite{CoronShyam} for $2\times 2$ system, when the velocities are positive.
%	For the Temple class systems and triangular type systems we refer the reader 
%	to \cite{An2} and \cite{SamJEE} respectively. 
	
	Let us briefly discuss the optimal controllability results for the case $f=g$.  Assume  the target function
	$k\in L^2_{loc}(\R)$,  support of $f'(k)$ is compact and $T>0$. We denote by  $J_{\{f=g\}}$, a cost functional, defined in the following way 
	\begin{equation}\begin{array}{ll}
	J_{\{f=g\}}(u_0)=\int\limits_{-\infty}^{\infty}|f^\prime(u(x,T))- f^\prime(k(x))|^2dx ,
	\end{array}                         
	\label{optimalcontrolf=g}\end{equation}
	where $u_0\in L^\infty(\mathbb{R})$, $u_0\equiv\theta_f$ outside a compact set, $\theta_f$  being the only critical point of the flux $f$.
	Here $u(\cdot,T)$ denotes the unique  weak  solution  at $t=T$ to the Cauchy problem 
	(\ref{conlaw-equation}), in the case  $f=g$ with  initial datum $u_0$. Then in this case, 
	the optimal control reads like: find a $w_0$ such that $J_{\{f=g\}}(w_0)=\min\limits_{u_0}J_{\{f=g\}}(u_0).$
	In \cite{Cas, Cas2}, they considered the above optimal control problem for the  Burgers' equation and proved such  minimizer exists and proposed a 
	numerical scheme called ``alternating decent algorithm", although the convergence of these scheme  still remains open.
	Whereas in \cite{Sop}, they made use of the Lax-Oleinik formula and derived a numerical backward construction which 
	converges to a solution  of the above problem. The latter method can be applied also to  general convex fluxes as long as a
	Lax-Oleinik type formula is available. It has to be noticed that even for the case $f=g$, 
	due to the occurrence of the shocks in the solution of (\ref{conlaw-equation}), one may have several minimizers of the optimal control problem
	(\ref{optimalcontrolf=g}).
	
	One if the main results in this paper is to characterize the reachable set (see subsection \ref{sec:reachable}) and then we prove the exact and optimal controllability (subsections \ref{ec} and \ref{osection} respectively) for (\ref{conlaw-equation}). In order to do so, we divide $\R\times (0,T)$ into three sub domains: $D_1$, $D_2$ and $D_3$ (see subsection \ref{section4.1}). These three domains correspond to the solution (a) with reflected characteristics, (b) having interface interactions and (c) solving pure initial value problem (i.e. $f=g$ case) respectively. 
		Now we define a reachable set at $t=T$ in such a way that a given solution corresponds to an element in the reachable set. This imposes a constraint on the elements of the reachable set. Then for any element in the reachable set, using this constraint, we first construct a data in $\bar{D}_1\cap \{(x,t): t=0\}$ and the solution.
		Using the `no forward interface rarefaction' from the $t$-axis (see lemma \ref{2.2}), we construct another initial data and a solution in $D_2$. Construction of solution in $D_3$ is as in the $f=g$ case.
		Using the R-H condition, we glue all the three solutions to form a single solution which corresponds to the given element in the reachable set.
 In the construction, we use finer analysis of characteristic curves from \cite{Sco,Sop,Ssh} and explicit formula from \cite{Kyoto}. Similar construction is valid (see section \ref{connection}) for $(A,B)$-connection \cite{Jde}. 
	
% a common solution and then by obtaining the characterization of
	The paper is organized as follows. In the next three subsections we state our main results. 
	In section \ref{PRE}, we recall some known results from \cite{Kyoto}.  Section \ref{technical} deals with the non existence of forward rarefaction from the interface, backward construction for shock and continuous solution. Also the construction of $(\tau_0,\xi_0)$ which is used to define the reachable set $\mathcal{R}(T)$. Subsections \ref{newsection} and section  \ref{backward} deals with the backward construction when the reflected characteristics exist. Sections  \ref{backward} and  \ref{optimalcontrol} proves the main results of the paper. In section \ref{connection} we indicate how to extend the above results for the $(A,B)$ entropy solution. In section \ref{appendix}, we prove a stability lemma which is used to prove the main results.
	
	Through out this article we assume that $f(\T_f)\leq g(\T_g)$. The other case $f(\T_f)> g(\T_g)$ follows in a similar way.
	
	\textbf{Authors declaration:} It is to be noted most of the ideas and technical details was in the arxiv version but   that there was a gap in the proof in \cite{Arxiv} because the reflected characteristics was not being considered in the proof and hence the definition of reachable set was incomplete. In this article we fill this gap in subsection \ref{newsection} and present a modified version of  \cite{Arxiv}. 
	\subsection{Reachable Set}\label{sec:reachable}	
Let  $\bar{\bar{\T}}_g \leq \T_f \leq \bar{\T}_g$ such that  $f(\bar{\bar{\T}}_g)=f(\bar{\T}_g)=g(\T_g).$
	Then we define:\\
	\noindent	{\bf Reachable set}: Let $T>0$, $0\le R_{2}\le R_{1}$, $y:(-\f, R_{2})\cup(R_{1}, \f)\rr\R$ be a function be given. Then $(T, R_{1}, R_2, y(\cdot))$ is called an element in the reachable set if the following holds:
	\begin{description}
		\descitem{(1)}{1pg4} $y(\cdot)$ is an non-decreasing function such that
		%	\begin{itemize}
		\item[(i)] $y(x)\le0$ for all $x\in(-\f, R_{2})$.
		\item[(ii)] $y(x)\ge0$ for all $x\in(R_{1}, \f)$.
		\item[(iii)] $\sup\{|x-y(x)|: x\in(-\f, R_{2})\cup(R_{1}, \f)\}<\f$.
		%	\end{itemize} 
		\descitem{(2)}{2pg4} Let $\gamma_{0}(t)=R_{1}+f'(\bar{\theta}_{g})(t-T)$ and we denote  $(0,T(R_1))$ to be the point of intersection of lines $t$-axis and $\ga_0$, i.e.,  $\gamma_{0}(T(R_{1}))=0$. Suppose $T(R_{1})\ge 0$, let $(\tau_{0}, \xi_{0})$ be as in lemma \ref{3.6} with $\bar{\al}=\bar{\T}_g$, $T_{1}=T(R_{1})$, then 
		\begin{equation*}
		\xi_{0}\le y(R_{1}+).
		\end{equation*}  
		\descitem {(3)}{3pg4} If $R_{2}<R_{1}$, then $T(R_{1})\ge 0$.
	\end{description}
	Denote
	\begin{equation}\label{def:reachable}
	\mathcal{R}(T)=\{(T, R_{1}, R_2, y(\cdot)): \mbox{they satisfy \descref{1pg4}{(1)}, \descref{2pg4}{(2)} and \descref{3pg4}{(3)}}\}
	\end{equation}
	is called the reachable set.
	
	Then we have the following theorem.
	\begin{theorem}\label{theorem1.1}
		Let $(T, R_{1}, R_2, y(\cdot))\in \mathcal{R}(T)$ if and only if there exists a $u_0\in L^\f(\R)$ and the solution $u$ of (\ref{conlaw-equation}) such that for $i=1,2$, $R_i=R_i(T)$, $y(x)=y(x,T)$, where $R_i$ and $y(x,T)$ defined in theorem \ref{theorem3.1}.
	\end{theorem}	

	\subsection{Exact controllability}\label{ec}
	\begin{theorem}\label{theorem1.3}
		Let $(T, R_{1}, R_2, y(\cdot))\in \mathcal{R}(T)$ where $\mathcal{R}(T)$ is defined as in \eqref{def:reachable} and $C_{1}<0<R_{1}<C_{2}$, $B_{1}<0<B_{2}$ be given. Assume that 
		\begin{eqnarray}
		y(C_{1}+)&>&B_{1},\\
		y(C_{2}-)&<&B_{2}
		\end{eqnarray}
		and $u_{1, 0}\in L^{\f}(\R\setminus(B_{1}, B_{2}))$, then there exist a $\tilde{u}_{0}\in L^{\f}(B_{1}, B_{2})$ and the solution $u$ of \eqref{conlaw-equation} with initial data $u_{0}$ satisfying 
		\begin{eqnarray}
		u_{0}(x)=\left\{\begin{array}{lll}
		u_{1, 0}(x)&\mbox{if}&x\notin(B_{1}, B_{2}),\\
		\tilde{u}_{0}(x)&\mbox{if}&x\in(B_{1}, B_{2}).
		\end{array}\right.
		\end{eqnarray}
		Let $(T, R_{1}(T), R_{2}(T), y(\cdot, T))$ be the element in $\mathcal{R}(T)$ corresponds to $u(\cdot, T)$, then 
		\begin{eqnarray}
		R_{i}&=&R_{i}(T) \mbox{ for } i=1, 2,\\
		y(x)&=&y(x, T) \mbox{ for all } x\in(C_{1}, R_{2})\cup(R_{1}, C_{2}).
		\end{eqnarray}
	\end{theorem}
	
	\subsection{Optimal control}\label{osection}
	Let $u_0\in L^{\f}(\R)$ and $u$ be the solution of \eqref{conlaw-equation} with initial data $u_0$. Let $T>0$. Let $k\in L^{\f}(\R)$ and $c>0$ such that
	\begin{equation*}
	k(x)=\left\{\begin{array}{rl}
	\theta_g&\mbox{ if }x<-c,\\
	\theta_f&\mbox{ if }x>c.
	\end{array}\right.
	\end{equation*}
	Define 
	\[K(x)=\left\{\begin{array}{rl}
	f^{\p}(k(x))&\mbox{ if }x>0,\\
	g^{\p}(k(x))&\mbox{ if }x<0.
	\end{array}\right.\]
	Note that $K\in L^{\f}(\R)$ and support of $K\subset[-c,c]$. 
	Denote $g_+^{-1}$ to be the inverse of $g$ on $[g(\T_g),\f)$. Let $u_0\in L^\f(\R)$ and $u$ be the corresponding solution of (\ref{conlaw-equation}).
	Define the cost functional $J:L^{\f}(\R)\rr\R$  by
	%    
	%\noindent\textbf{Cost functional}: Assume that $f^*(0)\geq g^*(0)$ and $u_0\in A.$
	%Let $u$ be the entropy solution of (\ref{conlaw-equation}). Let $T>0$ and define the cost functional $J$ by 
	\begin{eqnarray}
	\begin{array}{lllll}\label{cost1}
	J(u_0)&=& \displaystyle \int\limits_{-\infty}^{0}|g^\prime(u(x,T)) - K(x)|^2dx +
	+ \displaystyle\int\limits_{0}^{R_2(T)}|g^\prime\circ g_+^{-1}\circ f(u(x,T))- K(x)|^2dx\\
	&+&\displaystyle\int\limits_{R_2(T)}^{\infty}|f^\prime(u(x,T)) -
	K(x)|^2dx.
	\end{array}
	\end{eqnarray}
	%where $(T, R_{2}(T), R_{1}(T))$ corresponds to $u$ at $t=T$.
	Then we have the following result on optimal control problem.

	\begin{theorem}\label{theorem1.2} 
		Let  $\mathcal{A}$ be the admissible class of functions defined by  
		\begin{eqnarray*}\begin{array}{lll}
				\mathcal{A}&=&\left\{u_{0}\in L^{\f}(\R): \exists M>0 \mbox{ such that } 
				u_{0}(x)=\left\{\begin{array}{lll}
					\theta_{g}&\mbox{if}&x<-M,\\
					\theta_{f}&\mbox{if}&x>M.
				\end{array}\right.\right\},\\
				%\mathcal{\tilde{A}}&=&\{(R_{1}, R_{2}, y)\in R(T): y(x)=x \mbox{ outside a compact set } \}.
			\end{array}
		\end{eqnarray*} 
		Then there exists a $u_0\in \mathcal{A}$ such that 
		\begin{eqnarray}
		J(u_0)=\min_{w_0\in \mathcal{A}}{J}(w_0).\label{optimalcontrol}
		\end{eqnarray}  \end{theorem}
	We prove the Theorem \ref{theorem1.2} via an explicit construction and hence can be adopted to numerical computation.
	
	%Moreover the existence of $u_0$ is constructive

	\begin{remark}
We can obtain the similar results, when one of the flux is concave and another one is convex in the equation (\ref{conlaw-equation}).  One can use the explicit formulas as in \cite{Jde} and similar analysis in the present paper.     
\end{remark}
	\section{Preliminaries}\label{PRE}
	In order to make the paper self contained we recall  some results, definitions and notations  from \cite{Kyoto}.

	\begin{definition}\label{CC} \textbf{Control curve}:
		(See figure \ref{fig-1} for illustration) We say $\gamma\in C([0,t],\mathbb{R})$ is a control curve if it verifies the following conditions:
		\begin{enumerate}
			\item $\gamma$ is piece-wise affine and it can have at most 3 affine segments such that each affine part lies completely in either $[0,\f)\times[0,\f)$ or $(-\f,0]\times[0,\f)$.
			
			\item If $\ga$ has three affine segments $\{\gamma_i;\,i=1,2,3\}$ defined as $\gamma_i=\gamma|_{[t_{i-1},t_{i}]}$ where $0=t_0\leq t_1\leq t_2\leq t_3=t$,
			then $\ga_2(s)=0$ for all $s\in(t_1,t_2)$ and for all $s\in (t_0,t_1)\cup (t_2,t)$, either $\ga_1(s), \ga_3(s)$ are in $(-\f,0)$ or in $(0,\f)$. 
			
			%then for all $s\in(t_0,t_1)\cup (t_2,t)$, either $\ga_1(s), \ga_2(s)$ are in $(-\f,0)$ or in $(0,\f)$.
		\end{enumerate}
		
		Let $0 <t,\ x\in \mathbb{R} $ and let $c(x,t)$ be the set of all control curves such that $\gamma(t)=x$. The set $c(x,t)$ can be partitioned into three categories:
		\begin{enumerate}
			\item $c_0(x,t)\subset c(x,t)$ consists of control curves $\gamma$ which have only one affine segment and satisfies $x\gamma(s)\geq0$ for $s\in[0,t]$.
			
			\item $c_r(x,t)\subset c(x,t)$ consists of control curves $\gamma$ which have exactly 3 affine segments and satisfies $x\gamma(s)\geq0$ for $s\in[0,t]$. Here we say $c_r(x,t)$ to be the set of all reflected control curves.  
			
			\item $c_b(x,t)= c(x,t)\setminus \{c_0(x,t) \cup c_r(x,t)\}$.
		\end{enumerate}
	\end{definition} 
	
%	\end{center}
	\begin{figure}
		\centering
%		%	\begin{picture}
		\begin{tikzpicture}
\draw (-2,0)--(2,0);\draw (3,0)--(7,0); \draw (8,0)--(12,0);
\draw (0,0)--(0,3);\draw (5,0)--(5,3); \draw (10,0)--(10,3);
\draw (-2,4.5)--(2,4.5);\draw (3,4.5)--(7,4.5); \draw (8,4.5)--(12,4.5);
\draw (0,4.5)--(0,7.5);\draw (5,4.5)--(5,7.5); \draw (10,4.5)--(10,7.5);
\draw (1,4.5)--(2,7.5);\draw (-1,0)--(-2,3);
\draw (4.5,4.5)--(5,6);\draw (5,6)--(7,7.5);
\draw (11,4.5)--(10,5.5);\draw[very thick](10,5.5)--(10,6.5);\draw (10,6.5)--(12,7.5);
\draw (5.5,0)--(5,1.5);\draw (5,1.5)--(3,3);
\draw (9,0)--(10,1);\draw[very thick](10,1)--(10,2);\draw (10,2)--(8,3);
%numbering
\draw (0,4)node{\small$(a)$};\draw (5,4)node{\small$(b)$};\draw (10,4)node{\small$(c)$};
\draw (0,-.5)node{\small$(d)$};\draw (5,-.5)node{\small$(e)$};\draw (10,-.5)node{\small$(f)$};
\draw (-1,-.25)node{\small$\gamma(0)$};
\draw (-1,0)node{\tiny $\bullet$};
\draw (1,4.25)node{\small$\gamma(0)$};
\draw (1,4.5)node{\tiny$\bullet$};
\draw (11.1,4.25)node{\small$\gamma(0)$};
\draw (11,4.5)node{\tiny$\bullet$};
\draw (4.4,4.25)node{\small$\gamma(0)$};
\draw (4.5,4.5)node{\tiny$\bullet$};
\draw (5.6,-.25)node{\small$\gamma(0)$};
\draw (5.5,0)node{\tiny $\bullet$};
\draw (9,-.25)node{\small$\gamma(0)$};
\draw (9,0)node{\tiny $\bullet$};
\draw (8,3.25)node{\small$(x,t)$};
\draw (8,3)node{\tiny$\bullet$};
\draw (3,3.25)node{\small$(x,t)$};
\draw (3,3)node{\tiny$\bullet$};
\draw (12,7.75)node{\small$(x,t)$};
\draw (12,7.5)node{\tiny$\bullet$};
\draw (7,7.75) node{\small$(x,t)$};
\draw (7,7.5)node{\tiny$\bullet$};
\draw (-2,3.25) node{\small$(x,t)$};
\draw (-2,3)node{\tiny$\bullet$};
\draw (2,7.75) node{\small$(x,t)$};
\draw (2,7.5)node{\tiny$\bullet$};
		\end{tikzpicture}
		\caption{Figures (a), (b) and (c) is representing control curves for the case $x>0$, figures (d), (e) and (f) is for $x<0$. Note that the figures  (c), (f) represents the reflected control curves.}\label{fig-1}
%		%\end{picture}
	\end{figure}
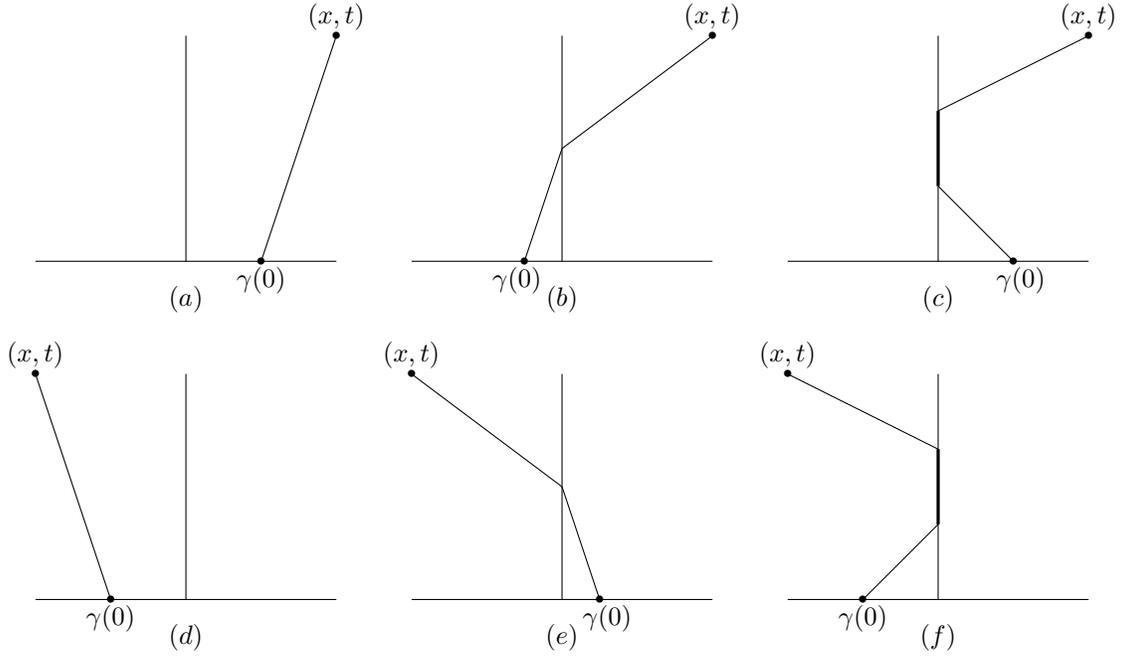
	\begin{definition}\textbf{Convex dual}: Let $f$ be a $C^1$ convex function with superlinear growth, that is $f$ satisfies $\lim\limits_{|s|\rightarrow \infty} \frac{f(s)}{|s|}=\infty.$ Then we denote the convex dual of $f$ by $f^*$ and defined by 
		$f^*(p)=\sup_{q}\{pq-f(q)\}.$ Observe that $(f^*)'=(f')^{-1}$. 
		
	\end{definition}
	
	\begin{definition}\label{CF}\textbf{Cost function}: Let  $f^*, g^*$ be the respective convex duals of the fluxes $f$ and $g$. 
		Let us assume that $v_0: \mathbb{R}\rightarrow \mathbb{R}$ be an  uniformly Lipschitz continuous function. 
		Let $(x,t)\in\mathbb{R}\times \mathbb{R}_+$,  $\gamma\in c(x,t)$.
		The cost functional $\Gamma$ associated to $v_0$ is defined by 
		\begin{equation*}
		\begin{array}{llll}
		\Gamma_{v_0,\gamma}(x,t)&=v_0(\gamma(0))+\int\limits_{\{\theta\in[0,t]\ :\ \gamma(\theta)>0 \}} f^*(\dot{\gamma}) d\theta
		+\int\limits_{\{\theta\in[0,t]\ :\ \gamma(\theta)<0 \}} g^*(\dot{\gamma}) d\theta \\
		&+ \mbox{meas}\{\theta\in[0,t]\
		:\ \gamma(\theta)=0 \}\mbox{min}\{f^*(0),g^*(0)\}.
		\end{array}
		\end{equation*}
		Then we define the value function   $v:\mathbb{R}\times \mathbb{R}_+ \rightarrow \mathbb{R}$ by 
		%\begin{equation}
		$$v(x,t)=\inf\limits_{\gamma\in c(x,t)}\{ \Gamma_{v_0, \gamma}(x,t) \}.$$
		%\end{equation}
	\end{definition}
	\begin{definition}\label{reflected} Let us define by 
		$ch(x,t)=\{\gamma \ : \ \Gamma_{v_0,\gamma}(x,t)=v(x,t)\},$ the set of \textbf{characteristics curves}. We say an element of the set  $ch(x,t)\cap c_r(x,t)$ to  be a \textbf{reflected characteristics curves}.
	\end{definition}	
	Let $t>0$, define (see figure \ref{fig-2} for further illustration)
	\begin{eqnarray*}
		R_{1}(t)&=&\inf\{x\ge 0: ch(x,t)\subset c_{0}(x,t)\}.\\
		R_{2}(t)&=&\begin{cases*}
			\inf\{x:0\le x\le R_{1}(t),ch(x,t)\cap c_{r}(x,t)\not=\emptyset\},\\
			R_{1}(t)  \mbox{ if the above set is empty}.
		\end{cases*}\\
		L_{1}(t)&=&\sup\{x\le 0: ch(x,t)\subset c_{0}(x,t)\}.\\
		L_{2}(t)&=&\begin{cases*}
			\sup\{x:L_{1}(t)\le x\le 0,ch(x,t)\cap c_{r}(x,t)\not=\emptyset\},\\
			L_{1}(t)  \mbox{ if the above set is empty}.\\
		\end{cases*}\\
		y(x,t)&=&\inf\{\gamma(0):\gamma\in ch(x,t), x\in (-\f,L_{1}(t))\cup(R_{1}(t),\f)\}.
	\end{eqnarray*}
	Let $0\le x\le R_{1}(t)$, define 
	\begin{eqnarray*}
		t_{+}(x,t)&=&\inf\{t_{1}:\gamma (t_{1})=0,\gamma(\theta)>0, \forall \ \theta\in(t_{1}, t), \gamma\in ch(x,t) \}.\\
		t_{+}(R_{i}(t)-,t)&=&\lim\limits_{x\uparrow R_{i}(t)}t_{+}(x,t),\quad i=1,2.
	\end{eqnarray*}
	For $t_{+}(R_{2}(t)-,t)\le s\le t$, define
	\begin{equation*}
	y_{-,0}(s)=\inf\{\gamma(0): \gamma\in ch(0,s)\}.
	\end{equation*}
	Let $L_{1}(t)\le x\le 0$, define
	\begin{eqnarray*}
		t_{-}(x,t)&=&\inf\{t_{1}:\gamma (t_{1})=0,\gamma(\theta)<0, \forall\ \theta\in(t_{1},t), \gamma\in ch(x,t) \}.\\
		t_{-}(L_{i}(t)+,t)&=&\lim\limits_{x\downarrow L_{i}(t)}t_{-}(x,t),\ i=1,2.
	\end{eqnarray*}
	For $t_{-}(L_{2}(t)+,t)\le s\le t$, define
	\begin{equation*}
	y_{+,0}(s)=\inf\{\gamma(0): \gamma\in ch(0,s)\}.
	\end{equation*}
	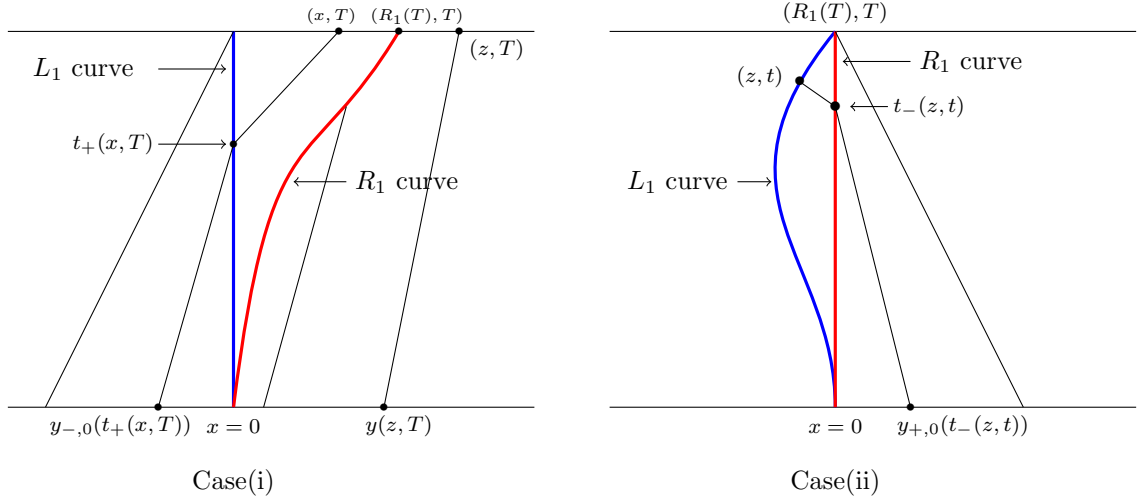
\begin{figure}
		\centering
		\begin{tikzpicture}
		\draw (-3,0)--(4,0);\draw (0,0)--(0,5);\draw (-3,5)--(4,5);
		\draw (5,0)--(12,0);\draw (8,0)--(8,5);\draw (5,5)--(12,5);
		\draw [very thick,blue](0,5)--(0,0);
		\draw [very thick,red] (0,0) .. controls (.5,4) and (1,3) .. (2.2,5);
		\draw (-1,0)--(0,3.5);\draw(0,3.5)--(1.4,5);
		\draw (-2.5,0)--(0,5);\draw (.4,0)--(1.5,4);\draw (2,0)--(3,5);
		\draw [very thick,blue] (8,0) .. controls (8,2) and (6.2,2.8) .. (8,5);
		\draw [very thick,red](8,0)--(8,5);
		\draw (9,0)--(8,4);
		\draw (8,4)--(7.5,4.35);\draw (10.5,0)--(8,5);
		%numbering
		\draw (8,5.25)node{\scriptsize$(R_{1}(T),T)$};\draw (9.7,-.25)node{\scriptsize $y_{+,0}(t_{-}(z,t))$};
		\draw (9,0)node{\tiny$\bullet$};%\draw (10,3)node{\small$t_{-}(z,t)$};%\draw [->](9.3,3)--(8.1,4);
		\draw [<-](8.2,4)--(8.7,4);
		\draw (9.25,4)node{\scriptsize $t_{-}(z,t)$};
		\draw (8,4)node{\footnotesize  $\bullet$};
		\draw (0,-.25)node{\scriptsize $x=0$};
		\draw (8,-.25)node{\scriptsize$x=0$};
		\draw [<-](.8,3)--(1.5,3);
		\draw (2.3,3)node{\small$R_{1}$ curve};
		\draw [<-](-.1,4.5)--(-1,4.5);
		\draw (-2,4.5)node{\small$L_{1}$ curve};
		\draw [<-](8.1,4.6)--(9,4.6);
		\draw (9.8,4.6)node{\small$R_{1}$ curve};
		\draw [<-](7.15,3)--(6.7,3);
		\draw (5.9,3)node{\small$L_{1}$ curve};
		\draw (7,4.4)node{\scriptsize $(z,t)$};
		\draw (7.53,4.34)node{\scriptsize $\bullet$};
		\draw (3,5)node{\tiny$\bullet$};
		\draw (2.2,5)node{\tiny$\bullet$};
		\draw (1.4,5)node{\tiny$\bullet$};\draw (1.3,5.2)node{\tiny$(x,T)$};\draw (2.44,5.2)node{\tiny$(R_{1}(T),T)$};\draw (3.5,4.75)node{\scriptsize $(z,T)$};
		\draw (2.2,-.25)node{\scriptsize$y(z,T)$};
		\draw (2,0)node{\tiny$\bullet$};
		\draw (-1,0)node{\tiny$\bullet$};
		\draw (-1.5,-.25)node{\scriptsize$y_{-,0}(t_{+}(x,T))$};
		\draw (0,3.5)node{\tiny$\bullet$};
		\draw (-1.6,3.5)node{\scriptsize$t_{+}(x,T)$};
		\draw [<-](-.1,3.5)--(-1,3.5);
		\draw (0,-1) node{\small Case(i)};\draw (8,-1) node{\small Case(ii)};
		\end{tikzpicture}
			\caption{Illustrations of $R_1(t), L_1(t), t_\pm, y_{\pm,0}, y$. In case (i), $L_1(t)=0$ for all $t\in (0,T)$ and for the  case (ii), $R_1(t)=0$ for all $t\in (0,T)$.}\label{fig-2}
		%	\captionof{figure}{Case (i) : when $R_1(T)>0, L_1(T)=0$, Case (iii) : when $R_1(T)=0, L_1(T)=0.$}
	\end{figure}

	\begin{definition}
		Let $(X,d)$ be a metric space and $A_{k},A$ are subsets of $X$ for each $k\ge 1$. We say that 
		$	\lim\limits_{k\rr\f}A_{k}\subset A$
		if for every sequence $\{ x_{k}\}$ with $x_{k}\in A_{k}$, there exists a subsequence $\{x_{k_{i}}\}$ converges to some $x\in A$. \end{definition}
	
	Let $f,g$ be in $C^1(\R)$ and strictly convex function with $F(x,u)=H(x)f(u)+(1-H(x))g(u)$. Let $\{f_k\}$, $\{g_k\}$in $C^1(\R)$ be sequences of strictly convex functions and $F_k(x,u)=H(x)f_k(u) + (1-H(x))g_k(u)$ such that 
	\begin{eqnarray*}
		&\lim\limits_{k\rr \f} (f_k(u), g_k(u))=(f(u),g(u)) \ \mbox{in}\ C^1_{\mbox{loc}}(\R),\\
		& \lim\limits_{|p|\rr \f} \left(\inf\limits_{k} \frac{f_k(p)}{|p|}, \inf\limits_{k} \frac{g_k(p)}{|p|}  \right) =(\f,\f).
	\end{eqnarray*}
	Let $u_0\in L^\f(\R)$ and $\{u_{0,k}\}\subset L^\f(\R)$ be such that 
	$$\lim\limits_{k\rr \f}u_{0,k}=u_0,\ \mbox{in } L^\f(\R)-\mbox{weak}*.$$
	Let $v_0(x)=\int\limits_0^x u_0(\T)d\T$ and $v_{0,k}(x)=\int\limits_0^x u_{0,k}(\T)d\T$
	be the associated primitives of $u_0$ and $u_{0,k}$. Notice that in \cite{Kyoto}, it was assumed that $u_0\in C(\R)$. Now it is easy to show that results in \cite{Kyoto} continue to hold for $u_0\in L^\f(\R)$. In order to prove this result, we need the following stability lemma.  %Then we have the following stability result. 
	\begin{lemma}[Stability lemma]\label{stability}
		With the data as above, let $v, v_k$ be the corresponding value functions associated to the fluxes $F$ and $F_k$ and initial data $v_0$ and $v_{0,k}$ respectively. Let $ch(x,t)$ and $ch_k(x,t)$ be the respective characteristic sets. Let $u=\frac{\partial v}{\partial x}$ and $u_k=\frac{\partial v_k}{\partial x}$, then 
		\begin{enumerate}[label=\arabic*.]
			\item\label{lemma-stab-1} $\lim\limits_{k\rr \f} v_k=v \ \mbox{in}\ C^1_{\mbox{loc}}(\R\times [0,\f))$,
			\item\label{lemma-stab-2} $\lim\limits_{k\rr \f} u_k=u \ \mbox{in}\ \mathcal{D}'(\R\times [0,\f))$,
			\item\label{lemma-stab-3} $\lim\limits_{k\rr \f} ch_k(x,t)\subset ch(x,t)$.
		\end{enumerate}
	\end{lemma}
	Proof of this lemma is given in the appendix (section \ref{appendix}).
	%	
	%	\begin{lemma}
	%		Let $u_{0}\in L^{\f}(\R), \{u_{0,k}\}\subset L^{\f}(\R)\cap C(\R)$ such that 
	%	\begin{align*}
	%	\sup_{k}||u_{0,k}||_{L^{\f}(\R)}<\f, \	\lim u_{0,k}=u_{0} \ \ \mbox{ in weak* $L^{\f}(\R)$}.
	%		\end{align*}
	%		Let $v$ and $v_{k}$ be the value function corresponding to $u_{0}$ and $u_{0,k}$. Let $ch(x,t)$ and $ch_{k}(x,t)$ be the characteristic sets corresponding to $v$ and $v_{k}$. Then
	%		\begin{eqnarray}
	%		\left\{\begin{array}{llll}
	%		v_{k}\rr v$\ \text{in} \ $C^{0}_{loc}(\R\times[0,\f)),\\
	%		(v_{k})_{x}\rr(v_{k})_{x}$\ \text{in}\ $D^{'}(\R\times(0,\f)),\\
	%		\lim\limits_{k\rr\f}ch_{k}(x,t)\subset ch(x,t).
	%		\end{array}\right.
	%		\end{eqnarray}
	%	\end{lemma}
	%		********************************

\begin{theorem}[\cite{Kyoto}]\label{theorem3.1}
	Let $u_0\in L^{\f}(\R)$ and $v$ be the corresponding value function defined in definition \ref{CF}. Then $u(x,t)=\frac{\partial v}{\partial x}(x,t)$ exists for $t>0$, a.e. $x\in R$ and is a solution to \eqref{conlaw-equation}. Furthermore there exist Lipschitz continuous curves $R_i(t), L_i(t)$, for $i=1,2$ such that for each $t>0$ a.e. $x\in\R$, we have
	\begin{description}[font=\normalfont]
	\descitem{1.}{t1item1} No two characteristics intersects in the region $\{(x,t):x\not=0, t>0\}$, i.e., if $\gamma_{i}\in ch(x_{i},t_{i}), i=1,2$, then if for some $\theta\in(0,min(t_{1},t_{2}))$, $\gamma_{1}(\theta)\not=0, \gamma_{2}(\theta)\not=0$, then $\gamma_{1}(\theta)\not=\gamma_{2}(\theta)$, provided $\gamma_{1}$ and $\gamma_{2}$ are two different characteristic curves.
			 	\descitem{2.}{t1item6} Let $T>0$, then one of the following holds:
		 	\subitem(i) If $R_{1}(T)>0$, then $L_{1}(T)=0$ and for all $t\in(t_{+}(R_{1}(T)-,T),T), R_{1}(t)>0.$
		 	\subitem(ii) If $L_{1}(T)<0$, then $R_{1}(T)=0$ and for all $t\in(t_{-}(L_{1}(T)+,T),T), L_{1}(t)<0.$
		 	\subitem(iii) $R_{1}(T)=L_{1}(T)=0$.
		 	\descitem{3.}{t1item3} The following properties are true:
		 	
		 	\begin{description}[font=\normalfont]
	\descitem{(i)}{i.} %{\color{red}$f^{*}(0)\le g^{*}(0)$, then $R_{1}(t)=0$ and} 
	If $f^{*}(0)\ge g^{*}(0)$ (equivalently $f(\T_f)\leq g(\T_g)$), then $L_{1}(t)=0$ and 
		if $f^{*}(0)\le g^{*}(0)$ (equivalently $f(\T_f)\geq g(\T_g)$), then $R_{1}(t)=0$.
		\descitem{(ii)}{ii.} $x\mapsto y(x,t)$ is a non decreasing function and $x\mapsto t_{+}(x,t)$ is a non increasing function on the domain of definitions.
		\descitem{(iii)}{iii.} For $R_{2}(t)<x<R_{1}(t)$, $\frac{x}{t-t_{+}(x,t)}>0$ and for a.e. $x$,
		\begin{equation*}
		g(\theta_{g})=f\circ f^{*'}\left(\frac{x}{t-t_{+}(x,t)}\right).
		\end{equation*}
		\descitem{(iv)}{iv.} $s\mapsto y_{-,0}(s)$ is non increasing function.
		\descitem{(v)}{v.} $x\mapsto t_{-}(x,t)$ is non decreasing function.
		\descitem{(vi)}{vi.} For $L_{1}(t)<x<L_{2}(t)$, $\frac{x}{t-t_{-}(x,t)}<0$ and for a.e. $x$,
		\begin{equation*}
		f(\theta_{f})=g\circ g^{*'}\left(\frac{x}{t-t_{-}(x,t)}\right).
		\end{equation*}
		\descitem{(vii)}{vii.} $s\mapsto y_{+,0}(s)$ is non decreasing function.
		\descitem{(viii)}{viii.} $u(0+,t)$, $u(0-,t)$ exist and RH  condition holds, i.e., $f(u(0+,t))=g(u(0-,t))$ for a.e. $t>0$,\\
		Interface entropy condition: Let $\mathcal{L}^{1}$- denotes the one dimensional Lebesgue measure, then 
		\begin{equation}
	\mathcal{L}^{1}\{t:f'(u(0+,t))>0,  g'(u(0-,t))<0\}=0.
		\end{equation} 
		\descitem{(ix)}{ix.} The entropy solution $u$ is explicitly given by the following Lax-Oleinik type formula, for $t>0$, a.e., $x\in\R$,
		\begin{eqnarray}
		u(x,t)=\left\{\begin{array}{llll}
		f^{*'}\bigg(\frac{x-y(x,t)}{t}\bigg) &\mbox{ if }& x>R_{1}(t),\\
		f^{*'}\bigg(\frac{x}{t-t_{+}(x,t)}\bigg) &\mbox{ if }& 0<x<R_{1}(t),\\
		g^{*'}\bigg(\frac{x-y(x,t)}{t}\bigg) &\mbox{ if }& x<L_{1}(t),\\
		g^{*'}\bigg(\frac{x}{t-t_{-}(x,t)}\bigg) &\mbox{ if }& L_{1}(t)<x<0.
		\end{array}\right.			 	\end{eqnarray}
		\descitem{(x)}{x.} For a.e., $x\in(0, R_{2}(T))$,
		\begin{equation*}
	 \frac{x}{t-t_{+}(x,t)}=f'(u(x,t))=f'(u(0+,t_{+}(x,t))), g'(u(0-,t_{+}(x,t)))=-\frac{y_{-,0}(t_{+}(x,t))}{t_{+}(x,t)}.
		\end{equation*}
		\descitem{(xi)}{xi.} For a.e., $x\in(L_{2}(T),0)$,
		\begin{equation*}
		\frac{x}{t-t_{-}(x,t)}=g'(u(x,t))=g'(u(0-,t_{-}(x,t))), f'(u(0+,t_{-}(x,t)))=-\frac{y_{+,0}(t_{-}(x,t))}{t_{-}(x,t)}.
		\end{equation*}
		\descitem{(xii)}{xii.} $L^{1}$- Contractivity: Let $u_{0}$, $v_{0}\in L^{\f}(\R)$ and $u, v$ be the solution of \eqref{conlaw-equation} with corresponding initial data $u_{0}$, $v_{0}$ respectively. Assume that the set of discontinuity of $u$ and $v$ are discrete set of Lipschitz curves. Then
		\begin{equation*}
		\int\limits_{a}^{b}|u(x,t)-v(x,t)|dx\le \int\limits_{a-Mt}^{b+Mt}|u_{0}(x)-v_{0}(x)|dx,
		\end{equation*}
		  where with $M_{1}=\max(||u_{0}||_{\f}, ||v_{0}||_{\f})$
		  \begin{equation*}
		  M=\max\left\{ \frac{|f(a)-f(b)|}{|a-b|}, : a\not=b, a, b\in(-M_{1}, M_{1})\right\}.
		  \end{equation*}
		  	\end{description}
	\end{description}
\end{theorem}
%For the convenience of the reader, we  give the page numbers and line numbers where \descref{i.}{(i)} to \descref{xii.}{(xii)} are found in \cite{Kyoto}. 
We remark that there is a change in the notation used here and in \cite{Kyoto} and is as follows: 

See equations (4.13), (4.10), page 51 in \cite{Kyoto}:
\begin{eqnarray}
y(x,t):=\left\{\begin{array}{llll}
y_{+}(x,t)&\mbox{ if } &x>R_{1}(t),\\
y_{-}(x,t)&\mbox{ if }& x<L_{1}(t).
\end{array}\right.
\end{eqnarray}
\begin{eqnarray}
t_{+}(x,t)&:=&y_{+}(x,t) \mbox{ if } 0<x<R_{1}(t),\\
t_{-}(x,t)&:=&y_{-}(x,t)\mbox{ if } L_{1}(t)<x<0.
\end{eqnarray}
See equation (4.25), page 54 in \cite{Kyoto}
\begin{eqnarray}
y_{-,0}\left(t-\frac{x}{q_{1}}\right)&:=&y_{-}\left(0, t-\frac{x}{q_{1}}\right)\mbox{ if } 0<x<R_{2}(t),\\
y_{+,0}\left(t-\frac{x}{q_{1}}\right)&:=&y_{+}\left(0, t-\frac{x}{q_{1}}\right)\mbox{ if } L_{2}(t)<x<0.
\end{eqnarray}
Now comes to the identification of \descref{i.}{(i)} to \descref{xii.}{(xii)} in \cite{Kyoto}.  \descref{i.}{(i)} follows from (i) of \cite[Lemma 4.9, page 51]{Kyoto}, \descref{ii.}{(ii)}, \descref{iv.}{(iv)}, \descref{v.}{(v)}, \descref{vii.}{(vii)} and \descref{viii.}{(viii)} to follows from the non intersecting proved in \cite[Lemma 4.8, 4.9, page 50 and page 51]{Kyoto}. \descref{iii.}{(iii)} follows from (4.20) to (4.25) of page 53, (4.26) and last 4 lines of page 54 and first 3 lines of page 55 in \cite{Kyoto}. \descref{viii.}{(viii)} follows from \cite[Lemma 4.10 in page 55]{Kyoto}. \descref{ix.}{(ix)} follows from \cite[Theorem 3.2]{Kyoto}. \descref{x.}{(x)} and \descref{xii.}{(xii)} follows from (4.10) page 55 and (ix), (xii) in \cite[Theorem 2.2, page 30]{Kyoto}.

\section{Some technical lemmas}\label{technical}
	First observe that $\eta\in ch(x,t)$, then $\eta$ is a curve consists of atmost three line segments and denote $\dot{\eta}(\theta)=(p_{1},p_{2},p_{3})$, where $p_{i}$ is the slope of $i$th line segment. By abuse of notations we denote $p_{i}=\emptyset$ if the $i_{th}$ line segment does not exist. Note that if $0<x<R_{2}(t)$ or $L_{2}(t)<x<0$, then for any $\eta\in ch(x,t)$, $\dot{\eta}(\theta)=(p_{1},\emptyset,p_{3})$, $p_{1}>0,p_{3}>0$ if $x>0$ and $p_{1}<0, p_{3}<0$ if $x<0$.
	\begin{definition}
		For $0<x<R_{2}(t)$ or $L_{2}(t)<x<0$, define 
		\begin{eqnarray*}
			ch_{+}(x,t)=\{
				p_{1}: \exists \ \eta\in ch(x,t) \mbox{ such that } \dot{\eta}(\theta)=(p_{1},\emptyset, p_{3}) \mbox{ for some } p_{3}	\}.\\
				ch_{-}(x,t)=\{
				p_{3}: \exists \  \eta\in ch(x,t) \mbox{ such that } \dot{\eta}(\theta)=(p_{1},\emptyset, p_{3}) \mbox{ for some } p_{1}	\}.
			\end{eqnarray*}
	\end{definition}
\begin{lemma}\label{lemma2.1}
	Let $0<x_{1}<x_{2}<R_{2}(t)$ or $L_{2}(t)<x_{1}<x_{2}<0$, $p_{1}\in ch_{+}(x_1,t), q_{1}\in ch_{+}(x_{2},t), p_{3}\in ch_{-}(x_1,t), q_{3}\in ch_{-}(x_{2},t)$, then
	\begin{eqnarray*}
	\begin{cases*}
		\frac{x_{1}}{p_{1}}\le \frac{x_{2}}{q_{1}} \mbox{ if } x_{1}>0,\\
		\frac{x_{2}}{q_{1}}\le \frac{x_{1}}{p_{1}} \mbox{ if } x_{2}<0,
	\end{cases*}
\end{eqnarray*}
\begin{equation*}
-p_{3}\bigg(t-\frac{x_{1}}{p_{1}}\bigg)\le -q_{3}\bigg(t-\frac{x_{2}}{p_{2}}\bigg).
\end{equation*}
\end{lemma}
\begin{proof}
	Let $0<x_{1}<x_{2}<R_{2}(t)$ and 
	\begin{eqnarray*}
	\gamma_{1}(\theta)=\left\{\begin{array}{llll}
		x_{1}+p_{1}(\theta-t) &\mbox{ if }& t-\frac{x_{1}}{p_{1}}\le \theta<t,\\
		p_{3}\bigg(\theta-t+\frac{x_{1}}{p_{1}}\bigg) &\mbox{ if }& 0\le \T\le t-\frac{x_{1}}{p_{1}}.
\end{array}\right. \end{eqnarray*}

	\begin{eqnarray*}
	\gamma_{2}(\theta)=\left\{\begin{array}{llll}
		x_{2}+q_{1}(\theta-t) &\mbox{ if }& t-\frac{x_{2}}{p_{2}}\le \theta<t,\\
		q_{3}\bigg(\theta-t+\frac{x_{2}}{q_{1}}\bigg) &\mbox{ if }& 0\le 
		\T\le t-\frac{x_{2}}{p_{2}}.
\end{array}\right. \end{eqnarray*}
Then by dynamic programming principle, $\gamma_{1}(\theta)\in ch(x_{1},t)$, $\gamma_{2}(\theta)\in ch(x_{2},t)$. Hence from \descref{t1item1}{(1)} of theorem \ref{theorem3.1} we have $\gamma_{1}$ and $\gamma_{2}$ do not intersect in $x\not=0$. Hence if $\theta_{1}, \theta_{2}$ be such that $\gamma_{i}(\theta_{i})=0,$ then for $x_{1}>0$, $\theta_{2}\le\theta_{1}$ and $\gamma_1(0)\le\gamma_{2}(0).$ That is 
	\begin{eqnarray*}
	\begin{cases*}
		0\le t-\frac{x_{2}}{q_{1}}\le t-\frac{x_{1}}{p_{1}}, \\
		-p_{3}\bigg(t-\frac{x_{1}}{p_{1}}\bigg)\le -q_{3}\bigg(t-\frac{x_{2}}{q_{1}}\bigg).
	\end{cases*}
\end{eqnarray*}
If $x_{2}<0$, then $\theta_{1}\le\theta_{2}$ and $\gamma_{1}(0)\le\gamma_{2}(0)$. This prove the lemma.
\end{proof}
\subsection{No rarefaction from  the interface}\label{NR}
%In order to show that the backward construction to work, first we need to show that there are no forward rarefactions from the interface, 
One of the key factor of this article is that there exists no rarefaction from the interface for the solution of \eqref{conlaw-equation}.
This is useful for backward construction (section \ref{backward}).
%Therefore first we prove this result and then move on to the construction and gluing of solutions.
	\begin{definition}[Forward rarefaction from the interface]
		We say that the solution $u$ admits a forward rarefaction on the interface if $\exists \  0<x_{1}<x_{2}<R_{2}(t)$ or $L_{2}(t)<x_{1}<x_{2}<0$, $t_{0}\in (0,t)$ and $p_{1}\in ch_{+}(x_{1},t), q_{1}\in ch_{+}(x_{2},t)$ such that $t_{0}=t-x_{1}/p_{1}=t-x_{2}/q_{1}.$
	\end{definition}
\begin{lemma}\label{2.2}
	There does not exists forward rarefaction from the interface (see figure \ref{fig-3}).
\end{lemma}
\begin{figure}
	\centering
	\begin{tikzpicture}
	\draw (-5,0)--(7,0);
	\draw (0,0)--(0,5.5);
	\draw (-5,5.5)--(7,5.5);
	\draw [dashed](-5,3)--(7,3);
	\draw (-1,0)--(0,2);\draw (0,2)--(1.5,3);
	\draw (0,3)--(4.5,5.5);\draw (0,3)--(3,5.5);
	\draw[dashed](0,3)--(3.3,5.5);
	\draw[dashed](0,3)--(3.6,5.5);
	\draw[dashed](0,3)--(3.9,5.5);
	\draw[dashed](0,3)--(4.2,5.5);
	\draw (-2.25,0)--(0,4);\draw (0,4)--(2,5.5);
	\draw[very thick,blue](0,5.5)--(0,0);
	\draw [very thick,red](2,3) .. controls (4,4) and (4,3) .. (6,5.5);
	%numbering
	\draw (7.5,0) node{\small$t=0$};
	\draw (7.5,5.5) node{\small$t=T$};
	\draw (1.5,5.75) node{\scriptsize $(x_{0,k},T)$};
	\draw (3,5.75) node{\scriptsize$(x_{0},T)$};
	\draw (4.5,5.75) node{\scriptsize$(x_{1},T)$};
	\draw (7.5,3) node{\small$t=t_{0}$};
	\draw[<-](4.7,4)--(5.2,4);
	\draw (6,4) node{\small$R_{1}$ curve };
	\draw (-1.3,4.5) node{\small$L_{1}$ curve };
	\draw[<-](-.1,4.5)--(-.5,4.5);
	\draw (1.5,2.99) node{\tiny$\bullet$  };
	\draw (2,5.5) node{\tiny$\bullet$};
	\draw (3,5.5) node{\tiny$\bullet$};
	\draw (4.5,5.5) node{\tiny$\bullet$};
	%	\draw (-2.25,0) node{\tiny$\bullet$};
	\draw (0,3) node{\tiny$\bullet$};
	\draw (1.5,3.25) node{\scriptsize $(x_{1,k},t_{0})$  };
	\draw (0,-.25) node{\scriptsize $x=0$};
	\draw (0,0)node{\tiny$\bullet$};
	\end{tikzpicture}
	\caption{Rarefaction from the interface cannot occur, hence the above figure is not possible.}\label{fig-3}
	%\captionof{figure}{An illustration of the Lemma \ref{norarefaction}, showing that rarefaction from  the interface is not possible.}
\end{figure}
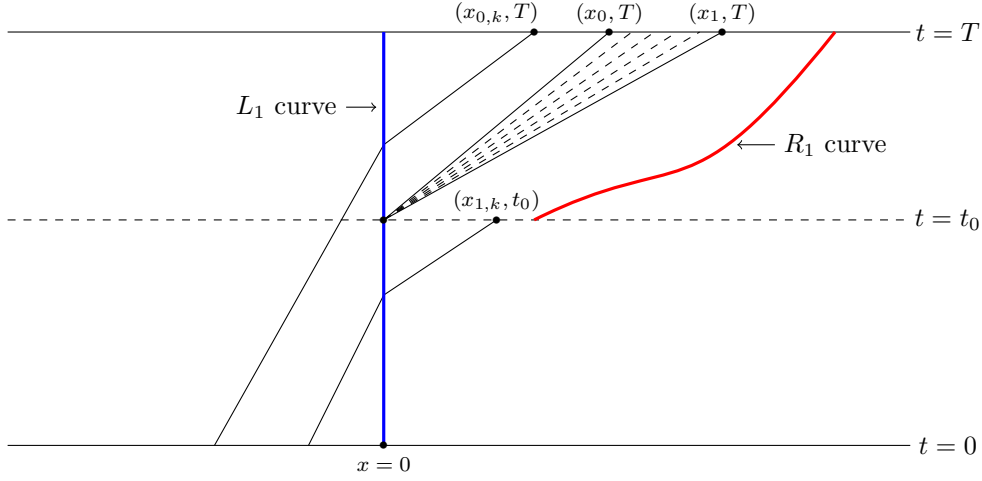
\begin{proof}
	Suppose not, without loss of generality we can assume that there exist $0<x_{1}<x_{2}<R_{2}(t)$, $p_{1}\in ch_{+}(x_{1},t), q_{1}\in ch_{+}(x_{2},t)$, $t_{0}\in(0, t)$ such that $t_{0}=t-x_{1}/p_{1}=t-x_{2}/q_{1}.$ Therefore from \descref{t1item1}{(1)} of theorem \ref{theorem3.1}, if $x_{1}<x<x_{2}, \gamma\in ch(x,t)$, then $\exists \ p_{1}(x)\in ch_{+}(x,t), p_{3}(x)\in ch_{-}(x,t)$ with
	\begin{eqnarray*}
		\gamma(\theta)=
\left\{\begin{array}{llll}			x+p_{1}(x)(\theta-t) &\mbox{ if }& t-\frac{x}{p_{1}(x)}\le \theta\le t,\\
			p_{3}(x)\bigg(\theta -t+\frac{x}{p_{1}(x)}\bigg)&\mbox{ if }& 0\le x\le t-\frac{x}{p_{1}(x)},
\end{array}\right.	\end{eqnarray*}
and $t_{0}=t-x/p_{1}(x)$. Hence for each $x_{1}<x<x_{2},\ p_{1}(x)$ is unique.

Let $u_{0,k}\in C(\R)\cap L^{\f}(\R)$ be satisfying $u_{0,k}\rr u_{0}$ in $L^{1}_{loc}(\R)$ as $k\rr \f$. Let $x_{1}<x<y<x_{2}$, then from \descref{t1item3}{(3)} of theorem \ref{theorem3.1}, for large $k\in\mathbb{N}$, there exist $\gamma_{k}\in ch_k(x,t), \dot{\gamma_{k}}(\theta)=(p_{1,k},\emptyset, p_{3,k}), \eta_{k}\in ch_{k}(y,t),  \dot{\eta_{k}}(\theta)=(q_{1,k},\emptyset, q_{3,k})$ such that
$\lim\limits_{k\rr\f}((p_{1,k},\emptyset, p_{3,k}),  (q_{1,k},\emptyset, q_{3,k}))=((\ti{p}_{1},\emptyset, p_{3}), (\ti{q}_{1},\emptyset, q_{3})).$
From lemma \ref{stability} and by the uniqueness of $p_{1}(x)$, $p_{1}(y)$, it follows that $\tilde{p_{1}}=p_{1}(x), \tilde{q_{1}}=p_{1}(y)$. Since $v_{0,k}\in C^{1}(\R)$, hence by minimizing property, we have
\begin{eqnarray}
\frac{\partial}{\partial p_{3}}\Gamma_{v_{0,k},\lambda}(x,t)|_{\lambda=\gamma}&=&0,\\
\frac{\partial}{\partial p_{1}}\Gamma_{v_{0,k},\lambda}(x,t)|_{\lambda=\gamma}&=&0,\\
g(g^{*'}(p_{3,k}))&=&f(f^{*'}(p_{1,k})),\\
g(g^{*'}(q_{3,k}))&=&f(f^{*'}(q_{1,k})),
\end{eqnarray}
and from lemma \ref{lemma2.1}, we have
$$-p_{3,k}\bigg(t-\frac{x}{p_{1,k}}\bigg)\le-q_{3,k}\bigg(t-\frac{y}{q_{1,k}}\bigg).
$$ Letting $k\rr\f$ to obtain 
\begin{eqnarray}\label{eq4.6}
\begin{array}{lll}
g(g^{*'}(p_{3}))&=&f(f^{*'}(p_{1}(x))),\\
g(g^{*'}(q_{3}))&=&f(f^{*'}(p_{1}(y))),
\end{array}
\end{eqnarray}
 hence 
$$-p_{3}\bigg(t-\frac{x}{p_{1}(x)}\bigg)\le-q_{3}\bigg(t-\frac{y}{p_{1}(y)}\bigg).$$
Since $t-\frac{x}{p_{1}(x)}=t-\frac{y}{p_{1}(y)}=t_{0}$, we have $q_{3}\le p_{3}$. Due to $q_{3}\ge0$, we obtain $0\le q_{3}\le p_{3}$ and $\theta_{g}\le g^{*'}(q_{3}), \theta_{f}\le f^{*'}(p_{3})$. As $g$ is an increasing function on $(\theta_{g}, \f)$, we get $g(g^{*'}(p_{3}))\ge g(g^{*'}(q_{3}))$. This implies that $f(f^{*'}(p_{1}(x)))\ge f(f^{*'}(p_{1}(y))).$ Because of the fact that $p_{1}(x)\ge 0$, $q_{1}(y)\ge 0$ and $f$ is an increasing function on $(\theta_{f},\f)$, it follows that $p_{1}(x)\ge p_{1}(y)$. Therefore we have 
$\frac{x}{p_{1}(x)}<\frac{y}{p_{1}(x)}\le\frac{y}{p_{1}(y)}=\frac{x}{p_{1}(x)},$
which is contradiction. This proves the Lemma.
\end{proof}

%\begin{center}
%	\begin{tikzpicture}
%	\draw (-5,0)--(7,0);
%	\draw (0,0)--(0,5.5);
%	\draw (-5,5.5)--(7,5.5);
%	\draw [dashed](-5,3)--(7,3);
%	\draw (-1,0)--(0,2);\draw (0,2)--(1.5,3);
%	\draw (0,3)--(4.5,5.5);\draw (0,3)--(3,5.5);
%	\draw[dashed](0,3)--(3.3,5.5);
%	\draw[dashed](0,3)--(3.6,5.5);
%	\draw[dashed](0,3)--(3.9,5.5);
%	\draw[dashed](0,3)--(4.2,5.5);
%	\draw (-2.25,0)--(0,4);\draw (0,4)--(2,5.5);
%	\draw[very thick](0,5.5)--(0,3);
%	\draw [very thick](2,3) .. controls (4,4) and (4,3) .. (6,5.5);
%	%numbering
%	\draw (7.5,0) node{\small$t=0$};
%	\draw (7.5,5.5) node{\small$t=T$};
%	\draw (1.5,5.75) node{\footnotesize$(x_{0,k},T)$};
%	\draw (3,5.75) node{\footnotesize$(x_{0},T)$};
%	\draw (4.5,5.75) node{\footnotesize$(x_{1},T)$};
%	\draw (7.5,3) node{\small$t=t_{0}$};
%	\draw (6,4) node{\small$R_{1}$ curve };
%	\draw (-1,4.5) node{\small$L_{1}$ curve };
%	\draw (1.5,2.99) node{\tiny$\bullet$  };
%	\draw (1.5,3.25) node{\footnotesize$(x_{1,k},t_{0})$  };
%	\draw (0,-.25) node{\small$x=0$};
%	\end{tikzpicture}
%	%\captionof{figure}{An illustration of the Lemma \ref{norarefaction}, showing that rarefaction from  the interface is not possible.}
%\end{center}

As an immediate consequence of this lemma, we have the following:
\begin{lemma}\label{lemma4.3}
	Let $u,R_{2}(T),L_{2}(T),t_{+}, t_{-}$ be as in theorem \ref{theorem3.1}. Then for all $t>0$, $x\mapsto t_{+}(x,t), x\in (0, R_{2}(t))$ is a strictly decreasing function and $x\mapsto t_{-}(x,t)$, $x\in(0,L_{2}(t))$ is strictly increasing function.
\end{lemma}
	\begin{proof}
		We prove this for $t_{+}(x,t)$ and similarly proof holds for $t_{-}(x,t)$. Suppose $x\mapsto t_{+}(x,t)$ is not strictly decreasing function, then there exist $0<x_{1}<x_{2}<R_{2}(t)$ and a $t_{0}\in(t_{+}(R_{2}(t)-,t),t)$ such that $t_{0}=t_{+}(x_{1},t)=t_{+}(x_{2},t)$. Then there exist $p_{1}$ and $q_{1}$ such that $t_{+}(x_{1},t)=t-\frac{x_{1}}{p_{1}}, t_{+}(x_{2},t)=t-\frac{x_{2}}{p_{2}}$. Hence $u$ admits a forward rarefaction from the interface and from lemma \ref{2.2} we get a contradiction. This proves the lemma.
	\end{proof}
			\begin{definition}
%			\noindent Case (1): $f(\theta_{f})\le g(\theta_{g})$. 
%			Then define $\bar{\bar{\T}}_g\leq \T_f \leq \bar{\T}_g$ such that 
%			\begin{eqnarray*}
%			f(\bar{\bar{\T}}_g)=f(\bar{\T_g})=g(\T_g).
%%			f'(\bar{\bar{\T}}_g)\leq 0, f'(\bar{\T}_g)\geq 0.
%			\end{eqnarray*}
			Let $I_+= [f'(\bar{\theta}_{g}),\f)$, $g_+=g|_{[\T_g,\f)}$, then define 
			$h_{+}:I_{+}\rr[0,\f)$ by
						\begin{equation*}
						h_{+}(p)=g'\circ g_{+}^{-1}\circ f\circ (f')^{-1}(p).
						\end{equation*}
%						\noindent Case (2):$f(\theta_{f})\geq g(\theta_{g})$. 
%			Then define $\bar{\bar{\T}}_f\leq \T_g \leq \bar{\T}_f$ such that 
%%			\begin{eqnarray*}
%%			g(\bar{\bar{\T}}_f)=g(\bar{\T_f})=f(\T_f).
%%			%g'(\bar{\bar{\T}}_f)\leq 0, g'(\bar{\T}_f)\geq 0.
%%			\end{eqnarray*}
%			Let $I_-= (-\f, g'(\bar{\bar{\T}}_f)]$, $f_-=f|_{(-\f, \T_f]}$, then define 
%			$h_{-}:I_{-}\rr(-\f,0]$ by
%						\begin{equation*}
%						h_{-}(p)=f'\circ f_{-}^{-1}\circ g\circ (g')^{-1}(p).
%						\end{equation*}			
								\end{definition}

	\begin{lemma}\label{lemma3.1}
	Let $T>0$ and denote $t_{\pm}(x,T)=t_{\pm}(x)$. Then
%	\begin{enumerate}[label=\arabic*.]
%		\item\label{lemma-3.1.-1}
		 For a.e., $x\in (0,R_{2}(T)),\ 
			-\frac{y_{-,0}(t_{+}(x))}{t_{+}(x)}=h_{+}\left(\frac{x}{T-t_{+}(x)}\right).$
%		\item\label{lemma-3.1.-2} {\color{red}For a.e., $x\in(L_{2}(T),0),\ 
%		-\frac{y_{+,0}(t_{-}(x))}{t_{-}(x)}=h_{-}\left(\frac{x}{T-t_{-}(x)}\right).$}	\end{enumerate}
	\end{lemma} 
		\begin{proof}
			Let $R_{2}(T)>0$ and $u_{0}\in C_{0}(\R)\cap L^{\f}(\R)$. Let $x\in(0,R_{2}(T))$, then from non-intersecting of characteristics, $L_{1}(T)=0$. Let $\eta\in ch(x,T)$, $\dot{\eta}=(q_{1},\emptyset,q_{3})$, then we have		$	0=\frac{\partial}{\partial q_{3}}\Gamma_{v_{0},\eta}(x,T)=\frac{\partial}{\partial q_{1}}\Gamma_{v_{0},\eta}(x,T).$			This implies that 
			\begin{align*}
			v_{0}^{'}\left(-q_{3}\left(x-\frac{T}{q_{1}}\right)\right)=g^{*'}(q_{3})\ \mbox{and }			v_{0}^{'}(-q_{3}(x-\frac{T}{q_{1}}))=q_{1}f^{*'}(q_{1})-f^{*}(q_{1})+g^{*}(q_{3}).
			\end{align*}
			As $f(f^{*'}(q))=qf^{*}(q)-f^{*}(q)$, therefore we have from the above identities $f(f^{*'}(q_{1}))=g(g^{*'}(q_{3}))$. Due to $q_{1}\ge 0$, $q_{3}\ge 0$, we get $q_{3}=h_{+}(q_{1})$.
			Since $h_{+}$ is an increasing function, hence if $ch_{+}(x,T)=\{q_{1}\}$, then $ch_{-}(x,T)=\{q_{3}\}$. Therefore if $ch_{+}(x,T)=\{q_{1}\}$ then $\{\eta\}=ch(x,T)$ and $q_{1}=\frac{x}{T-t_{+}(x)}$, $q_{3}=-\frac{y_{-,0}(t_{+}(x))}{t_{+}(x)}$ and 
			\begin{equation}\label{eq-5}
			-\frac{y_{-,0}(t_{+}(x))}{t_{+}(x)}=h_{+}\bigg(\frac{x}{T-t_{+}(x)}\bigg).
			\end{equation}
			Let $u_{0}\in L^{\f}(\R)$ and $u_{0,k}\in C(\R)\cap L^{\f}(\R)$ such that $u_{0,k}\rr u_{0}$ in $L^{1}_{loc}(\R)$ as $k\rr\f$. Then from lemma \ref{stability}, 
$			\lim\limits_{k\rr\f} ch_{k}(x,T)\subset ch(x,T).
$			

Let $u_0\in L^\f(\R)$ and $u_{0,k}\in C^0(\R)\cap L^\f(\R)$ such that $u_{0,k}\rr u_0$ in $L^1_{\mbox{loc}}(\R)$ as $k\rr \f$. Then from lemma \ref{stability}, $R_{2,k}(T)\rr R_2(T)$ and $ch_k(x,t)\rr ch(x,t)$ as $k\rr \f$. From lemma \ref{lemma4.3}, for $x\in (0,R_2(T))$, $x\mapsto t_+(x)$ is a strictly decreasing function. Hence from \descref{t1item3}{3}-\descref{iv.}{(iv)} of theorem \ref{theorem3.1}, $x\mapsto y_{-,0}(t_+(x))$ is a non decreasing function. Therefore there exists a countable set $N\subset (0,R_2(T))$ such that for $x\notin N$, $ch(x,T)$ is a singleton set. Therefore from (\ref{lemma-stab-3}) of lemma \ref{stability}, for $x\notin N$, $\lim\limits_{k\rr \f} ch_k(x,T)=ch(x,T)$ and hence $\lim\limits_{k\rr \f}\left(t_{+,k}(x), y_{-,k}(t_{+,k}(x))\right)=\left(t_+,y_{-,0}(t_+(x))\right).$ 
Hence from \eqref{eq-5} for $x\notin N$, $x\in (0,R_2(T))$
$$-\frac{y_{-,0}(t_+(x))}{t_+(x)}=h_+\left(\frac{x}{T-t_+}\right).$$
This proves the lemma.
%\eqref{lemma-3.1.-1} {\color{red}and similarly \eqref{lemma-3.1.-2} follows.}
			\end{proof}
%	\begin{lemma}
%		Let $\rho:[\alpha,\beta] \subset(-\f,0]$ be a non-decreasing function and $T>0$. Then 
%		\begin{itemize}
%			\item[1.]

\begin{lemma}\label{3.3}
Let $[\al,\beta]\subset [0,\f)$ and $\rho:[\al,\beta]\rr (-\f,0]$ be a non decreasing function such that if $\rho(x)=0$, then 
$T-\frac{x}{f^{\p}(\bar{\T}_g)}\geq 0.$
\begin{enumerate}[(1)]
\item\label{3.3-1} Then there exists a strictly decreasing function 

 $t:(\al,\beta]\rr [0,T]$ such that for all $x\in (\al,\beta]$
\begin{enumerate}[(i)]
\item\label{3.3-1.i} $\frac{x}{T-t_+(x)}\in I_+$.
\item\label{3.3-1.ii} $-\frac{\rho(x)}{t_+(x)}=h_+\left(\frac{x}{T-t_+(x)}\right)$.
\item\label{3.3-1.iii}  If $\rho$ is continuous, then $t_+(x)$ is continuous.
\item\label{3.3-1.iv}  If $\al>0$, then $\lim\limits_{x\downarrow\al}t_+(x)=t(\al+)$ exist and satisfies 
$\frac{\rho(\al)}{t(\al+)}=h_+\left(\frac{\al}{T-t(\al+)}\right).$
\item\label{3.3-1.v}  If $\al=0$, then $t(0+)=T$ and $\lim\limits_{x\downarrow 0}\frac{x}{T-t_+(x)}=f'(p_0)$ exist and satisfies $-\frac{\rho(0)}{T}=h_+(f'(p_0)).$
\end{enumerate}
\item\label{3.3-2} Let $x_{0}>0$, $0<t_{2}<t_{1}<T$ such that $\frac{x_{0}}{T-t_{i}}\in I_{+}$ for $i=1,2$. Define for $i=1,2$,
			\begin{eqnarray}
			f'(a_{i})&=&\frac{x_{0}}{T-t_{i}},\\
			-\frac{\rho_{i}}{t_{i}}&=&h_{+}\bigg(\frac{x_{0}}{T-t_{i}}\bigg).\label{3}
			\end{eqnarray}
			Suppose $\rho_{1}<\rho_{2}\le 0$, then there exist $t_{3}\in(t_{2},t_{1})$, $\rho_{3}\in(\rho_{1},\rho_{2})$, $b_{1},b_{2}$ with $a_{2}<a_{1}, b_{2}<b_{1}$ such that 
			\begin{eqnarray}
			g(b_{i})&=&f(a_{i}), g'(b_{i})\ge0,\\
			-x_{0}&=&(T-t_{3})\bigg(\frac{f(a_{1})-f(a_{2})}{a_{1}-a_{2}}\bigg),\\
			-x_{0}&=&t_{3}\bigg(\frac{g(b_{1})-g(b_{2})}{b_{1}-b_{2}}\bigg).\label{6}
			\end{eqnarray}
			\end{enumerate}
			\end{lemma}
	\begin{proof}
Now $\frac{x}{T-t}\in I_+$ if and only if $t\geq T-\frac{x}{f'(\bar{\T}_g)}$. For fixed $x\in (\al,\beta]$ define
$F(t):=h_+\left(\frac{x}{T-t}\right) +\frac{\rho(x)}{t}$ for $t\in \left[\max\left\{0,T-\frac{x}{f^\p(\bar{\T}_g)}, \right\}, T\right]$. If $\rho(x)=0$, then by the hypothesis, $T-\frac{x}{f^\p(\bar{\T}_g)}\geq 0$. Hence for $t_+(x)=T-\frac{x}{f^\p(\bar{\T}_g)}$, $F(t_+(x))=0$. Let us consider the case $\rho(x)\neq 0$. If $T-\frac{x}{f^\p(\bar{\T}_g)}\leq 0$, then take $t_0=0$ to obtain $F(0)=-\f$. If $T-\frac{x}{f^\p(\bar{\T}_g)}>0$, then take $t_0=T-\frac{x}{f^\p(\bar{\T}_g)}$ to obtain $F(t_0)<0$. As $x\neq 0$, we take $t=T$ to obtain $F(T)=\f$. Since $t\mapsto F(t)$ is continuous, hence there exists a $t_+(x)\in \left[\max\left\{0,T-\frac{x}{f^\p(\bar{\T}_g)}, \right\}, T\right]$ such that $F(t_+(x))=0$.
\par Let $\al<x_1<x_2<\beta$ and suppose $t(x_1)\leq t(x_2)$. Then 
$\frac{x_1}{T-t(x_1)}\leq \frac{x_1}{T-t(x_2)}<\frac{x_2}{T-t(x_2)}$. Subsequently, we have 
\begin{eqnarray*}
-\rho(x_1)&=& t(x_1)h_+\left(\frac{x_1}{T-t(x_1)}\right) < t(x_1) h_+\left(\frac{x_2}{T-t(x_2)}\right)\\
&\leq & t(x_2) h_+\left(\frac{x_2}{T-t(x_2)}\right)=-\rho(x_2),
\end{eqnarray*}
contradicting the non decreasing hypothesis on $\rho.$ This proves \ref{3.3-1.i} and \ref{3.3-1.ii}. If $\al>0$, then $T-\frac{\al}{f^\p(\bar{\T}_g)}<T$ and hence \ref{3.3-1.i} and \ref{3.3-1.ii} together imply \ref{3.3-1.iv}. Let $\al=0$ and $t_0=\lim\limits_{x\downarrow 0} t_+(x).$ Suppose $t_0<T$, then $\frac{x}{T-t_+(x)}\rr 0$ as $x\rr 0$. As $-\frac{\rho(0)}{t_0}=h_+(0)$, we have $0\in I_+.$ Therefore, $f'(\bar{\T}_g)=0$ and $h_+(0)=0$. This implies $\rho(0)=0$. Due to the fact that $\rho$ is a non decreasing function we obtain for all $x\in [0,\beta],$ $0=\rho(0)\leq \rho(x) \leq 0$ and therefore $\rho(x)=0$ for all $x\in [0,\beta]$. But for $x\neq 0, x\in (0,\beta]$, $-\f=T-\frac{x}{f^\p(\bar{\T}_g)}\geq 0$, which is a contradiction. Hence $t_0=T$ and $\frac{x}{T-t_+(x)}=p(x)$ is bounded. Let $p_0=\lim\limits_{x\rr 0} p(x),$ then $-\frac{\rho(0)}{T}=h_+(p_0).$ This proves \ref{3.3-1.v}.

%		Let $\alpha=\max\{0,f'(\overline{\theta}_{g})\}$. Then $\frac{x}{T-t}\in I_{+}$ If and only if $t\in\bigg[T-\frac{x}{\alpha},T\bigg]$. Define for $t\in\bigg[T-\frac{x}{\alpha},T\bigg]$,
%$		F(t)=h_{+}\bigg(\frac{x}{T-t}\bigg)+\frac{\rho(x)}{t}.$\\
%		Case(i): Let $\rho(x)=0$, then by then by the hypothesis, $f(\theta_{f})\leq g(\theta_{g})$ and $T-\frac{x}{f'(\overline{\theta}_{g})}\ge0$. Hence $t_+(x)=T-\frac{x}{f'(\overline{\theta}_{g})}$ satisfies $F(t_{+}(x))=0$.
%		
%		\noindent Case(ii): Let $\rho(x)\not=0$, then $F(T)=\f$. Suppose $T-\frac{x}{\alpha}\ge 0$, then $F\left(T-\frac{x}{\alpha}\right)=h_{+}(\alpha)+\frac{\rho(x)}{T-\frac{x}{\alpha}}<0$. If $T-\frac{x}{\alpha}<0$, then $F(0)=-\f$. Hence there exists a $t_+(x)$ such that $F(t_+(x))=0$. Let $0<x_{1}<x_{2}$ and suppose that $t(x_{1})\le t(x_{2})$. Then 
%$		\frac{x_{1}}{T-t(x_{1})}\le \frac{x_{1}}{T-t(x_{2})}\le \frac{x_{2}}{T-t(x_{2})}.
%$		Therefore we have 
%		\begin{eqnarray*}
%		-\rho(x_{1})=t(x_{1})h_{+}\bigg(\frac{x_{1}}{T-t(x_{1})}\bigg) <t(x_{1})h_{+}\bigg(\frac{x_{2}}{T-t(x_{2})}\bigg)
%		\le t(x_{1})h_{+}\bigg(\frac{x_{2}}{T-t(x_{2})}\bigg)
%		=-\rho(x_{2}).
%		\end{eqnarray*}
%	Contracting the non-decreasing hypothesis on $\rho$. Hence if $\rho(x)$ is continuous, then $t_+(x)$ is continuous by uniqueness. This proves (1).
		Proof of \ref{3.3-2}: As $t_{2}<t_{1}$, we have $\frac{x_{0}}{T-t_{2}}<\frac{x_{0}}{T-t_{1}}$. Thus, $a_{2}<a_{1}$ and $b_{2}<b_{1}$. By the choice of $a_{1}$ and $a_{2}$, $x_{0}$ satisfies $x_{0}+(t_{i}-T)f'(a_{i})=0$, for $i=1,2$. Since $f$ is convex, we get $f'(a_2)<\frac{f(a_{2})-f(a_{1})}{a_{2}-a_{1}}<f'(a_{1})$, hence the line $r(\theta)=x_{0}+(\theta-T)\frac{f(a_{2})-f(a_{1})}{a_{2}-a_{1}}$ meet the $t-axis$ at $t_{3}\in(t_{2},t_{1})$, that is $r(t_{3})=0$. This proves \eqref{6}. Again from the convexity of $g$, $g'(b_{2})<\frac{g(b_{2})-g(b_{1})}{b_{2}-b_{1}}<g'(b_{1})$ and thus 
$	-t_{3}g'(b_{1})<-t_{3}\frac{g(b_{2})-g(b_{1})}{b_{2}-b_{1}}<-t_{3}g'(b_{2}).
$	From \eqref{3}, we have, $-\rho_{i}=t_{i}h_{+}(f'(a_{i}))=t_{i}g'(b_{i})$ and $t_{3}\in(t_{2},t_{1})$ implies that $\rho_{3}=-t_{3}\bigg(\frac{g(b_{2})-g(b_{1})}{b_{2}-b_{1}}\bigg)\in(\rho_{1},\rho_{2})$. This proves the lemma.
		\end{proof}
		\subsection{Building blocks: Construction of shock solution and continuous solution}
 	\begin{lemma}\label{3.4}(Shock solution)
		Let $T>0$, $x_{0}>0$, $\rho_{1}<\rho_{2}\le0$. Assume that for $t\in[0,T]$, $\frac{x_0}{T-t}\geq f'(\bar{\T}_g)$ and  if $\rho_{2}=0$, then $T-\frac{x_{0}}{f'(\overline{\theta}_{g})}=0$. Let $a_{1}, a_{2}, b_{1}, b_{2}, t_{1}, t_{2}, t_{3}$ and $\rho_{3}$ be as in lemma \ref{3.3}. Define
		\begin{eqnarray}
		u_0(x)=\begin{cases*}
		b_{1} \quad \mbox{ if } x<\rho_{3},\\
		b_{2} \quad \mbox{ if } \rho_{3}<x<0,\\
		a_{2}\quad \mbox{ if } x>0,
		\end{cases*}
		\end{eqnarray}
		then the solution $u$ of \eqref{conlaw-equation} in $\R\times[0,T]$ with initial data $u_{0}$ is given by (see figure \ref{fig-4})
		\begin{eqnarray}
		u(x,t)=\left\{\begin{array}{llll}
		b_{1} &\mbox{ if }& x<0, x<\rho_{3}+\frac{g(b_{1})-g(b_{2})}{b_{1}-b_{2}}t,\\
		b_{2}  &\mbox{ if }& x<0, x>\rho_{3}+\frac{g(b_{1})-g(b_{2})}{b_{1}-b_{2}}t,\\
		a_{1} &\mbox{ if }& x>0, x<\frac{f(a_{1})-f(a_{2})}{a_{1}-a_{2}}(t-t_{3}),\\
		a_{2}  &\mbox{ if }& x>0, x>\frac{f(a_{1})-f(a_{2})}{a_{1}-a_{2}}(t-t_{3}).
\end{array}\right.		\end{eqnarray}		
\end{lemma}
\begin{proof}
	From lemma \ref{3.3}, $\rho_{3}=-t_{3}\bigg(\frac{g(b_{1})-g(b_{2})}{b_{1}-b_{2}}\bigg)$ and $f(a_{i})=g(b_{i})$, hence $u$ is a weak solution satisfying the interior and interface entropy condition with initial data $u_{0}$. This proves the lemma.
\end{proof}
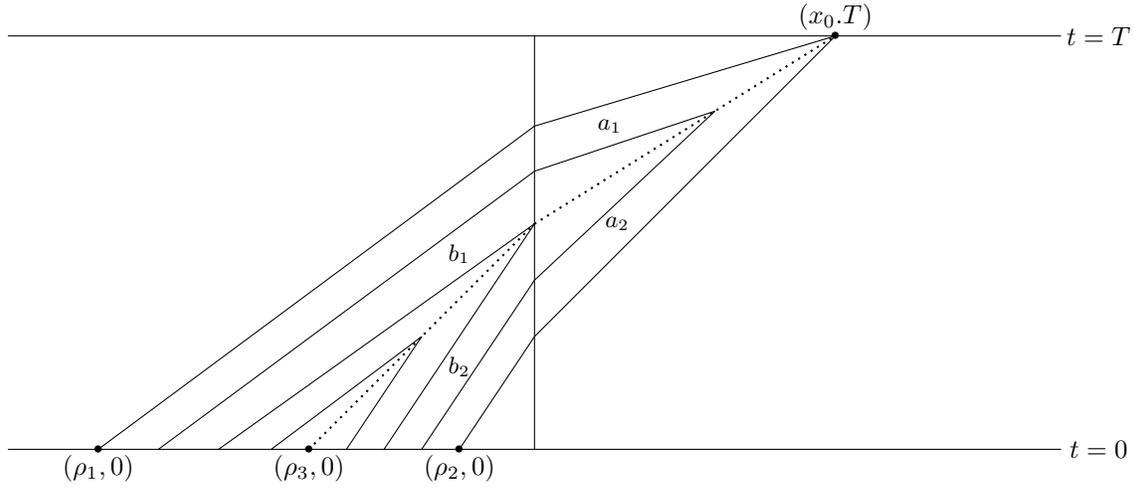
\begin{figure}
	\centering
	\begin{tikzpicture}
	\draw (-7,0)--(7,0);
	\draw (0,0)--(0,5.5);
	\draw (-7,5.5)--(7,5.5);
	\draw (-1,0)--(0,1.5);
	\draw (-1.5,0)--(0,2.25);
	\draw (-2,0)--(0,3);
	\draw (-2.5,0)--(-1.5,1.5);
	\draw [dotted,thick](-3,0)--(0,3);
	\draw (-3.5,0)--(-1.5,1.5);
	\draw (-4.2,0)--(0,3);
	\draw (-5,0)--(0,3.7);
	\draw (-5.8,0)--(0,4.3);
	\draw (0,1.5)--(4,5.5);
	\draw [dotted,thick] (0,3)--(4,5.5);
	\draw (0,2.25)--(2.4,4.5);
	\draw (0,3.7)--(2.4,4.5);
	\draw (0,4.3)--(4,5.5);
	%numbering
	\draw (-3,0) node{\tiny$\bullet$};
	\draw (-5.8,0) node{\tiny$\bullet$};
	\draw (-1,0) node{\tiny $\bullet$};
	\draw (4,5.5) node{\tiny $\bullet$};
	\draw (-5.8,-.25) node{\small$(\rho_{1},0)$};
	\draw (-3,-.25) node{\small$(\rho_{3},0)$};
	\draw (-1,-.25) node{\small$(\rho_{2},0)$};
	\draw (7.5,0) node{\small$t=0$};
	\draw (7.5,5.5) node{\small$t=T$};
	\draw (-1,2.6) node{\footnotesize$b_{1}$};
	\draw (-1,1.1) node{\footnotesize$b_{2}$};
	\draw (1,4.3) node{\footnotesize$a_{1}$};
	\draw (1.1,3) node{\footnotesize$a_{2}$};
	\draw (4,5.75) node{\small$(x_{0}.T)$};
	\end{tikzpicture}
	\caption{The figure illustrates shock solution.}\label{fig-4}
\end{figure} 
%\begin{center}
%		\begin{tikzpicture}
%		\draw (-7,0)--(7,0);
%		\draw (0,0)--(0,5.5);
%		\draw (-7,5.5)--(7,5.5);
%		\draw (-1,0)--(0,1.5);
%		\draw (-1.5,0)--(0,2.25);
%		\draw (-2,0)--(0,3);
%		\draw (-2.5,0)--(-1.5,1.5);
%		\draw [dotted,thick](-3,0)--(0,3);
%		\draw (-3.5,0)--(-1.5,1.5);
%		\draw (-4.2,0)--(0,3);
%		\draw (-5,0)--(0,3.7);
%		\draw (-5.8,0)--(0,4.3);
%		\draw (0,1.5)--(4,5.5);
%		\draw [dotted,thick] (0,3)--(4,5.5);
%		\draw (0,2.25)--(2.4,4.5);
%		\draw (0,3.7)--(2.4,4.5);
%		\draw (0,4.3)--(4,5.5);
%		%numbering
%		\draw (-3,0) node{\tiny$\bullet$};
%		\draw (-5.8,0) node{\tiny$\bullet$};
%		\draw (-1,0) node{\tiny $\bullet$};
%		\draw (-5.8,-.25) node{\small$(\rho_{1},0)$};
%		\draw (-3,-.25) node{\small$(\rho_{3},0)$};
%		\draw (-1,-.25) node{\small$(\rho_{2},0)$};
%		\draw (7.5,0) node{\small$t=0$};
%		\draw (7.5,5.5) node{\small$t=T$};
%		\draw (-1,2.6) node{\footnotesize$b_{1}$};
%		\draw (-1,1.1) node{\footnotesize$b_{2}$};
%		\draw (1,4.3) node{\footnotesize$a_{1}$};
%		\draw (1.1,3) node{\footnotesize$a_{2}$};
%		\draw (4,5.75) node{\small$(x_{0}.T)$};
%		\end{tikzpicture}
%		%\captionof{figure}{The dotted line is the shock originating from the point $(\rho_{3},0)$ until the point $(x_{0},T)$.}
%	\end{center} 
	\begin{remark}
Suppose $x_0=0$, then by \ref{3.3-1}--\ref{3.3-1.v} of lemma \ref{3.3}, we have $f'(p_0)=\lim\limits_{x\rr 0} \frac{x}{T-t_+(x)}$, hence $f'(p_0)\in I_+$ implies that $p_0\geq \bar{\T}_g$. Let $q_0 \geq \T_g$ such that $f(p_0)=g(q_0).$ Hence 
$$-\frac{\rho(0)}{T}=h_+\left(f'(p_0)\right)=g'(q_0).$$
Define 
\begin{eqnarray*}
u_0(x)=\left\{\begin{array}{lllll}
q_0 &\mbox{if}& x\leq 0,\\
p_0 &\mbox{if}& x\geq 0,\\
\end{array}\right.
\end{eqnarray*}
then $u(x,t)=u_0$ is the solution of (\ref{conlaw-equation}), (\ref{entropy-condition}).		\end{remark}

	In lemma \ref{3.4}, under suitable hypothesis on $x_{0}$, $\rho_{i}$, $i=1,2$, we constructed a solution which admits shocks. Next we consider the case where $0\leq x_{1}<x_{2}$ and $\rho_{0}<0$. Under a suitable hypothesis, we construct a continuous solution to \eqref{conlaw-equation}.
	
	Let $T>0$, $0\le x_{1}<x_{2}$, $\rho_{0}<0$. From lemma \ref{3.4}, let $0<t_{i}<T$, $i=1,2$, be such that
	\begin{equation*}
	h_{+}\bigg(\frac{x_{i}}{T-t_{i}}\bigg)=-\frac{\rho_{0}}{t_{i}}.
	\end{equation*}
	Let $f'(a_{i})=\frac{x_{i}}{T-t_{i}}$, $f(a_{i})=g(b_{i})$, $g'(b_{i})>0$. Again from lemma \ref{3.4}, let $t_+(x):[x_{1},x_{2}]\rr[t_{2},t_{1}]$ be the unique continuous strictly decreasing function satisfying
	 \begin{equation*}
	 h_{+}\bigg(\frac{x}{T-t_{+}(x)}\bigg)=-\frac{\rho_{0}}{t_+(x)}, \quad x\in[x_{1},x_{2}].
	 \end{equation*}
	By the uniqueness of $t_+(x)$, $t_+(x_{i})=t_{i}$ and $t_+(\cdot)$ is a homeomorphism.	 

%	Let $\rho_{0}<0$, $0<t_{2}<t_{1}<T$, define $a_{1}, a_{2}, b_{1}, b_{2}$ by 
%	\begin{eqnarray*}
%	g'(b_{1})&=&-\frac{\rho_{0}}{t_{1}},\\
%	g'(b_{2})&=&-\frac{\rho_{0}}{t_{2}},\\
%	f(a_{1})&=&g(b_{1}), f'(a_{1})>0\\
%	f(a_{2})&=&g(b_{2}), f'(a_{2})>0.
%	\end{eqnarray*}
	For $i=1, 2$, let 
	\begin{eqnarray*}
	\eta_{i}(t)&=&-\frac{\rho_{0}}{t_{i}}(t-t_{i}),\\
	\gamma_{i}(t)&=&f'(a_{i})(t-t_{i}),\\
	u_{0}(x)&=&\left\{\begin{array}{llll}
	b_{1}&\mbox{if}& x<\rho_{0},\\
	b_{2}&\mbox{if}&\rho_{0}<x<0,\\
	a_{2}&\mbox{if}& x>0.\\
	\end{array}\right.
	\end{eqnarray*}
	For $x\ge0$, let $t(x, t)$ be the unique solution of 
	\begin{equation*}
	h_{+}\left(\frac{x}{t-t_+(x, t)}\right)=-\frac{\rho_{0}}{t_+(x, t)}.
	\end{equation*}
	Then we have following:
	\begin{lemma}\label{3.5}(Continuous solution) (See figure \ref{fig-5})
		Let $\rho_{0}<0$, $0<t_{2}<t_{1}$, $a_{i}, b_{i}, \eta_{i}, \gamma_{i}$, $i=1, 2$ be as above. Let $u(x, t)$ be the solution of \eqref{conlaw-equation} with initial data $u_{0}$ as above. Then 
		\begin{eqnarray}
		u(x, t)=\left\{\begin{array}{llll}
		b_{1}&\mbox{if}& x<\min\{\eta_{1}(t),0\},\\
		(g')^{-1}\left(\frac{x-\rho_{0}}{t}\right)&\mbox{if}& \min\{\eta_{1}(t),0\}<x<\min\{\eta_{2}(t),0\},\\
		b_{2}&\mbox{if}& \eta_{2}(t)<x<0,\\
		a_{2}&\mbox{if }& \max\{0, \gamma_{2}(t)\}<x,\\
		(f')^{-1}\left(\frac{x}{t-t_+(x, t)}\right)&\mbox{if}&\max\{\gamma_{1}(t), 0\}<x<\gamma_{2}(t),\\
		a_{1}&\mbox{if}& 0<x<\gamma_{1}(t).\\
		\end{array}\right.
		\end{eqnarray}
	\end{lemma}
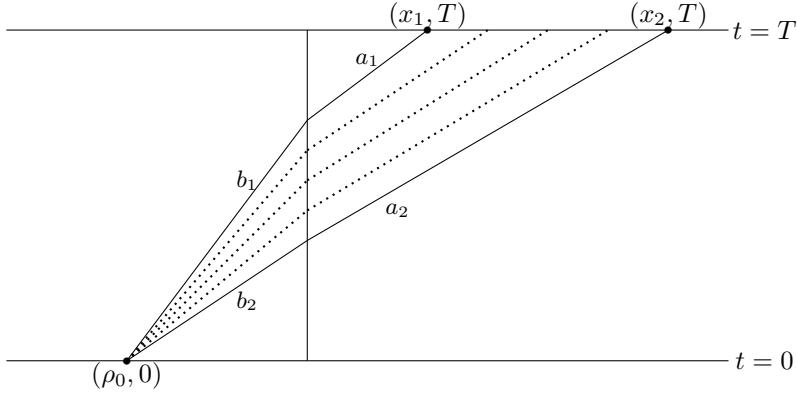
\begin{figure}
	\centering
	\begin{tikzpicture}[scale=.8]
	\draw (-5,0)--(7,0);
	\draw (0,0)--(0,5.5);
	\draw (-5,5.5)--(7,5.5);
	\draw [dotted,thick](-3,0)--(0,3);
	\draw[dotted, thick](-3,0)--(0,2.5);
	\draw[dotted, thick](-3,0)--(0,3.5);
	\draw (-3,0)--(0,2);
	\draw (-3,0)--(0,4);
	\draw (0,2)--(6,5.5);
	\draw (0,4)--(2,5.5);
	\draw[dotted, thick](0,3.5)--(3,5.5);
	\draw[dotted, thick](0,3)--(4,5.5);
	\draw[dotted, thick](0,2.5)--(5,5.5);
	%numbering
	\draw (2,5.75) node{\small$(x_{1},T)$};
	\draw (6,5.75) node{\small$(x_{2},T)$};
	\draw (-3,-.25) node{\small$(\rho_{0},0)$};
	\draw (7.6,0) node{\small$t=0$};
	\draw (7.6,5.5) node{\small$t=T$};
	\draw (-1,3) node{\footnotesize$b_{1}$};
	\draw (-1,1) node{\footnotesize$b_{2}$};
	\draw (1,5) node{\footnotesize$a_{1}$};
	\draw (1.5,2.5) node{\footnotesize$a_{2}$};
	\draw (2,5.5) node{\tiny $\bullet$};
	\draw (6,5.5) node{\tiny $\bullet$};
	\draw (-3,0) node{\tiny $\bullet$};
	\end{tikzpicture}
	\caption{The figure illustrates continuous solution.}\label{fig-5}
	%\captionof{figure}{A rarefaction originating from the point $(\rho_{0},0)$.}
\end{figure}

\begin{proof}

Define the regions in $\R\times(0,T)$ by 
\begin{eqnarray*}
\Omega_1&=&\{(x,t):\ 0<x<\ga_1(t)\},\\
\Omega_2&=&\{(x,t):\ \mbox{max}\{0,\ga_1(t)\}<x<\ga_2(t)\},\\
\Omega_3&=&\{(x,t):\ \mbox{max}\{0,\ga_2(t)\}<x\}.
\end{eqnarray*}
Let $x>0$, $0<t\leq T$ and $w\in c_b(x,t) \cap ch(x,t)$. Then $w=(w_1,\emptyset,w_3)$ is given by 
\begin{eqnarray*}
w_1(\T)&=& x+\frac{x}{t-\tau}(\T-t),\ \tau\leq \T\leq t,\\
w_3(\T)&=& -\frac{w_3(0)}{\tau} (\T-\tau),
\end{eqnarray*}
where $\tau>0$ satisfies $w_3(\tau)=w_1(\tau)=0$. Since $w\in ch(x,t)$, hence $\frac{\partial}{\partial \tau} \Gamma _{v_0,w}(x,t)=0$.  That is 
$$0=-f^*\left(\frac{x}{t-\tau}\right)+\left(\frac{x}{t-\tau}\right)(f^*)'\left(\frac{x}{t-\tau}\right) + g^*\left(-\frac{w_3(0)}{\tau}\right)
+\frac{w_3(0)}{\tau} (g^*)'\left(-\frac{w_3(0)}{\tau}\right).$$
Let $f'(p_1)=\frac{x}{t-\tau}$, $g'(q_1)=-\frac{w_3(0)}{\tau}$, then from the convexity of $f$ and $g$ and the above relation gives $$f\left((f^*)'\left(\frac{x}{t-\tau}\right)\right)=g\left((g^*)'\left(-\frac{w_3(0)}{\tau}\right)\right).$$
That is $f(p_1)=g(q_1).$ Observe that $(x,t)\in \Omega_3$ if and only if $x>\ga_2(t)=f'(a_2)(t-t_2).$ If $t-t_2\leq 0$, then $\tau<t\leq t_2$. Now $\ti{\ga}(t)=g'(b_2)(t-\tau)$ is the only characteristic of $u$ in $x<0$ and $\ti{\ga}(\tau)=0$. Since $w_3$ is a characteristic in $x<0$ and $w_3(\tau)=0$, hence $w_3(t)=g'(b_2)(t-\tau)$. This implies $\rho_0<w_3(0)<0$. 
If $t>t_2$, then $\frac{x}{t-t_2}>f'(a_2)$.
Suppose  $\tau>t_2$, then 
$\frac{x}{t-t_2}< \frac{x}{t-\tau}$ and we get
$$-\frac{w_3(0)}{\tau}=h_+\left(\frac{x}{t-\tau}\right)>h_+\left(\frac{x}{t-t_2}\right).$$
As $y\mapsto t_+(y,t)$ is an increasing continuous function, constant on the line $\ga_2(t)$ and hence for $x>\ga_2(t)$, it follows that $t_+(x,t)<t_2$. Due to $t>t_2$, we therefore have 
$$-\frac{w_3(0)}{\tau}>h_+\left(\frac{x}{t-t_2}\right)>h_+\left(\frac{x}{t-t_+(x,t)}\right)=-\frac{\rho_0}{t_+(x,t)}>-\frac{\rho_0}{t_2}>-\frac{\rho_0}{\tau}$$
and therefore $w_3(0)<\rho_0$. Since no two characteristic intersects, hence $\tau>t_1$ and $-\frac{w_3(0)}{\tau}=g'(b_1)$. Note that $x>\ga_2(t)=f'(a_2)(t-t_2),$ and subsequently, we obtain 
$$-\frac{w_3(0)}{\tau}=g'(b_1)>h_+\left(\frac{x}{t-t_2}\right)>h_+(f'(a_2))=g'(b_2)$$
and therefore $b_1>b_2$. But $b_1<b_2$ which is a contradiction. Hence $\tau<t_2$ and from the non intersecting of characteristics, it follows that $\rho_0<w_3(0)<0$ and $w_3(0)=-g'(b_2)\tau.$
Conversely if, $w=(w_1,\emptyset, w_3)\in ch(x,t)$ and $\rho_0<w_3(0)<0$ then $x>\ga_2(t)$. For let $w_3(\tau)=0$, then $w_3(0)=-g'(b_2)\tau$ and $-\frac{w_3(0)}{\tau}=g'(b_2)=h_+\left(\frac{x}{t-\tau}\right),$ implies that $\frac{x}{t-\tau}=f'(a_2).$ Thus $(x,t)\in \Omega_3$ and $u(x,t)=a_2$. Similarly $(x,t)\in \Omega_1$ if and only if $\forall w\in ch(x,t)$ with $w=(w_1,\emptyset,w_3)$, $w_3(0)<\rho_0$, $u(x,t)=a_1$. 
\par As a consequence of this if $\max\{\ga_1(t),0\}<x<\ga_2(t)$ and $w=(w_1,\emptyset, w_3)\in ch(x,t)$. Then $w_3(0)=\rho_0$ and $-\frac{\rho_0}{\tau}=h_+\left(\frac{x}{t-\tau}\right)$, where $w_3(\tau)=0$. Hence $\tau=t_+(x,t)$ and $t_2<t_+(x,t)<t_1$. Since $x\mapsto t_+(x,t)$ is an increasing function, hence at the point of differentiability of $t(\cdot, t)$, we have $u(x,t)=\frac{\partial}{\partial x} \Gamma_{v,w}(x,t)=(f')^{-1}\left(\frac{x}{t-t_+(x,t)}\right)$. This proves the lemma.

\end{proof}

Previous construction (as in lemma \ref{3.4} and \ref{3.5}) of solutions corresponds to the case when no reflected characteristic occurs. Now we deal with the case when reflected characteristics occur and the definition of $(\tau_0,\xi_0)$ (see subsection \ref{sec:reachable}) which is needed to define the reachable set $\mathcal{R}(T).$ We do it in the region $x>0$. Similar construction follows for $x<0$.

%
%
%Following results to be stated before the concept of reachable set $R(T)$. This is due to the fact that reflections must be taken into consideration which gives a new constraint on the function $y$.

Let $\bar{\bar{\alpha}}\leq \theta_{f}\leq \bar{\alpha}$,
 %$0\le T_{1}<T$, $\xi_{1}>0$ 
 such that
\begin{description}[font=\normalfont]
	\descitem {(i)}{i} $f(\bar{\bar{\alpha}})=f(\bar{\alpha})$.
	\descitem{(ii)}{ii}  $D=\{(R,T):\ T\geq 0, 0\leq R\leq f'(\bar{\al})T\}$ and for $(R_1,T)\in D$, define $0\leq T_1 \leq T$ by $f'(\bar{\al})=\frac{R_1}{T-T_1}.$
	\descitem {(iii)}{iii} $\xi_{1}=-f'(\bar{\bar{\alpha}})T_{1}$.
\end{description}
For $0\le\xi\le\xi_{1}$, define $\beta=\beta(\xi)$, $\tau=\tau(\xi)$ by,
\begin{description}[font=\normalfont]
	\descitem{(iv)}{iv} $\xi=-f'(\bar{\bar{\alpha}})\tau$.
	\descitem{(v)}{v} $f'(\beta)=\frac{R_{1}-\xi}{T}$.
\end{description}
In fact $f'(\bar{\bar{\alpha}})$ and $f'(\beta)$ are the inverse of the slopes of line joining between $(0,\tau)$, $(\xi,0)$ and $(R_{1},T)$, $(\xi,0)$ respectively. Since  $0\le\xi\le\xi_{1}$, we have
\begin{eqnarray*}
\left\{\begin{array}{ll} 
&\bar{\bar{\alpha}}\le \beta\le \bar{\alpha},\\
&\beta=\bar{\alpha} \mbox{ if and only if } T_{1}=0, \xi_{1}=0.
\end{array}\right.
\end{eqnarray*}
For $u\in\R$, define
\begin{eqnarray*}
L_{u}(t)&=&\xi+f'(u)t,\\
\Omega(\xi)&=&\{(x, t): t>0, L_{ \bar{\bar{\alpha}}}(t)<x<L_{\beta}(t)\},\\
v(x,t)&=&(f')^{-1}\left(\frac{x-\xi}{t}\right) \mbox{ for } (x,t)\in\Omega(\xi),
\end{eqnarray*}
the rarefaction wave in the $\Omega(\xi)$, which satisfies the equation 
\begin{equation*}
v_{t}+f(v)_{x}=0 \mbox{ in } \Omega(\xi).
\end{equation*}
Let $k\ge 1$, $\bar{\bar{\alpha}}=u_{0}<u_{1}<\cdots<u_{k}=\beta$ such that 
\begin{description}[font=\normalfont]
	\descitem {(vi)}{vi} $|u_{i+1}-u_{i}|\le \frac{\beta-\bar{\bar{\alpha}}}{k+1}$.
\end{description}
For $0\le i\le k$, define the lines passing through $(\xi,0)$ by 
\begin{description}[font=\normalfont]
	\descitem {(vii)}{vii} $l_{i}(t)=L_{u_{i}}(t)=\xi+f'(u_{i})t$.\\
	\descitem {(viii)}{viii} $v_{k}(x,t)=u_{i}$ if $l_{i-1}(t)\le x<l_{i}(t)$, $1\leq i \leq k$.
	\end{description}
Observe that for $l_{i-1}(t)\le x<l_{i}(t)$, $u_{i-1}\le (f')^{-1}\left(\frac{x-\xi}{t}\right)<u_{i}$ and hence, 
\begin{description}[font=\normalfont]
	\descitem {(ix)}{ix} $\left|v_{k}(x,t)-(f')^{-1}\left(\frac{x-\xi}{t}\right)\right|\le|u_{i}-u_{i-1}|\le \frac{\beta-\bar{\bar{\alpha}}}{k+1}$.\\
	\descitem {(x)}{x} $\lim\limits_{k\rr\f}v_{k}(x,t)=(f')^{-1}\left(\frac{x-\xi}{t}\right)$  uniformly in $\overline{\Omega}(\xi)$.
\end{description}
Then we have the following:
\begin{lemma}\label{3.6} 
	Let $T>0$, $\bar{\bar{\alpha}}\leq \theta_{f} \leq \bar{\alpha}$, $0\le T_{1}<T$, $\xi_{1}\geq 0$ and $L_{u}$ as defined above. Then for all $0\le\xi<\xi_{1}$, there exists a $0\le t_{0}(\xi)\le T$ and a Lipschitz curve $s_{\xi}:[t_{0}(\xi),T]\rr[0,T]$ (See figure \ref{fig-13} for illustration) such that 
	\begin{description}[font=\normalfont]
		\descitem{(1)}{1} $s_\xi(T)=R_1$ and   either $s_{\xi}(t_{0}(\xi))=L_{\bar{\bar{\al}}}(t_{0}(\xi))$ or $s_{\xi}(t_{0}(\xi))=0$ and  $L_{\bar{\bar{\al}}}(t_{0}(\xi))\leq 0$.
		\descitem{(2)}{2} $t\mapsto s_{\xi}(t)$ is a non decreasing convex function with 
		\begin{equation*}
		\frac{ds_{\xi}}{dt}=\frac{f(\bar{\alpha})-f\left((f')^{-1}\left(\frac{s_{\xi}(t)-\xi}{t}\right)\right)}{\bar{\alpha}-(f')^{-1}\left(\frac{s_{\xi}(t)-\xi}{t}\right)}.
		\end{equation*}
		\descitem  {(3)}{3} For $0\le \xi\le\eta\le\xi_{1}$, $s_{\xi}(t)\le s_{\eta}(t)$ if $t\in[t_{0}(\eta),T]$. Hence $s_\xi$ is unique. 
		\descitem {(4)}{4} $\xi\mapsto t_{0}(\xi)$ is continuous.
		\descitem {(5)}{5} For $(x,t)\in \Omega(\xi)$, let
				\begin{eqnarray}
	w_{\xi}(x,t)=\left\{\begin{array}{lll} 
		\bar{\alpha} &\mbox{ if }& x>s_{\xi}(t),\\
		(f')^{-1}\left(\frac{x-\xi}{t}\right) &\mbox{ if  }&x<s_{\xi}(t).
		\end{array}\right.
		\end{eqnarray}
Then $w_{\xi}$ is an entropy solution of 
\begin{equation*}
u_{t}+f(u)_{x}=0 \mbox{ in } \Omega(\xi).
\end{equation*}
		\descitem {(6)}{6} There exists $(\xi_{0}, \tau_{0})$ such that $\xi_{0}=-f'(\bar{\bar{\alpha}})\tau_{0}$, $t_{0}(\xi_0)=\tau_{0}$ and $s_{\xi_0}(\tau_0)=0$.
		
\descitem {(7)}{7} For $(R_1,T)\in D$, denote $\xi_0=\xi_0(R_1,T)$, $\tau_0=\tau_0(R_1,T)$, $s_{\xi_0}=s_{\xi_0(R_1,T)}$ be as in \descref{6}{(6)}. Then $(R_1,T)\rr (\xi_0(R_1,T), \tau_0(R_1,T), s_{\xi_{0}(R_1,T)})$ is continuous with 
\begin{enumerate}[(i)]
\item $\tau_0(f'(\bar{\al})T,T)=\xi_0(f'(\bar{\al})T,T)=0$. \\
$s_{\xi_0(f'(\bar{\al})T,T)}(t)=f'(\bar{\al})t$ for $0\leq t \leq T.$
\item $\tau_0(0,T)=T, \xi_0(0,T)=-f'(\bar{\bar{\al}})\tau_0(0,T)$, $s_{\xi_0(0,T)}\equiv 0.$
\end{enumerate}

	\end{description}
	\end{lemma}
\begin{figure}
	\centering
	\begin{tikzpicture}
	\draw (-3,0)--(7,0);
	\draw (-3,6)--(7,6);
	\draw (0,0)--(0,6);
	% %left hand side
	%      \draw (0,2)--(-3,4);
	
	%right hand side
	%\draw[red](1.5,0)--(0,1.7);
	\draw[red] (0,4)--(4.5,6);
	\draw[red](0,4)--(0,1.7);
	\draw (1.5,0)--(-2.5,4.55);
	\draw (1.5,0)--(4.5,6);
	% \draw (1.5,0)--(2.1,1.3);
	\draw (1.5,0)--(1.9,1.3);
	\draw (1.5,0)--(1.6,1.3);
	\draw (1.5,0)--(1.3,1.3);
	%\draw (1.5,0)--()
	\draw (1.5,0)--(1,1.3);
	% \draw (1.5,0)--(.95,1);
	\draw (1.5,0)--(.7,1.3);
	% \draw (1.5,0)--(.4,1.3);
	\draw (1.9,1.9)node{\tiny $\bullet$};
	\draw (0,1.7) .. controls (.3,4) and (3,5) .. (4.5,6);
	\draw (1.9,1.9)[dashed] .. controls (2.5,4) and (3,5) .. (4.5,6);
	\draw [dashed](3.529,0)--(0,4);
	% \fill[white!40!white,rotate around={30:(-1,0.5)}] (-1,.6) rectangle (-0.2,3);
	% \fill[white!40!white,rotate around={90:(.1,.1)}] (.1,.5) rectangle (1,2.5);
	\coordinate (P) at (1.3,2);
	\draw (P) node[rotate=-47] (N){\tiny $(f')^{-1}\left(\frac{x-\xi_{0}}{t}\right)$};
	\draw[red](1.5,0)--(0,1.7);
	%numbering
	%x-axis
	\draw (0,-.2)node{\scriptsize $x=0$};
	\draw (0,0)node{\tiny $\bullet$};
	\draw (1.5,-.2)node{\scriptsize $(\xi_{0}, 0)$};
	\draw (1.5,0)node{\tiny $\bullet$};
	\draw (3.529,-.2)node{\scriptsize $(\xi_{1}, 0)$};
	\draw (3.529,0)node{\tiny $\bullet$};
	% \draw (6,0) node{\tiny $\bullet$};
	\draw (1.5,0) node{\tiny $\bullet$};
	\draw (0,0) node{\tiny $\bullet$};
	%t-axis
	\draw (-.45,1.6)node{\scriptsize $(0, \tau_{0})$};
	\draw (0,1.7)node{\tiny $\bullet$};
	%\draw (-.2,2)node{\scriptsize $\tau_{0}^{+}$};
	\draw (-.5,4)node{\scriptsize $(0, T_{1})$};
	\draw (0,4) node{\tiny $\bullet$};
	% %left hand side
	\coordinate (P) at (-1.5,3.15);
	\draw (P) node[rotate=-47] (N) {\tiny $f'(\bar{\bar{\alpha}})(t-\tau_{0})$};
	\draw (-.7,3)node{\scriptsize $\Omega(\xi)$};
	%right hand side
	\draw (4.5,6.25)node{\scriptsize $(R_{1}, T)$};
	% \draw (3.8,1)node{\scriptsize $\gamma_{1}$};
	% \draw (2.3,1)node{\scriptsize $\gamma_{2}$};
	% \draw (.8,.65)node{\scriptsize $\gamma_{3}$};
	\draw (1,4)node{\scriptsize $\bar{\alpha}$};
	\draw (1,5.4)node{\scriptsize $\bar{\alpha}$};
	\draw (.3,3.3)node{\scriptsize $\bar{\alpha}$};
	\draw[<-](2.1,4.5)--(4,4.5);
	\draw (4.4, 4.5)node{\scriptsize $s_{\xi_{0}}(t)$};
	\draw[<-](2.55,3.7)--(4,3.7);
	\draw (4.4, 3.7)node{\scriptsize $s_{\xi_{1}}(t)$};
	\draw (4.5,6) node{\tiny $\bullet$};
	\coordinate (P) at (3,1);
	\draw (P) node[rotate=-45] (N) {\tiny $f'(\bar{\bar{\alpha}})(t-T_{1})$};
	\coordinate (P) at (1.5,4.9);
	\draw (P) node[rotate=20] (N) {\tiny $R_{1}+f'(\bar{\alpha})(t-T)$};
	\end{tikzpicture}
	\caption{This figure illustrates construction of the curve $s_\xi$.}\label{fig-13}
\end{figure}
\begin{proof}
	If $T_{1}=0$, then $\xi_{1}=0$ and take $s_\xi(t)=L_{\bar{\al}}(t)$. Let $T_{1}>0$, $0\le\xi\le\xi_{1}$, $\beta$ and $\tau$ be as in \descref{iv}{(iv)} and \descref{v}{(v)}. Let $k\ge 1$ and $\bar{\bar{\alpha}}=u_{0}<u_{1}<\cdots<u_{k}=\beta$ be a discretization of $[\bar{\bar{\alpha}}, \beta]$ satisfying \descref{vi}{(vi)}. Let $l_{i}$ and $v_{k}$ be defined as in \descref{vii}{(vii)} and \descref{viii}{(viii)}. Define $t_{p}<t_{1}<\cdots<t_{k}=T$ and $s_{k}$ inductively by 
	\begin{description}[font=\normalfont]
		\descitem {($a_{1}$)}{a1} $s_{k}(t_{k})=s_{k}(T)=R_1$.
		\descitem {($a_{2}$)}{a2} $s_{k}$ is linear in $[t_{i-1}, t_{i}]$ and for $t\in (t_{i-1},t_{i})$,
		 \begin{eqnarray*}
		 %\left\{\begin{array}{lll} 
		 	\frac{ds_{k}}{dt}&=&\frac{f(\bar{\alpha})-f\left(u_{i}\right)}{\bar{\alpha}-u_{i}},\\
		s_{k}(t_{i-1})&=& l_{i-1}(t_{i-1}).
		 %\end{array}\right.
		 \end{eqnarray*}
		 \descitem {($a_{3}$)}{a3} Either $s_{k}(t_{p})=0$ or if $s_{k}(t_{p})>0$, then $p=0$ and  $l_{0}(t_{p})=s_{k}(t_{p})$. 
	\end{description}
 From the convexity, we prove this by induction on $i$. As $\bar{\bar{\alpha}}\le \beta<\bar{\alpha}$, by the convexity of $f$, we have
 \begin{equation*}
 f'(\bar{\bar{\alpha}})\le f'(\beta)<\frac{f(\bar{\alpha})-f\left(\beta\right)}{\bar{\alpha}-\beta}<f'(\bar{\alpha}).
 \end{equation*}
 Hence integrating from $\T$ to $T$ to obtain
 \begin{eqnarray*}
 R_{1}+f'(\bar{\alpha})(\theta-T)&\le& R_{1}+\frac{f(\bar{\alpha})-f\left(\beta\right)}{\bar{\alpha}-\beta}(\theta-T),\\
 &\le&R_{1}+f'(\beta)(\theta-T).
 \end{eqnarray*}
 Now choose $t_{k-1}$ by 
 \begin{eqnarray*}
 R_{1}+\frac{f(\bar{\alpha})-f\left(\beta\right)}{\bar{\alpha}-\beta}(t_{k-1}-T)&=&s_{k}(t_{k-1})=l_{k-1}(t_{k-1})
 =\xi+f'(u_{k-1})t_{k-1}.
 \end{eqnarray*}
 Since $R_{1}-\xi=Tf'(\beta)$ and $u_{k-1}<u_{k}=\beta$ to obtain
 \begin{equation*}
 t_{k-1}=T\frac{\left(\frac{f(\bar{\alpha})-f\left(u_{k}\right)}{\bar{\alpha}-u_{k}}-f'(\beta)\right)}{\left(\frac{f(\bar{\alpha})-f\left(u_{k}\right)}{\bar{\alpha}-u_{k}}-f'(u_{k-1})\right)}>0
 \end{equation*}
 and $t_{k-1}<T$ since $u_{k-1}<u_{k}$. For $t\in (t_{k-1}, t_k)$, define 
 $s_k(t)=R_1+\frac{f(\bar{\al})-f(u_k)}{\bar{\al}-u_k}(t-T).$
  Now by induction on $i$ \descref{a1}{($a_{1}$)}, \descref{a2}{$(a_{2})$} and \descref{a3}{$(a_{3})$} holds.
 
 Let $v_{k}(x,t)$ be as in \descref{viii}{(viii)}, then for $t\notin\{t_{p}, t_{1}, \cdots, t_{k}\}$,
 \begin{equation*}
 	\frac{ds_{k}}{dt}=\frac{f(\bar{\alpha})-f\left(v_{k}(s_{k}(t),t)\right)}{\bar{\alpha}-v_{k}(s_{k}(t),t)}.
 \end{equation*}
 Hence
 \begin{equation*}
 s_{k}(t)=s_{k}(T)+\int_{t}^{T}\left(\frac{f(\bar{\alpha})-f\left(v_{k}(s_{k}(\T),\T)\right)}{\bar{\alpha}-v_{k}(s_{k}(\T),\T)}\right)d\theta.
 \end{equation*}
 From Arzela-Ascoli, we can find a subsequence still denoted by $\{s_{\xi_k}\}$ such that 
 $s_{\xi_{k}}\rr s_{\xi}$ uniformly and $t_{p}(\xi_{k})\rr t_{0}(\xi)$ as $k\rr\f$.
 Since $s_k$ is convex for each $k$, hence $s_\xi$ is convex and satisfies \descref{1}{(1)} and \descref{2}{(2)}. 
 
	Let $0\le\xi\le\eta\le\xi_{1}$. Since $s_{\xi}(T)=s_{\eta}(T)=R_{1}$, hence if \descref{3}{(3)} is not true, then there exists $a<b\le T$ such that 
		\begin{eqnarray*}
	\left\{\begin{array}{lll} 
	s_{\eta}(t)&<&s_{\xi}(t), \mbox{ for all } t\in (a,b),\\
	s_{\eta}(b)&=&s_{\xi}(b).
%f''(\bar{\bar{\theta}}_{g})(\theta-T_{1}) &\mbox{ if }& 0\le \theta\le t_{1},
	\end{array}\right.
	\end{eqnarray*}
	Now $\xi\le\eta$ and hence $-\eta\le-\xi$ and for $ t\in (a,b),$
	\begin{equation*}
	(f')^{-1}\left(\frac{s_{\eta}(t)-\eta}{t}\right)<(f')^{-1}\left(\frac{s_{\xi}(t)-\xi}{t}\right)\le \bar{\alpha}.
	\end{equation*}
	Therefore by the convexity of $f$, we have for $t\in(a,b)$,
	\begin{equation*}
	\frac{ds_{\eta}(t)}{dt}<\frac{ds_{\xi}(t)}{dt}.
	\end{equation*}
	Integrating from $t$ to $b$ in $(a,b)$ to obtain
	\begin{equation*}
	s_{\eta}(b)-s_{\eta}(t)<s_{\xi}(b)-s_{\xi}(t)
	\end{equation*}
		and hence $s_{\xi}(t)<s_{\eta}(t)<s_{\xi}(t)$ which is a contradiction. Thus, $s_{\xi}(t)\le s_{\eta}(t)$ for all $t\in[t_{0}(\xi),T]$. This also shows that $s_{\xi}$ satisfying \descref{a1}{$(a_{1})$} and \descref{a2}{$(a_{2})$} is unique. 
 This proves \descref{3}{(3)}. From the uniqueness of $s_\xi$, $\xi\mapsto t_0(\xi)$ is continuous and hence \descref{4}{(4)} follows.   From Rankine-Hugoniot condition, $w_{\xi}$ is an entropy solution in $\Omega(\xi)$. This proves \descref{5}{(5)}.

Let $h(\xi)=-\frac{\xi}{f'(\bar{\bar{\al}})}-t_0(\xi)$, then we have $h(0)=-t_0(0)\leq 0$ and $h(\xi_1)=-\frac{\xi_1}{f'(\bar{\bar{\al}})}-t_0(\xi_1)=T_1-t_0(\xi_1)\geq 0$. Therefore there exists $(\xi_0,\tau_0)$ with $t_0(\xi_0)=\tau_0(\xi_0)=\tau_0$ such that $L_{\bar{\bar{\al}}}(\tau_0)=0$ and $\xi_0=-f'(\bar{\bar{\al}})\tau_0$. This proves \descref{6}{(6)}. From the uniqueness of $s_{\xi_0(R_1,T)}$, it follows that $(R_1,T)\rr (\xi_0(R_1,T), \tau_0(R_1,T), s_{\xi_0(R_1,T)})$ is continuous in $D$. 
Suppose $\tau_0=\tau_0(0,T)<T$, then $s_{\xi_0(0,T)}(\tau_0)=s_{\xi_0(0,T)}(T)=0$ and $s_{\xi_0(0,T)}$ is convex, hence $s_{\xi_0(0,T)}\equiv 0$. Integrating $\frac{d s_{\xi}}{dt}$ from $\tau_0$ to $T$ to obtain with $\xi_0(0,T)=\xi_0$, we have 
$$0=s_{\xi_0}(T)-s_{\xi_0}(\tau_0)=\int\limits_{\tau_0}^T\frac{f(\bar{\al})-f\left((f')^{-1}\left(-\frac{\xi_0}{t}\right)\right)}{\bar{\al}-(f')^{-1}\left(-\frac{\xi_0}{t}\right)}dt\neq 0,$$
since $(f')^{-1}\left(-\frac{\xi_0}{t}\right)\leq (f')^{-1}(0)\leq \bar{\al}$ and $f$ is convex. This is a contradiction and we get $\tau_0=T$.
\par If $(R_1,T)=(f'(\bar{\al})T,T)$, then $T_1=0$ and hence $\tau_0(R_1,T)=0$, $\xi_0(R_1,T)=0$. Therefore by uniqueness, $s_{\xi_0(R_1,T)}(t)=f'(\bar{\al})t$ is the given solution. This proves \descref{7}{(7)} and hence the lemma.
%
  %
  % Let $f(\xi)=\tau(\xi)-t_{0}(\xi)$, then from \descref{4}{(4)}, $f$ is a $f(\xi_{1})>0$. Also $t_{0}(\xi)\le0$ and $\tau(\xi)\rr 0$ as $\xi\rr0$. Hence $f(0)\le0$ and therefore there exists a $(\xi_{0}, \tau_{0})$ such that $t_{0}(\xi_{0})=\tau_{0}=\tau_{0}(\xi_{0})$. This proves \descref{6}{(6)} and hence the lemma.
\end{proof}
%
%
%(Note:**Change the label of steps properly**)
\begin{lemma}\label{lemma3.6}
Let $u_0\in L^\f(\R)$ and $u$ be the corresponding solution of (\ref{conlaw-equation}). Let $0\leq T_1 \leq T$ be such that $f'(\bar{\T}_g)=\frac{R_1(T)}{T-T_1}.$ Let $\bar{\al}=\bar{\T_g}$ and $(\xi_0,\tau_0,s_{\xi_0})$ be as in lemma \ref{3.6} at $(R_1(T),T)$. Then 
$$\xi_0\leq y(R_1(T)+,T).$$
\end{lemma}
\begin{proof}
First assume that $R_1(T)>0$. Then $f'(\bar{\T}_g)>0.$ Suppose $T_1=0$, then from lemma \ref{3.6}, $\xi_0=0$ and hence the lemma is true. Therefore we assume that $R_1(T)>0, T_1>0$ and $y(R_1(T)+,T)<\xi_0$. 

\begin{description}[font=\normalfont]
	\descitem{Step-1:}{3.6-step-1}
 For a.e. $x\in (0,R_1(t)),$ and $t\in(0,T]$, $u(x,t)\geq \bar{\T}_g$.\\ Suppose $R_2(t)<x<R_1(t)$, then from (iii) and (ix) of theorem \ref{theorem3.1}, we have for a.e. $x$, 
\begin{eqnarray*}
f(u(x,t))&=&g(\T_g),\\
f'(u(x,t))&=&\frac{x}{t-t_+(x,t)}\geq 0.
\end{eqnarray*}
Hence $u(x,t)=\bar{\T}_g$.  Suppose $0<x<R_2(t)$, let $\ga\in ch(x,t)$ such that $\ga=(\ga_1,\phi, \ga_3)$, $\dot{\ga}=(p_1,\phi, p_3)$ and $p_1\geq 0$, $p_3\geq 0$. From (\ref{eq4.6}) we have $g(g^{*{'}}(p_3))=f(f^{*{'}}(p_1)).$ Therefore $f^{*{'}}(p_1)\geq \bar{\T}_g$. Since for a.e. $x\in (0,R_2(T))$, $p_1=\frac{x}{t-t_+(x,t)},$ hence from \descref{ix.}{(ix)} of theorem \ref{theorem3.1}, we have $u(x,t)=(f^*)'\left(\frac{x}{t-t_+(x,t)}\right)\geq \bar{\T}_g$. This proves \descref{3.6-step-1}{Step-1}. 

\descitem{Step-2:}{3.6-step-2} For all $t\in [\tau_0,T], R_1(t)\leq s_{\xi_0}(t).$\\
Suppose not, since $R_1(T)=s_{\xi_0}(T)$, there exist $a<b$ such that for $t\in (a,b)$, we have 
$$s_{\xi_0}(t)<R_1(t),\ s_{\xi_0}(b)=R_1(b).$$
From the non-intersecting of characteristics, it follows that for $t<T$, $y(R_1(t)+,t)\leq y(R_1(T)+,T)<\xi_0$. Hence for $t\in (a,b)$, we have 
$$\frac{R_1(t)-y(R_1(t)+,t)}{t}>\frac{s_{\xi_0}(t)-\xi_0}{t}.$$
%Also from lemma**, for $0<x<R_1(t),$ for $f'(p)=\frac{x}{t-t_+(x,t)}\geq 0$, there exist $q$ such that $f(p)=g(q)$, $g'(q)\geq 0$. Hence $p\geq \bar{\T}.$ Hence by explicit formula we have $u(x,t)=(f')^{-1}\left(\frac{x}{t-t_+(x,t)}\right)\geq \bar{\T}_g.$ Therefore $u(R_1(t)-,t)\geq \bar{\T}_g$. 
From \descref{3.6-step-1}{Step-1} and  convexity of $f$, we have for a.e., $t\in (a,b)$, 
\begin{eqnarray*}
\frac{d R_1}{d t}&=& \frac{f(u(R_1(t)-,t))-f\left((f')^{-1}\left(\frac{R_1(t)-y(R_1(t)+,t)}{t}\right)\right)}{u(R_1(t)-,t)-(f')^{-1}\left(\frac{R_1(t)-y(R_1(t)+,t)}{t}\right)}\\
&\geq & \frac{f(\bar{\T}_g)-f\left((f')^{-1}\left(\frac{s_{\xi_0}(t)-\xi_0}{t}\right)\right)}{\bar{\T}_g-(f')^{-1}\left(\frac{s_{\xi_0}(t)-\xi_0}{t}\right)}\\
&=& \frac{d s_{\xi_0}}{dt}. 
\end{eqnarray*}
Integrating from $t$ to $b$ to obtain 
$$R_1(b)-R_1(t)\geq s_{\xi_0}(b)-s_{\xi_0}(t).$$
Hence for $t\in (a,b)$, $s_{\xi_0}(t)\geq R_1(t)>s_{\xi_0}(t)$ which is a contradiction. This proves  \descref{3.6-step-2}{Step-2}. From \descref{3.6-step-2}{Step-2}, we have 
$$R_1(\tau_0)=0, y(0+,\tau_0)\leq y(R_1(T)+, T)<\xi_0.$$
\descitem{Step-3:}{3.6-step-3} There exists an $\e_0>0$ such that for all $t\in (\tau_0-\e_0,\tau_0)$, $R_1(t)=0$.\\
 Since $\xi_0(0,\tau_0)=\xi_0$, $\tau_0(0,\tau_0)=\tau_0$, hence by continuity, there exist an $\e_1>0$ such that $\forall (y,t)\in \Omega=\{(\xi,s):\ \xi\geq 0, s\geq 0\} \cap B((0,\tau_0),\e_1)$, we have 
\begin{equation}\label{star1}
\frac{\xi_0+y(R_1(T)+,T)}{2}\leq \xi_0(y,t).
\end{equation}
Suppose \descref{3.6-step-3}{Step-3} is not true. Then there exists a $\tau_0-\e_1<\ti{t}<\tau_0$ such that $(R_1(t),t)\in \Omega$ for $\ti{t}\leq t \leq \tau_0$ and $R_1(\ti{t})>0.$ Choose $\ti{t}<\ti{t}_1<\tau_0$ such that for $t\in (\ti{t},\ti{t}_1)$, $R_1(t)>0$ and $R_1(\ti{t}_1)=0$. Note that $\ti{t}_1$ exist because $R_1(\tau_0)=0$. Let $(\xi_0(t), \tau_0(t), s_{\xi_0(t)})$ be as in lemma \ref{3.6} starting at $(R_1(t),t)$ for $t\in (\ti{t}, \ti{t}_1)$. From (\ref{star1}) and \descref{3.6-step-2}{Step-2} we obtain $R_1(\tau_0(t))=0.$ From \descref{7}{(7)} of lemma \ref{3.6}, we have $\tau_0(t)\rr \ti{t}_1$ as $t\rr \ti{t}_1$ and $t\mapsto \tau_0(t)$ is continuous. Since $\tau_0(\ti{t_1})<\ti{t}_1$, hence by continuity, there exists $t_2\in (\ti{t}_1, t_1)$ such that $\tau_0(t_2)=\ti{t}_1$. Therefore $0=R_1(\tau_0(t_2))=R_1(\ti{t}_1)>0$ which is a contradiction. Hence \descref{3.6-step-3}{Step-3} holds. 

\descitem{Step-4:}{3.6-step-4} From R-H condition, we have for a.e., $t\in (\tau_0-\epsilon_0,\tau_0)$, $f(u(0+,t))=g(u(0-,t)).$ Since $R_1(t)=0$, hence $f'(u(0+,t))\leq 0.$ Since $f(\T_f)\leq g(\T_g)$, hence $L_1(t)=0$, therefore  $g'(u(0-,t))\linebreak\geq 0$, hence $u(0+,t)\leq \bar{\bar{\T}}_g$.
Therefore $f'(u(0+,t))\leq f'(\bar{\bar{\T}}_g)$. Letting $t\rr \tau_0$ to obtain 
$-\frac{y(0+,\tau_0+)}{\tau_0}\leq \lim\limits_{t \uparrow \tau_0}$
$f'(u(0+,t))\leq f'(\bar{\bar{\T}}_g).$
 This implies that $y(0+,\tau_0+)\geq -\tau_0 f'(\bar{\bar{\T}}_g)=\xi_0$. But from the hypothesis we have $y(0+, \tau_0+)\leq y(R_1(T)+,T)<\xi_0$, which is a contradiction. This proves the lemma if $R_1(T)>0$.

\descitem{Step-5:}{3.6-step-5} If $R_1(T)=0$, repeat \descref{3.6-step-3}{Step-3}, \descref{3.6-step-4}{Step-4} to obtain a contradiction if $y(0+,T)<\xi_0$. Hence the lemma. 
\end{description}
\end{proof}

\subsection{Solution with reflected  characteristics}\label{newsection}

Earlier we build two solutions via backward construction, namely one has shock and other is a continuous solution. Now we need to construct another solution by backward construction for  the reflected case and is as follows:\\
\noindent Let $(T,R_1,R_2, y(\cdot))\in \mathcal{R}(T)$. Assume that there are constants $y_-\leq 0 \leq y_+$ such that 
\begin{eqnarray*}
y(x)=\left\{\begin{array}{llll}
y_- &\mbox{if}& x\in (-\f, R_2),\\
y_+ &\mbox{if}& x\in (R_1,\f).
\end{array}\right.
\end{eqnarray*}
Since $(T,R_1, R_2,y(\cdot))\in \mathcal{R}(T)$, we have to consider three cases. In each case we construct a $u_{1,0}\in L^\f(\R)$ and the description of the corresponding solution $u$ such that for $i=1,2$, $R_i(T)=R_i$, $y(\cdot, T)=y(\cdot).$\\

\begin{description}[font=\normalfont]
	\descitem{Case 1:}{case-1} (see figure \ref{fig-6} for illustration)
	 Let $0\leq R_2 \leq R_1$ and assume that there exist $0\leq T_1 \leq T_2 \leq T$ such that $f'(\bar{\T}_g)=\frac{R_1}{T-T_1}=\frac{R_2}{T-T_2}$ and $(\tau_0,\xi_0, s_{\xi_0})$ be as in lemma \ref{3.6} for $(R_1,T)$. \\
	 \begin{figure}
	 	\centering
	 	\begin{tikzpicture}
	 	\draw (-4,0)--(7,0);
	 	\draw (-4,6)--(7,6);
	 	\draw (0,0)--(0,6);
	 	%left hand side
	 	%\draw (-.3,0)--(-.3,4);
	 	\draw (-.6,0)--(-.6,3.15);
	 	\draw (-.89,0)--(-.89,2.5);
	 	\draw (-1.2,0)--(-1.2,1.78);
	 	\draw (-1.5,0)--(-1.5,1.15);
	 	\draw (-1.8,0)--(-1.8,.45);
	 	\draw (-3,0)--(0,5.3);
	 	\draw[dashed](-2,0)--(0,4.5);
	 	%right hand side
	 	\draw[dashed](0,4.5)--(3.35,6);
	 	\draw (0,3.5)--(3.35,6);
	 	\draw (0,5.3)--(3.35,6);
	 	\draw[red] (0,2.5)--(4.5,6);
	 	%\draw[red](0,2.5)--(0,.15);
			 	\draw[red](0,2.5)--(0,1);
	 	\draw[red] (1,0)--(0,1);
	 	%\draw [red](.7,0)--(0,.7);
	 	%\draw [red](.4,0)--(0,.4);
	 	%\draw [red](.15,0)--(0,.15);
	 	\draw (1,0)--(4.5,6);
	 	\draw (1,0)--(3.6,5.15);
	 	\draw (1,0)--(3,4.6);
	 	\draw (1,0)--(2.5,4.15);
	 	\draw (1,0)--(2,3.7);
	 	\draw (1,0)--(1.5,3.2);
	 	\draw (1,0)--(.95,2.6);
	 	\draw (1,0)--(.5,2.05);
	 	\draw (1,0)--(.2,1.55);
	 	\draw (0,1) .. controls (.3,2.5) and (3,4.5) .. (4.5,6);
	 	\draw[dashed](2.5,0)--(4.5,6);
	 	\draw (4,0)--(4.5,6);
	 	%numbering
	 	%x-axis
	 	\draw (0,-.2)node{\scriptsize $x=0$};
	 	\draw (-2,1)node{\scriptsize$w_{-}$};
	 	\draw (-.7,.7)node{\scriptsize$\theta_{g}$};
			 	\draw (.3,.37)node{\scriptsize$\bar{\bar{\theta}}_{g}$};
	 	\draw (-2,-.2)node{\scriptsize $\xi_{2}$};
	 	\draw (-3,-.2)node{\scriptsize $y_{-}$};
	 	\draw (1,-.2)node{\scriptsize $\xi_{0}$};
	 	\draw (2.5,-.2)node{\scriptsize $\xi_{1}$};
	 	\draw (4,-.2)node{\scriptsize $y_{+}$};
	 	\draw (4,0) node{\tiny $\bullet$};
	 	\draw (2.5,0) node{\tiny $\bullet$};
	 	\draw (1,0) node{\tiny $\bullet$};
	 	\draw (-3,0) node{\tiny $\bullet$};
	 	\draw (-2,0) node{\tiny $\bullet$};
	 	\draw (3.35,6) node{\tiny $\bullet$};
	 	\draw (4.5,6) node{\tiny $\bullet$};
	 	\draw (0,0) node{\tiny $\bullet$};
	 	\draw [<-](1,.7)--(5,.7);
	 	\draw (6.2,.7)node{\tiny $(f')^{-1}\left(\frac{x-\xi_{0}}{t}\right)$};
			 	%\draw (3.2,.7)node{\tiny $(f')^{-1}\left(\frac{x-\xi_{0}}{t}\right)$};
	 	%y-axis
	 	\draw (-.2,1)node{\scriptsize $\tau_{0}$};
	 	\draw (-.2,2.5)node{\scriptsize $T_{1}$};
	 	\draw (-.2,3.5)node{\scriptsize $T_{2}$};
	 	\draw (-.15,5.35)node{\scriptsize $t_{0}$};
	 	\draw (-.15,4.5)node{\scriptsize $t_{1}$};
	 	%left hand side
	 	\draw (-1.3,2)node{\scriptsize $\eta_{1}$};
	 	\draw (-2.6,1)node{\scriptsize $\eta_{2}$};
	 	%right hand side
	 	\draw (4.6,5)node{\scriptsize $\gamma_{1}$};
	 	%  \draw (4.7,2)node{\scriptsize $f'(u_{+})$};
	 	\draw (3.1,1)node{\scriptsize $\gamma_{2}$};
			 	\draw (6.1,3)node{\scriptsize $u_+$};
					\draw (3.8,3)node{\scriptsize $u_+$};
										\draw (3.6,4)node{\scriptsize $\beta_0$};
	 	\draw (2,1)node{\scriptsize $\gamma_{3}$};
	 	% \draw (2.6,2)node{\scriptsize $f'(\beta_{0})$};
	 	\draw (.5,.65)node{\scriptsize $\gamma_{4}$};
	 	\draw (.5,5.6)node{\scriptsize $\gamma_{5}$};
	 	\draw (.5,4.9)node{\scriptsize $\gamma_{6}$};
	 	\draw (1,5.2)node{\scriptsize$\bar{w}_{-}$};
			 	\draw (1,5.8)node{\scriptsize$\bar{w}_{-}$};
		\draw (-1,5.8)node{\scriptsize$w_{-}$};
\draw (1,4.6)node{\scriptsize$\bar{\theta}_{g}$};
	 	\draw (1,3.75)node{\scriptsize$\bar{\theta}_{g}$};
			 	\draw (1,2.95)node{\scriptsize$\bar{\theta}_{g}$};
	 	\draw [<-](2.2,3.8)--(5,3.8);
	 	\draw (5.3,3.8)node{\scriptsize $s_{\xi_{0}}$};
	 	\draw (3.35,6.25)node{\scriptsize $(R_{2},T)$};
	 	\draw (4.5,6.25)node{\scriptsize $(R_{1},T)$};
	 	\coordinate (P) at (2.65,2.5);
	 	\draw (P) node[rotate=59] (N) {\tiny Speed of $\gamma_{3}$=$f'(\beta_{0})$};
	 	\coordinate (P) at (4.4,2.5);
	 	\draw (P) node[rotate=85] (N) {\tiny Speed of $\gamma_{1}$=$f'(u_{+})$};
	 	\end{tikzpicture}
	 	\caption{Solution with reflected characteristics as in \descref{case-1}{Case 1}.}\label{fig-6}
	 	%	\captionof{figure}{}
	 \end{figure}
	Since $(T_1, R_1, R_2, y(\cdot))\in\mathcal{R}(T)$ hence $y_+$ must satisfy 
$$\xi_0\leq y_+.$$
In this case define the following quantities: let $T_2\leq t_0 \leq T$ be the unique solution of $h_+\left(\frac{R_2}{T-t_0}\right)=-\frac{y_-}{t_0}$. Let 
\begin{eqnarray*}
g'(w_-)&=& -\frac{y_-}{t_0}, \ f'(\bar{w}_-)=\frac{R_2}{T-t_0},\\
\eta_2(t) &=& g'(w_-)(t-t_0),\ 0\leq t \leq t_0,\\
\ga_5(t) &=& f'(\bar{w}_-)(t-t_0),\ t_0\leq t \leq T,\\
\ga_6(t) &=& R_2 +\frac{f(\bar{w}_-)-f(\bar{\T}_g)}{\bar{w}_--\bar{\T}_g} (t-T),\\
\eta_1(t) &=& \frac{g(w_-)-g(\T_g)}{w_--\T_g} (t-t_1),
\end{eqnarray*}
where $t_1$ be such that $\ga_6(t_1)=0$. Then by the definition of $h_+$, $t_0$ and convexity of $f, g$, it follows easily that $\eta_2(0)=y_-, \ga_5(T)=R_2$, $T_2\leq t_1 \leq t_0 \leq T$, $y_-\leq \xi_2=\eta_1(0)\leq 0$. Define 
\begin{eqnarray*}
f'(u_+)&=& \frac{R_1-y_+}{T}, \ f'(\beta_0)=\frac{R_1-\xi_0}{T},\\
\ga_1(t) &=& R_1+f'(u_+) (t-T),\\
\ga_2(t) &=& R_1+ \frac{f(u_+)-f(\beta_0)}{u_+-\beta_0},\\
\ga_3(t) &=& R_1+ f'(\beta_0)(t-T),\\
\ga_4(t) &=& f'(\bar{\bar{\T}}_g) (t-\tau_0).
\end{eqnarray*}
Since $\xi_0\leq y_+$, hence from convexity of $f$, we have 
$$\xi_0\leq \xi_1 =\ga_2(0)\leq y_+.$$
In this case define the initial data $u_{1,0}$ by 
\begin{eqnarray*}
u_{1,0}=\left\{\begin{array}{lllllll} 
w_- &\mbox{if}& x<\xi_2,\\
\T_g &\mbox{if}& \xi_2<x<0,\\
\bar{\bar{\T}}_g &\mbox{if}& 0<x<\xi_0,\\
\beta_0 &\mbox{if}& \xi_0<x<\xi_1,\\
u_+ &\mbox{if}& x>\xi_1.\\
\end{array}\right.
\end{eqnarray*}
It is easy to verify that the solution $u_1(\cdot,\cdot)$ of (\ref{conlaw-equation}) with initial data $u_{1,0}$ is given by 
\begin{eqnarray*}
u_{1}(x,t)=\left\{\begin{array}{lllllll} 
w_- &\mbox{if}& x<\min\{\eta_1(t),0\},\\
\T_g &\mbox{if}& \min\{\eta_1(t),0\}<x<0,\\
\bar{w}_- &\mbox{if}& 0<x<\ga_6(t),\\
\bar{\T}_g &\mbox{if}& \max\{\ga_6(t),0\}<x<s_{\xi_0}(t),\\
\bar{\bar{\T}}_g &\mbox{if}& 0<x<\ga_4(t),\\
(f')^{-1}\left(\frac{x-\xi_0}{t}\right)&\mbox{if}&\left\{\begin{array}{rl}
&s_{\xi_0}(t)<x<\ga_3(t), t\in (\tau_0,T)\\
&\mbox{or } \max\{\ga_4(t),0\}<x<\ga_3(t),
\end{array}\right. \\ 
\beta_0 &\mbox{if}& \ga_3(t)<x<\ga_2(t),\\
u_+ &\mbox{if}& x>\ga_2(t).\\
\end{array}\right.
\end{eqnarray*}
In this case define the domain $D_1$ for $t\leq T$ by 
$$D_1=\{(x,t):\ \min\{\eta_2(t),0\}<x\leq 0\}\cup \{(x,t):\ \max\{\ga_5(t),0\}<x<\ga_1(t)\}.$$
Then $u$ satisfies
$$u_1(\eta_2(t)+,t)=w_-, \ u_1(\ga_5(t)-,t)=\bar{w}_-,\  u_1(\ga_1(t)-,t)=u_+.$$

\descitem{Case 2:}{case-2} Let $R_1=R_2>0$ and for all $t\in [0,T]$, $f'(\bar{\T}_g)<\frac{R_1}{T-t}.$\\
%%%%%%%%%

%%%%%%

 Let 
$0\leq t_0 \leq T$ be the unique solution of $h_+\left(\frac{R_1}{T-t_0}\right)=-\frac{y_-}{t_0}$. 
As in \descref{case-1}{Case 1}, define 
\begin{eqnarray*}
g'(w_-)&=& -\frac{y_-}{t_0}, f'(\bar{w}_-)=\frac{R_1}{T-t_0}, \ f'(u_+)=\frac{R_1-y_+}{T},\\
\eta_2(t)&=& g'(w_-)(t-t_0)\ \mbox{if}\ 0\leq t\leq t_0,\\
\ga_4(t)&=& f'(\bar{w}_-)(t-t_0)\ \mbox{if}\ t_0\leq t\leq T,\\
\ga_1(t) &=& R_1+ f'(u_+)(t-T),\\
\ga_2(t) &=& R_1 + f'(\bar{\T}_g) (t-T),\\
\ga_3(t) &=& R_1+ \frac{f(u_+)-f(\bar{w}_-)}{u_+-\bar{w}_-} (t-T),\\
\ga_6(t) &=& R_1+ \frac{f(u_+)-f(\bar{\T}_g)}{u_+-\bar{\T}_g} (t-T),\\
\ga_5(t) &=& R_1+ \frac{f(\bar{\T}_g)- f(\bar{w}_-)}{\bar{\T}_g-\bar{w}_-} (t-T).
\end{eqnarray*}
Now we have to consider four sub-cases:
\begin{description}[font=\normalfont]

\descitem{Subcase 1:}{subcase-1} (see figure \ref{fig-7} for illustration) $y_+=\ga_1(0)<\ga_2(0)$ and $\ga_3(0).$

\begin{figure}
	\centering
	\begin{tikzpicture}
	\draw (-4,0)--(5,0);
	\draw (-4,4)--(5,4);
	\draw (0,0)--(0,4);
	%\left hand side
	\draw (-2,0)--(0,3);
	%right hand side 
	\draw (0,3)--(3,4);
	\draw (1,0)--(3,4);
	\draw (2.5,0)--(3,4);
	\draw (4,0)--(3,4);
	%numbering
	%x-axis
	\draw (2.5,-.25)node{\scriptsize $y_{+}$};
	\draw (-1.9,-.25)node{\scriptsize $y_{-}$};
	%left hand side 
	\draw (-1.55,1)node{\scriptsize $\eta_{2}$};
	\draw (.55,2.2)node{\scriptsize $D_{1}$};
	\draw (2.5,2.2)node{\scriptsize $D_{1}$};
	\draw (-.2,2.2)node{\scriptsize $D_{1}$};
	\draw (-2,2)node{\scriptsize $\omega_{-}$};
	\draw (-.5,.5)node{\scriptsize $\omega_{-}$};
	\draw (2.5,0) node{\tiny $\bullet$};
	%right hand side 
	\draw (1,3.5)node{\scriptsize $\gamma_{4}$};
	\draw (3.2,4.3)node{\scriptsize $(R_1,T)$};
	\draw (1,1.5)node{\scriptsize $\bar{\omega}_{-}$};
	\draw (2.9,1.5)node{\scriptsize $\gamma_{1}$};
	\draw (1.99,1.5)node{\scriptsize $\gamma_{3}$};
	\draw (3.2,.6)node{\scriptsize $u_{+}$};
	\draw (3.9,1.5)node{\scriptsize $\gamma_{2}$};
	\draw (3,4) node{\tiny $\bullet$};
	\end{tikzpicture}
	\caption{The figure illustrate as in \descref{subcase-1}{Subcase 1}.}\label{fig-7}
\end{figure}
\noindent Clearly $\bar{w}_-\geq u_+$, then define the initial data 
$u_{1,0}$ and the solution $u_1$ of (\ref{conlaw-equation}) by 
\begin{eqnarray*}
u_{1,0}(x)=\left\{\begin{array}{lllll}
w_-&\mbox{if}& x<0,\\
\bar{w}_-&\mbox{if}& 0<x<\ga_3(0),\\
u_+ &\mbox{if}& x>\ga_3(0)
 \end{array}\right.
\end{eqnarray*}
and the solution $u_1$ is given by 
\begin{eqnarray*}
u_{1}(x,t)=\left\{\begin{array}{lllll}
w_-&\mbox{if}& x<0,\\
\bar{w}_-&\mbox{if}& 0<x<\ga_3(t),\\
u_+ &\mbox{if}& x>\ga_3(t).
 \end{array}\right.
\end{eqnarray*}
Define for $0\leq t \leq T$,
$$D_1=\{(x,t): \min\{\eta_2(t),0\}<x\leq 0\}\cup\{(x,t):\max\{\ga_4(t),0\}<x<\ga_1(t)\}.$$
Then $u_1$ satisfies 
$$u_1(\eta_2(t)+,t)=w_-, u_1(\ga_4(t)-,t)=\bar{w}_-, u_1(\ga_1(t)-,t)=u_+.$$
\descitem{Subcase 2:}{subcase-2} (see figure \ref{fig-8} for illustration) Let $y_+=\ga_1(0)<\ga_2(0)$ and $\ga_3(0)<0$.\\

\begin{figure}
	\centering
	\begin{tikzpicture}
	\draw (-4,0)--(5,0);
	\draw (-4,4.5)--(5,4.5);
	\draw (0,0)--(0,4.5);
	%left hand side
	\draw (-3,0)--(0,3.5);
	\draw (-1,0)--(0,2);
	%right hand side
	\draw (0,3.5)--(3.5,4.5);
	\draw (0,2)--(3.5,4.5);
	\draw (2,0)--(3.5,4.5);
	\draw (3,0)--(3.5,4.5);
	%numbering
	%x-axis
	\draw (-3,-.25)node{\scriptsize $y_{-}$};
	\draw (-1,-.25)node{\scriptsize $\xi_{2}$};
	\draw (2,-.25)node{\scriptsize $y_{+}$};
	\draw (2,0) node{\tiny $\bullet$};
	\draw (-1,0) node{\tiny $\bullet$};
	\draw (-3,0) node{\tiny $\bullet$};
	%left hand side
	\draw (-2.5,1)node{\scriptsize $\eta_{2}$};
	\draw (-.7,1)node{\scriptsize $\eta_{1}$};
	\draw (-.2,1)node{\scriptsize $D_{1}$};
	\draw (1.5,.5)node{\scriptsize $D_{1}$};
	\draw (-.2,2.6)node{\scriptsize $D_{1}$};
	\draw (.3,2.6)node{\scriptsize $D_{1}$};
	%	\draw (-1.7,2)node{\scriptsize $g'(\omega_{-})$};
	\coordinate (P) at (-1.5,1.5);
	\draw (P) node[rotate=50] (N) {\tiny Speed of $\eta_{2}$=$g'(\omega_{-})$};
	%y-axis
	\draw (-.15,3.6)node{\scriptsize $t_{0}$};
	%right hand side
	\draw (1,1)node{\scriptsize $u_{+}$};
	\draw (1,4)node{\scriptsize $\gamma_{4}$};
	\draw (1.5,3.3)node{\scriptsize $\gamma_{3}$};
	\draw (3.3,1)node{\scriptsize $\gamma_{2}$};
	\draw (2.55,1)node{\scriptsize $\gamma_{1}$};
	\draw (3.5,4.5) node{\tiny $\bullet$};
	%	\draw (2.2,2)node{\scriptsize $f'(u_{+})$};
	%	\draw (3.8,3)node{\scriptsize $f'(\bar{\theta}_{g})$};
	%	\draw (2,4.25)node{\scriptsize $f'(\bar{\omega}_{-})$};
	\draw (3.5,4.75)node{\scriptsize $(R_{1},T)$};
	\coordinate (P) at (2.6,2.5);
	\draw (P) node[rotate=73] (N) {\tiny Speed of $\gamma_{1}$=$f'(u_{+})$};
	\coordinate (P) at (3.5,2.5);
	\draw (P) node[rotate=83] (N) {\tiny Speed of $\gamma_{2}$=$f'(\bar{\theta}_{g})$};
	\coordinate (P) at (1.3,3.7);
	\draw (P) node[rotate=16] (N) {\tiny Speed of $\gamma_{4}$=$f'(\bar{\omega}_{-})$};
	\end{tikzpicture}
	\caption{The figure illustrate as in  \descref{subcase-2}{Subcase 2}.}\label{fig-8}
	%	\captionof{figure}{}
\end{figure}
 Let $0<t_1<t_0$ be such that $\ga_3(t_1)=0$. Since $\ga_1(0)<\ga_2(0),$ hence $\bar{w}_-\geq u_+\geq \bar{\T}_g$, therefore there exists a unique $w_-\geq \bar{u}_+\geq \T_g$ such that $f(u_+)=g(\bar{u}_+)$. Let $\eta_1(t)=\frac{g(w_-)-g(\bar{u}_+)}{w_--\bar{u}_+}(t-t_1)$, then by convexity of $g$, it follows that $y_-\leq \eta_1(0)=\xi_2\leq 0.$  Let 
\begin{eqnarray*}
u_{1,0}(x)=\left\{\begin{array}{lllll}
w_- &\mbox{if}& x<\xi_2,\\
\bar{u}_+ &\mbox{if}& \xi_2<x<0,\\
u_+ &\mbox{if}& x>0,
\end{array}\right.
\end{eqnarray*}
then the solution $u_1$ to (\ref{conlaw-equation}) with initial data $u_{1,0}$ is given by 

\begin{eqnarray*}
u_{1}(x,t)=\left\{\begin{array}{lllll}
w_- &\mbox{if}& x<\min\{\eta_1(t),0\},\\
\bar{u}_+ &\mbox{if}& \eta_1(t)<x<0,\\
\bar{w}_- &\mbox{if}& 0<x<\ga_3(t),\\
u_+ &\mbox{if}& x>\ga_3(t).
\end{array}\right.
\end{eqnarray*}
Let 
$$D_1=\{(x,t):\ \min\{\eta_2(t),0\}<x\leq 0\}\cup \{(x,t):\ \max\{0,\ga_4(t)\}<x<\ga_1(t)\},$$
then $u_1$ satisfies
$$u_1(\eta_2(t)+,t)=w_-,\ u_!(\ga_4(t)+,t)=\bar{w}_-,\ u_1(\ga_1(t)-,t)=u_+.$$
\descitem{Subcase 3:}{subcase-3}  $0\leq \ga_2(0) \leq  \ga_1(0)=y_+$, $\ga_5(0)\geq 0$. \\Let 
\begin{eqnarray*}
u_{1,0}=\left\{\begin{array}{lllll}
w_- &\mbox{if}& x<0,\\
\bar{w}_- &\mbox{if}& 0<x<\ga_5(0),\\
\bar{\T}_g &\mbox{if}& \ga_5(0)<x<\ga_6(0),\\
u_+ &\mbox{if}& x>\ga_6(0)
\end{array}\right.
\end{eqnarray*}
and the corresponding solution $u_1$ in $\R\times [0,T]$ is given by 
\begin{eqnarray*}
u_1(x,t)=\left\{\begin{array}{lllll}
w_- &\mbox{if}& x<0,\\
\bar{w}_- &\mbox{if}& 0<x<\ga_5(t),\\
\bar{\T}_g &\mbox{if}& \ga_5(t)<x<\ga_6(t),\\
u_+ &\mbox{if}& x>\ga_6(t).
\end{array}\right.
\end{eqnarray*}
Define for $0<t\leq T$,
$$D_1=\{(x,t):\ \min\{\eta_2(t),0\}<x\leq 0\}\cup\{(x,t):\ \max\{\ga_4(t),0\}<x<\ga_1(t)\},$$
then $u_1$ satisfies 
$$u_1(\eta_2(t)+,t)=w_-,\ u_1(\ga_4(t)+,t)=\bar{w}_-,\ u_1(\ga_1(t)-,t)=u_+.$$

%
%\noindent Subcase 3: $0\leq \ga_2(0) \leq  \ga_1(0)=y_+$, $\ga_5(0)\geq 0$.
%Define 
%\begin{eqnarray*}
%u_{1,0} =\left\{\begin{array}{lllll}
%w_- &\mbox{if}& x<0,\\
%\bar{w}_-&\mbox{if}& 0<x<\ga_5(0),\\
%\bar{\T}_g&\mbox{if}& \ga_5(0)<x<\ga_6(0),\\
%u_+ &\mbox{if}& x>\ga_6(0)
%\end{array}\right.
%\end{eqnarray*}
%and the corresponding solution of (\ref{conlaw-equation}) is given by 
%\begin{eqnarray*}
%u_{1}(x,t) =\left\{\begin{array}{lllll}
%w_- &\mbox{if}& x<0,\\
%\bar{w}_-&\mbox{if}& 0<x<\ga_5(t),\\
%\bar{\T}_g&\mbox{if}& \ga_5(t)<x<\ga_6(t),\\
%u_+ &\mbox{if}& x>\ga_6(t).
%\end{array}\right.
%\end{eqnarray*}
\descitem{Subcase 4:}{subcase-4} (see figure \ref{fig-9} for illustration) $0\leq \ga_2(0) \leq \ga_1(0)$, $\ga_5(0)<0$.\\

\begin{figure}
	\centering
	\begin{tikzpicture}
	\draw (-4,0)--(5,0);
	\draw (-4,4)--(5,4);
	\draw (0,0)--(0,4);
	%\left hand side
	\draw (-2,0)--(0,3);
	\draw[dashed](-.7,0)--(0,2);
	%right hand side 
	\draw[dashed] (0,2)--(3,4);
	\draw (0,3)--(3,4);
	\draw (1,0)--(3,4);
	\draw [dashed](2.5,0)--(3,4);
	\draw (4,0)--(3,4);
	%numbering
	%x-axis
	\draw (2.5,-.25)node{\scriptsize $\xi_{1}$};
	\draw (4,-.25)node{\scriptsize $y_{+}$};
	\draw (-.7,-.25)node{\scriptsize $\xi_{2}$};
	\draw (-2,-.25)node{\scriptsize $y_{-}$};
	\draw (4,0) node{\tiny $\bullet$};
	\draw (2.5,0) node{\tiny $\bullet$};
	\draw (-.7,0) node{\tiny $\bullet$};
	\draw (-2,0) node{\tiny $\bullet$};
	%y-axis
	\draw (-.2,3)node{\scriptsize $t_{0}$};
	\draw (-.15,2)node{\scriptsize $t_{1}$};
	%left hand side 
	\draw (-1.55,1)node{\scriptsize $\eta_{2}$};
	%	\draw (-1.5,2)node{\scriptsize $g'(\omega_{-})$};
	\draw (-.7,.5)node{\scriptsize $\eta_{1}$};
	\coordinate (P) at (-1.05,1.05);
	\draw (P) node[rotate=57] (N) {\tiny Speed of $\eta_{2}$=$g'(\omega_{-})$};
	%right hand side 
	\draw (1,3.5)node{\scriptsize $\gamma_{4}$};
	\draw (2.9,1.5)node{\scriptsize $\gamma_{3}$};
	\draw (4.1,.7)node{\scriptsize $\gamma_{1}$};
	\draw (2,1.5)node{\scriptsize $\gamma_{2}$};
	\draw (3.2,4.3)node{\scriptsize $(R_1,T)$};
	\draw (2,3)node{\scriptsize $\gamma_{5}$};
	\draw (3,4) node{\tiny $\bullet$};
	%	\draw (1,3)node{\scriptsize $f'(\bar{\omega}_{-})$};
	%	\draw (1,.9)node{\scriptsize $f'(\bar{\theta}_{g})$};
	%	\draw (3.7,3)node{\scriptsize $f'(u_{+})$};
	\coordinate (P) at (3.6,2.5);
	\draw (P) node[rotate=104] (N) {\tiny Speed of $\gamma_{1}$=$f'(u_{+})$};
	\coordinate (P) at (1.6,1.7);
	\draw (P) node[rotate=62] (N) {\tiny Speed of $\gamma_{2}$=$f'(\bar{\theta}_{g})$};
	\coordinate (P) at (1.2,3.2);
	\draw (P) node[rotate=20] (N) {\tiny Speed of $\gamma_{4}$=$f'(\bar{\omega}_{-})$};
	\end{tikzpicture}
	\caption{The figure illustrate as in  \descref{subcase-4}{Subcase 4}.}\label{fig-9}
	%	\captionof{figure}{}
\end{figure}

Let $t_1$ be such that $\ga_5(t_1)=0$. Let $\eta_1(t)=\frac{g(w_-)-g(\T_g)}{w_--\T_g}$ and 
\begin{eqnarray*}
u_{1,0}(x) =\left\{\begin{array}{lllll}
w_- &\mbox{if}& x<\eta_1(0),\\
\T_g &\mbox{if}& \eta_1(0)<x<0,\\
\bar{\T}_g&\mbox{if}& 0<x<\ga_6(0),\\
u_+ &\mbox{if}& x>\ga_6(0)
\end{array}\right.
\end{eqnarray*}then the  corresponding solution $u_1$ of (\ref{conlaw-equation}) is given by 
\begin{eqnarray*}
u_{1}(x,t) =\left\{\begin{array}{lllll}
w_- &\mbox{if}& x<\min\{\eta_1(t),0\},\\
\T_g &\mbox{if}&\min\{\eta_1(t),0\}<x<0,\\
\bar{w}_-&\mbox{if}& 0<x<\ga_5(t),\\
\bar{\T}_g &\mbox{if}& \max\{\ga_5(t),0\}<x<\ga_6(t),\\
u_+ &\mbox{if}& x>\ga_6(t).
\end{array}\right.
\end{eqnarray*}
Define 
$$D_1=\{(x,t):\ \min\{\eta_2(t),0\}<x\leq 0\}\cup \{(x,t):\ \max\{0,\ga_4(t)\}<x<\ga_1(t)\},$$
then $u_1$ satisfies 
$$u_1(\eta_2(t)+,t) =w_-,\ u_1(\ga_4(t)+,t)=\bar{w}_-,\ u_1(\ga_1(t)-,t)=u_+.$$
\end{description}
\descitem{Case 3:}{case-3} $R_1=0,$ $y_-\leq 0 \leq \xi_0=-f'(\bar{\bar{\T}}_g)T\leq y_+.$\\
Define 
\begin{eqnarray*}
g'(w_-)&=&-\frac{y_-}{T}, f'(u_+)=\frac{R_1- y_+}{T},\\
\eta_2(t) &=& g'(w_-)(t-T),\\
\eta_1(t) &=&\left(\frac{g(w_-)-g(\T_g)}{w_--\T_g}\right)(t-T),\\
\ga_3(t) &=& f'(\bar{\bar{\T}}_g) (t-T),\\
\ga_2(t) &=& R_1+ \frac{f(\bar{\bar{\T}}_g)-f(u_+)}{\bar{\bar{\T}}_g - u_+}(t-T),\\
\ga_1(t) &=& R_1+ f'(u_+)(t-T).
\end{eqnarray*}
Due to $\xi_0\leq y_+=\ga_1(0)$, we have $u_+\leq \bar{\bar{\T}}_g$. Hence by convexity of $f$, $\xi_0\leq \xi_1=\ga_2(0)\leq y_+.$ Since $w_-\geq \T_g$, hence $y_-\leq \xi_2=\eta_1(0)\leq 0.$
Define 
\begin{eqnarray*}
u_{1,0}(x)=\left\{\begin{array}{llll}
w_- &\mbox{if}& x<\xi_2,\\
\T_g &\mbox{if}& \xi_2<x<0,\\
\bar{\bar{\T}}_g &\mbox{if}& 0<x<\xi_1,\\
u_+ &\mbox{if}& x>\xi_1
\end{array}\right.
\end{eqnarray*}
and the corresponding solution $u_1$ is given by 
\begin{eqnarray*}
u_1(x,t)=\left\{\begin{array}{llll}
w_- &\mbox{if}& x<\min\{0,\eta_1(t)\},\\
\T_g &\mbox{if}& \min\{0,\eta_1(t)\}<x<0,\\
\bar{\bar{\T}}_g &\mbox{if}& 0<x<\ga_2(t),\\
u_+ &\mbox{if}& x>\ga_2(t).
\end{array}\right.
\end{eqnarray*}
Let $$D_1=\{(x,t):\ \min\{0,\eta_2(t)\}<x\leq 0\}\cup \{(x,t): \ 0\leq x <\max\{0,\ga_1(t)\}\},$$
then $u_1$ satisfies 
$$u_1(\eta_2(t)+,t)=w_-,\ u_1(\ga_1(t)-,t)=u_+.$$
\end{description}

\begin{figure}
	\centering
	\begin{tikzpicture}
	\draw (-4,0)--(5,0);
	\draw (-4,4)--(5,4);
	\draw (0,0)--(0,4);
	%left hand side
	\draw (-1,0)--(0,4);
	\draw (-2.5,0)--(0,4);
	%right hand side
	\draw(1.5,0)--(0,4);
	\draw (3,0)--(0,4);
	\draw (4,0)--(0,4);
	%numbering
	%x-axis
	\draw (0,-.2)node{\scriptsize $x=0$};
	\draw (0,0)node{\tiny $\bullet$};
	\draw (-1,-.2)node{\scriptsize $\xi_{2}$};
	\draw (-1,0)node{\tiny $\bullet$};
	\draw (-2.5,-.2)node{\scriptsize $y_{-}$};
	\draw (-2.5,0)node{\tiny $\bullet$};
	\draw (1.5,-.2)node{\scriptsize $\xi_{0}$};
	\draw (1.5,0)node{\tiny $\bullet$};
	\draw (3,-.2)node{\scriptsize $\xi_{1}$};
	\draw (3,0)node{\tiny $\bullet$};
	\draw (4,-.2)node{\scriptsize $y_{+}$};
	\draw (4,0)node{\tiny $\bullet$};
	%t-axis
	\draw (0,4.25)node{\scriptsize $(R_{1}, T)$};
	\draw (0,4)node{\tiny $\bullet$};
	%left hand side
	\draw (-.4,.7)node{\scriptsize $D_{1}$};
	\draw (-1.4,.7)node{\scriptsize $D_{1}$};
	% \draw [<-](-.55,2)--(-2,2);
	\draw (-.7,2)node{\scriptsize $\eta_{1}$};
	%\draw [<-](-1,2.5)--(-2,2.5);
	\draw (-1.2,2.5)node{\scriptsize $\eta_{2}$};
	%right hand side
	\draw (.7,.7)node{\scriptsize $D_{1}$};
	\draw (1.8,.7)node{\scriptsize $D_{1}$};
	\draw (2.9,.7)node{\scriptsize $D_{1}$};
	%\draw [<-](.8,2)--(2.5,2);
	\draw (.6,2)node{\scriptsize $\gamma_{3}$};
	% \draw [<-](1.2,2.5)--(2.5,2.5);
	\draw (1.2,2.1)node{\scriptsize $\gamma_{2}$};
	% \draw [<-](1.15,3)--(2.5,3);
	\draw (1.8,2)node{\scriptsize $\gamma_{1}$};
	\end{tikzpicture}
	 	\caption{Solution with reflected characteristics as in \descref{case-3}{Case 3.}}	\label{fig-14}
\end{figure}
\section{Backward construction}\label{backward}

\begin{lemma}\label{lemma6.2}
Let $0<R_2$ and $y:[0, R_2]\rr (-\f,0]$ be a non decreasing function. Define 
\begin{eqnarray*}
y_0 &=& y(0+), y_1=y(R_2-),\\
h_+\left(\frac{R_2}{T-t_1}\right) &=&-\frac{y_1}{t_1},\\
g'(u_-)&=&-\frac{y_0}{T}, g'(w_-)=-\frac{y_1}{t_1}, f(\bar{w}_-)=g(w_-), f'(\bar{w}_-)\geq 0,\\
\eta_3(t) &=& g'(u_-)(t-T), \eta_2(t)= g'(w_-)(t-t_1),\\
\ti{\eta}_2 (t) &=& f'(\bar{w}_-)(t-t_1).
\end{eqnarray*}
Let  $$D_2=\{(x,t):\ \eta_3(t)<x<\min\{\eta_2(t),0\}\}\cup \{(x,t):\ 0\leq x<\max\{\ti{\eta}_2(t),0\}\}.$$
Then there exists a $u_{2,0}\in L^\f(\R)$ and the corresponding solution $u_2$ of (\ref{conlaw-equation}) such that 
$$u_2(\eta_3(t)+,t)=u_-,\ u_2(\eta_2(t)-,t)=w_-,\ u_2(\ti{\eta}_2(t)-,t)=\bar{w}_-.$$ 
\end{lemma}
\begin{proof}
Without loss of generality by approximation,  we assume that $y$ is a strictly increasing continuous function and $N>1$, let $k>1$ and define a discreatization by 
	\begin{eqnarray}
	\left\{\begin{array}{rll}
	y_{0}&=&z_{0}<z_{1}<\cdots<z_{k}=y_{1},\\
	|z_{i+1}-z_{i}|&<&\frac{1}{N},\\
	0&=&x_{0}<x_{1}<\cdots<x_{k}=R_{2},\\
	y(x_{i})&=&z_{i} \mbox{ with } y_{0}=y(0) \mbox{ and } y_{1}=y(R_{2}-).
	\end{array}\right.
	\end{eqnarray}
	Let $\tau_{0}=T$ and define $\{\tau_{i}\}$ for $1\le i\le 2k$, $\{a_{i}\}, \{b_{i}\}, \{\tau_{i}(x)\}$ for $1\le i\le k$ by 
	\begin{eqnarray*}
	h_{+}\left(\frac{x_{i}}{T-\tau_{2i-1}}\right)&=&-\frac{z_{i-1}}{\tau_{2i-1}},\\
		h_{+}\left(\frac{x_{i}}{T-\tau_{2i}}\right)&=&-\frac{z_{i}}{\tau_{2i}},\\
			%h_{+}\left(\frac{x}{T-\tau_{i}(x)}\right)&=&-\frac{z_{i}}{\tau_{i}(x)}, x\in[x_{i}, x_{i+1}]\\
				f'(a_{2i-1})&=&\frac{x_{i}}{T-\tau_{2i-1}},\\ f'(a_{2i})&=&\frac{x_{i}}{T-\tau_{2i}},\\
				f(a_{i})&=&g(b_{i}),\\
				g'(b_{i})&\ge&0.				
	\end{eqnarray*}
	Observe that $\tau_{2k}=t_{1}$, $g'(b_{0})=-\frac{y_{0}}{T}=g'(u_-)$, $g'(b_{2k})=w_{-}$ and $f'(a_{2k})=\bar{w}_{-}$. Then from lemma \ref{3.3}, we have $a_{2i-1}>a_{2i}$, $b_{2i-1}>b_{2i}$, $T=\tau_{0}>\tau_{1}>\tau_{2}>\cdots>\tau_{2k}=t_{1}$. Define 
	\begin{eqnarray*}
	s_{i}&=&\frac{f(a_{2i-1})-f(a_{2i})}{a_{2i-1}-a_{2i}}, 1\le i\le k,\\
	S_{i}&=&\frac{g(b_{2i-1})-g(b_{2i})}{b_{2i-1}-b_{2i}}, 1\le i\le k,\\
	r_{0}(t)&=&g'(b_{0})(t-T)=g'(u_{-})(t-\tau_{0}),\\
	r_{i}(t)&=&g'(b_{i})(t-\tau_{i})\mbox{ for }1\le i\le 2k,\\
	\tilde{r}_{i}(t)&=&f'(a_{i})(t-\tau_{i})\mbox{ for }1\le i\le 2k,\\
	\alpha_{i}(t)&=&x_{i}+s_{i}(t-T),\\
	\beta_{i}(t)&=&S_{i}(t-\delta_i),
	\end{eqnarray*}
	where $\delta_i$ is defined by $\alpha_{i}(\delta_i)=0$. Then from the convexity of $f$ and $g$, we have 
	$\tau_{2i-1}<\delta_i<\tau_{2i}$, $z_{i-1}<\beta_{i}(0)<z_{i}$. Since $g'(u_{-})=g'(b_{0})=-\frac{y_{0}}{T}$, $\tau_{2k}=t_{1}$, $ f'(a_{2k})=\frac{x_{k}}{T-\tau_{2k}}=\frac{R_{2}}{T-t_{1}}=f'(\bar{w}_{-})$ and $g'(b_{2k})=-\frac{z_{k}}{\tau_{2k}}=-\frac{y_{1}}{t_{1}}$, hence $b_{2k}=w_{-}$. 
	Define 
	\begin{eqnarray}
	u^{N}_{2, 0}=\left\{\begin{array}{lll}
	u_{-}&\mbox{if}& x<y_{0}=z_{0},\\
	b_{2i-1}&\mbox{if}&z_{i-1}<x<\beta_{i}(0), 1\le i\le k,\\
	b_{2i}&\mbox{if}&\beta_{i}(0)<x<z_{i},\\
	w_{-}&\mbox{if}&z_{2k}<x<0,\\
	\bar{w}_- &\mbox{if}& x>0.
	\end{array}\right.
	\end{eqnarray}
	Then the solution $u^{N}_{2}$ of \eqref{conlaw-equation} with initial data $u^{N}_{2, 0}$ in $\R\times(0, T)$ is given by (see figure \ref{fig-10})
	\begin{eqnarray}
	u^{N}_2(x, t)=\left\{\begin{array}{lll}
	u_{-}&\mbox{if}&x<r_{0}(t),\\
	(g')^{-1}\left(\frac{x-z_{i}}{t}\right)&\mbox{if}&r_{2i}(t)<x<\min\{r_{2i+1}(t), 0\},\\
	(f')^{-1}\left(\frac{x}{t-\tilde{t}_{i}(x, t)}\right)&\mbox{if}&\max\{\tilde{r}_{2i}(t), 0\}<x<\tilde{r}_{2i+1}(t),\\
	b_{2i-1}&\mbox{if}&r_{2i-1}(t)<x<\min\{S_{i}(t), 0\},\\
	b_{2i}&\mbox{if}&S_{i}(t)<x<\min\{r_{2i}(t), 0\},\\
	a_{2i-1}&\mbox{if}&\max\{\tilde{r}_{2i+1}(t), 0\}<x<s_{i}(t),\\
	a_{2i}&\mbox{if}&\max\{s_{i}(t), 0\}<x<\tilde{r}_{2i}(t),\\
	w_{-}&\mbox{if}&r_{2k}(t)<x<0,\\
	\bar{w}_{-}&\mbox{if}&\max\{\tilde{r}_{2k}(t), 0\}<x,	
	\end{array}\right.
	\end{eqnarray}
	where $\tilde{t}_{i}(x, t)$ is the unique solution of 
	\begin{equation*}
	h_{+}\left(\frac{x}{t-\tilde{t}_{i}(x, t)}\right)=-\frac{z_{i}}{\tilde{t}_{i}(x, t)},\mbox{ for } x\in(x_{i}, x_{i+1}), i\le k-1.
	\end{equation*}
	 \begin{figure}
		\centering
		\begin{tikzpicture}
		\draw (-7,0)--(7,0);
		\draw (0,0)--(0,5.5);
		\draw (-7,5.5)--(7,5.5);
		\draw [dotted,thick](-2,0)--(0,2.7);
		\draw[dotted, thick](-2,0)--(0,2);
		\draw[dotted, thick](-2,0)--(0,2.3);
		\draw (-2,0)--(0,1.7);
		\draw (-2,0)--(0,3);
		\draw (0,1.7)--(6.5,5.5);
		\draw (0,3)--(3,5.5);
		\draw[dotted, thick](0,2.7)--(3.9,5.5);
		\draw[dotted, thick](0,2)--(5.7,5.5);
		\draw[dotted, thick](0,2.3)--(4.8,5.5);
		\draw (3,5.5)--(0,4);
		\draw (3,5.5)--(0,3.5);
		\draw (0,3.5)--(-3,0);
		\draw (0,4)--(-4.2,0);
		\draw [dashed](-3.5,0)--(0,3.5);
		\draw [dashed](-2.5,0)--(0,3.5);
		\draw[dotted,thick](-3.25,0)--(-1.75,1.48);
		\draw[dotted,thick](-2.75,0)--(-1.75,1.48);
		\draw [dotted, thick](-3.75,0)--(0,3.75);
		\draw [dotted, thick](-2.25,0)--(0,3.25);
		\draw [dotted, thick](0,3.75)--(1.45,4.47);
		\draw [dotted, thick](0,3.25)--(1.45,4.47);
		\draw (-4.2,0)--(0,5);
		\draw [dashed] (-4.2,0)--(0,4.75);
		\draw[dashed](-4.2,0)--(0,4.5);
		\draw[dashed](-4.2,0)--(0,4.25);
		\draw [dashed](0,4.25)--(2.5,5.5);
		\draw[dashed](0,4.5)--(2,5.5);
		\draw[dashed](0,4.75)--(1.5,5.5);
		\draw (0,5)--(1,5.5);
		%numbering
		\draw (1,5.75) node{\small$(x_{1},T)$};
		\draw (3,5.75) node{\small$(x_{2},T)$};
		\draw (6.5,5.75) node{\small$(x_{3},T)$};
		\draw (-2,-.25) node{\footnotesize$(\rho_{2},0)$};
		\draw (-3,-.25) node{\footnotesize$(S,0)$};
		\draw (-4.2,-.25) node{\footnotesize$(\rho_{1},0)$};
		\draw (7.5,0) node{\small$t=0$};
		\draw (7.5,5.5) node{\small$t=T$};
		\draw (-3,2) node{\footnotesize$b_{1}$};
		\draw (-2.5,1) node{\footnotesize$b_{2}$};
		\draw (-2,.6) node{\footnotesize$b_{3}$};
		\draw (-.5,.5) node{\footnotesize$b_{4}$};
		\draw (2,1) node{\footnotesize$a_{4}$};
		\draw (.35,3.5) node{\footnotesize$a_{3}$};
		\draw (.35,3.95) node{\footnotesize$a_{2}$};
		\draw (.3,5.3) node{\footnotesize$a_{1}$};
		\draw (-4.2,0) node{\tiny $\bullet$};
		\draw (-2,0) node{\tiny $\bullet$};
		\draw (-3,0) node{\tiny $\bullet$};
		\draw (1,5.5) node{\tiny $\bullet$};
		\draw (3,5.5) node{\tiny $\bullet$};
		\draw (6.5,5.5) node{\tiny $\bullet$};
		%\draw (1.5,2.5) node{\footnotesize$a_{2}$};
		\end{tikzpicture}
		\caption{The figure illustrates the approximate solution in $D_2$.}\label{fig-10}
		%\captionof{figure}{An illustration to construct initial data when there is exactly one shock at $(x_{2},T)$.}
	\end{figure}
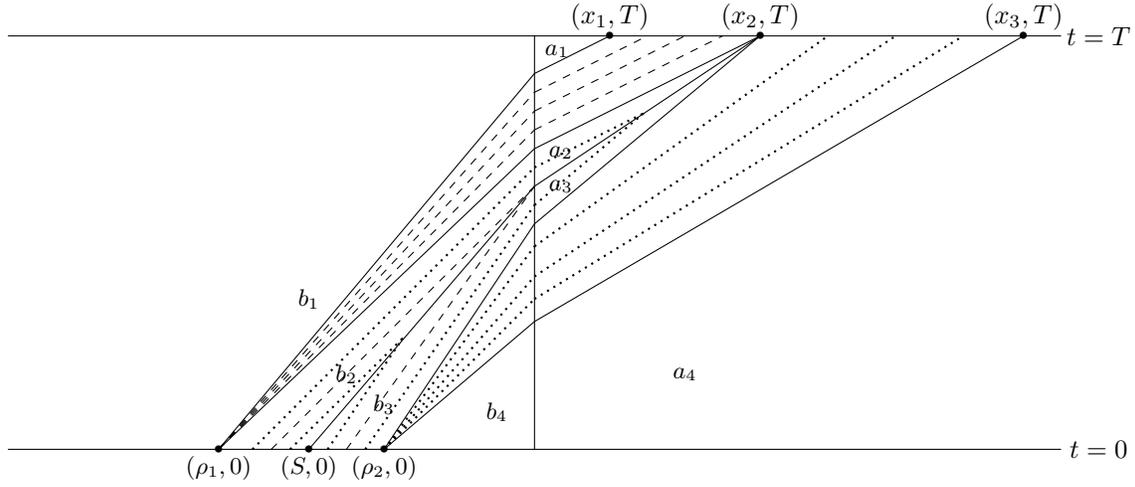
Next we show that the above sequences converges: \\
%%%%%%%%%%%%%%*******Starts-here-(12)-(16)*************%%%%%%%%%%
\noi{\textbf{Convergence Analysis:}}
First we show that $\left\{ || u_{2,0}^N||_\f\right\}$ is uniformly bounded. Let $i_0=\sup\{i:\ \tau_i\leq T/2\}.$ For $i\leq i_0$, we have 
\begin{eqnarray*}
f'(a_i)=\left\{\begin{array}{llll}
\frac{x_l}{T-\tau_{2l}} &\mbox{if}& l=i/2,\\
 \frac{x_l}{T-\tau_{2l-1}} &\mbox{if}& l=(i+1)/2.
\end{array}\right.
\end{eqnarray*}
Hence $f'(a_i)\leq \frac{2R_2}{T}.$
For $i\geq i_0$, then 
\begin{eqnarray*}
g'(b_i)=\left\{\begin{array}{lll}
-\frac{z_{l-1}}{\tau_{2l-1}} &\mbox{if}& l=(i+1)/2,\\
 -\frac{z_l}{\tau_{2l}} &\mbox{if}& l=i/2.
\end{array}\right.
\end{eqnarray*}
Thus, we have $g'(b_i)\leq \frac{2|y_0|}{T}.$ Since $f(a_i)=g(b_i)$, $g'(b_i)\geq 0$, we get $\{b_i\}$ is uniformly bounded in $\R$ and  $\{u_{2,0}^N\}$ is uniformly bounded in $L^\f(\R)$. 
%
%Since $\tilde{u}^N_{3,0}$ is independent of $N$ in $\R\setminus[y_0,y_1]$, hence need to consider only in $[y_0,y_1]$. For any function $A:[y_0,y_1]\rr\R$ , denote $TV(A)$ to be the total variation in $[y_0,y_1]$. Observe that  
%\begin{align}
%g^{\p}(b_{2i-1})&=h_+(f^{\p}(a_{2i-1}))\nonumber\\
%&=h_+\left(\frac{x_i}{T-\tau_{2i-1}}\right)\nonumber\\
%&=-\frac{z_{i-1}}{\tau_{2i-1}}\nonumber
%\end{align}
%and similarly $g^{\p}(b_{2i})=-\frac{z_{i}}{\tau_{2i}}$. Now $t_1\leq \tau_{2i}\leq T,\,z_i\in[y_0,y_1]$, hence $\{TV(g^{\p}(u^N_{3,0}))\}$ is uniformly bounded.

First assume that $f$ and $g$ are uniformly convex. Then $h_+$ is a Lipschitz continuous function.
\begin{align*}
TV(g^{\p}(u^N_{2,0}))&=\sum\limits_{i=1}^{2k-1}\abs{g^{\p}(b_{i+1})-g^{\p}(b_i)}\\
&=\sum\limits_{i=1}^{k}\abs{g^{\p}(b_{2i-1})-g^{\p}(b_{2i})}+\sum\limits_{i=1}^{k-1}\abs{g^{\p}(b_{2i})-g^{\p}(b_{2i+1})}\\
&=\sum\limits_{i=1}^{k}\abs{\frac{z_{i-1}}{\tau_{2i-1}}-\frac{z_i}{\tau_{2i}}}+\sum\limits_{i=1}^{k-1}\abs{\frac{z_i}{\tau_{2i}}-\frac{z_{i+1}}{\tau_{2i+1}}}\\
&=\sum\limits_{\tau_{2i}\leq T/2}\abs{\frac{z_{i-1}}{\tau_{2i-1}}-\frac{z_i}{\tau_{2i}}}+\sum\limits_{\tau_{2i+1}\leq T/2}\abs{\frac{z_i}{\tau_{2i}}-\frac{z_{i+1}}{\tau_{2i+1}}}\\
&+\sum\limits_{\tau_{2i}> T/2}\abs{\frac{z_{i-1}}{\tau_{2i-1}}-\frac{z_i}{\tau_{2i}}}+\sum\limits_{\tau_{2i+1}> T/2}\abs{\frac{z_i}{\tau_{2i}}-\frac{z_{i+1}}{\tau_{2i+1}}}\\
&=I_1+I_2,
\end{align*}
where 
\begin{align*}
I_1&=\sum\limits_{\tau_{2i}\leq T/2}\abs{\frac{z_{i-1}}{\tau_{2i-1}}-\frac{z_i}{\tau_{2i}}}+\sum\limits_{\tau_{2i+1}\leq T/2}\abs{\frac{z_i}{\tau_{2i}}-\frac{z_{i+1}}{\tau_{2i+1}}}\\
&=\sum\limits_{\tau_{2i}\leq T/2}\abs{h_+\left(\frac{x_i}{T-\tau_{2i-1}}\right)-h_+\left(\frac{x_i}{T-\tau_{2i}}\right)}+\sum\limits_{\tau_{2i+1}\leq T/2}\abs{h_+\left(\frac{x_i}{T-\tau_{2i}}\right)-h_+\left(\frac{x_{i+1}}{T-\tau_{2i+1}}\right)}.
\end{align*}
As $f,g$ are uniformly convex, we get $h_+$ is a locally Lipschitz function. Due to $\tau_{2i}\leq T/2$, $\tau_{2i+1}\leq T/2$, we obtain $T-\tau_{2i}\geq T/2$, $T-\tau_{2i+1}\geq T/2$, hence
 $\frac{x_i}{T-\tau_{2i}},\frac{x_i}{T-\tau_{2i+1}}$ are bounded. Let $M=$ Lipschitz constant of $h_+$ on $\left[\bar{\theta}_g,\frac{2R_2}{T}\right]$, then 
 \begin{align*}
 I_1&\leq M\left[\sum\limits_{\tau_{2i}\leq T/2}\abs{\frac{x_i}{T-\tau_{2i-1}}-\frac{x_i}{T-\tau_{2i}}}+\sum\limits_{\tau_{2i+1}\leq T/2}\abs{\frac{x_i}{T-\tau_{2i}}-\frac{x_{i+1}}{T-\tau_{2i+1}}}\right]\\
 &\leq \frac{4R_2M}{T^2}\left[\sum\limits_{\tau_{2i}\leq T/2}\abs{\tau_{2i}-\tau_{2i-1}}+\sum\limits_{\tau_{2i+1}\leq T/2}\abs{\tau_{2i+1}-\tau_{2i}}\right]+\frac{4M}{T^2}\sum\limits_{\tau_{2i+1}\leq T/2}\abs{x_i-x_{i+1}}\\
 &\leq\frac{4M}{T^2}\left(R_2+1\right).
 \end{align*}
Since $\{\tau_{2i}\}$ is a decreasing sequence and $\{x_i\}$ is an increasing sequence, we have 
\begin{align*}
I_2&=\sum\limits_{\tau_{2i}> T/2}\abs{\frac{z_{i-1}}{\tau_{2i-1}}-\frac{z_i}{\tau_{2i}}}+\sum\limits_{\tau_{2i+1}> T/2}\abs{\frac{z_i}{\tau_{2i}}-\frac{z_{i+1}}{\tau_{2i+1}}}\\
&\leq \frac{4\abs{y_0}}{T^2}\sum\limits_{i=1}^{2k-1}\abs{\tau_{i+1}-\tau_{i}}+\frac{4}{T^2}\sum\limits_{i=1}^{k-1}\abs{z_{i+1}-z_{i}}\\
&\leq \frac{4}{T^2}\left\{(T-t_1)\abs{y_0}+y_1-y_0\right\}.
\end{align*}

 Therefore, by Helly's Theorem, there exists a subsequence still denoting by $\{g^{\p}(u^N_{2,0})\}$ converges pointwise to $g^{\p}(u_{2,0})$. Hence $\forall y\in\R$,
\[\lim\limits_{N\rr\f}u^N_{2,0}(y)=u_{2,0}(y),\]
and $u_{2,0}\in L^{\f}(\R)$ with 
\begin{equation*}
u_{2,0}(y)=\left\{\begin{array}{rl}
u_-&\mbox{ if }y<y_0,\\
w_-&\mbox{ if }y_1<y<0,\\
\bar{w}_-&\mbox{ if }y>0.
\end{array}\right.
\end{equation*}
Let 
\[y_N(x)=\sum\limits_{i=0}^{k-1}z_i\chi_{[x_i,x_{i+1})}(x).\]
Then
\begin{align}
\abs{y(x)-y_N(x)}&=\abs{\sum\limits_{i=0}^{k-1}(z_i-y(x))\chi_{[x_i,x_{i+1})}}\nonumber\\
&\leq \frac{1}{N}.\nonumber
\end{align}
Thus, $y_N\rr y$ in $L^\f[0,R_2]$. Let $\tau_i(x)=\tilde{t}_i(x,T)$ for $x\in[x_i,x_{i+1}]$. Define
	\begin{equation*}
	t^N(x)=\tau_i(x),\mbox{ if }x\in(x_i,x_{i+1}),
	\end{equation*}
then for a.e. $x\in(0,R_2)$, we have, $t^N$ is a strictly increasing function, $t_1\leq t^N(x)\leq T$ and for a.e., $x\in(0,R_2)$, we have
\begin{equation}\label{eqn:yN-tildeuN}
\begin{array}{rl}
-\frac{y_N(x)}{t^N(x)}&=h_+\left(\frac{x}{T-t^N(x)}\right),\\
u^N_2(x,T)&=(f^{\p})^{-1}\left(\frac{x}{T-t^N(x)}\right),\\
u_2^N(\eta_3(t)+,t)&=u_-,\\
u_2^N(\eta_2(t)-,t)&=w_-,\\
u_2^N(\ti{\eta}_2(t)-,t)&=\bar{w}_-.
\end{array}
\end{equation} 
From the construction, set of discontinuities of $u^N_2$ are discrete set of Lipschitz curves in $\R\times[0,T]$, therefore, from  \descref{xii.}{(xii)} of theorem \ref{theorem3.1}
\begin{equation*}
\int\limits_{\R}\abs{u^{N_1}_2(x,t)-u^{N_2}_2(x,t)}\,dx\leq\int\limits_{y_0}^{y_1}\abs{u^{N_1}_{2,0}(x)-u^{N_2}_{2,0}(x)}\,dx.
\end{equation*}
Subsequently, we have
\begin{align}
\int\limits_{0}^T\int\limits_{\R}\abs{u^{N_1}_2(x,t)-u^{N_2}_2(x,t)}\,dxdt&\leq T\int\limits_{y_0}^{y_1}\abs{u^{N_1}_{2,0}(x)-u^{N_2}_{2,0}(x)}\,dx\nonumber\\
&\rr0\mbox{ as }N_1,N_2\rr\f.\nonumber
\end{align}
Hence for a subsequence still denoted by $\{u^N_2\}$ converges to $u_2$, a solution of \eqref{conlaw-equation} with initial data $u_{2,0}$. From Helly's Theorem, again for a subsequence,
\[\lim\limits_{N\rr\f}t^N(x)=t_+(x).\]

%
%**********************************************************
%
Then from \eqref{eqn:yN-tildeuN}, letting $N\rr\f$ to obtain for a.e. $x$
\begin{equation}
\begin{array}{rl}
-\frac{y(x)}{t_+(x)}&=h_+\left(\frac{x}{T-t_+(x)}\right),\\
{u}_2(x,T)&=(f^{\p})^{-1}\left(\frac{x}{T-t_+(x)}\right)
\end{array}
\end{equation}
and $u_2$ satisfies the conclusion of the lemma. If $f$ and $g$ are not uniformly convex (and just strictly convex), then approximate $f$ and $g$ by $f_\epsilon$ and $g_\epsilon$ respectively which are uniformly convex and by stability lemma \ref{stability}, the lemma follows as $\e\rr 0.$
\end{proof}

\subsection{Proof of theorem \ref{theorem1.1}}\label{section4.1}
\begin{proof}[Proof of theorem \ref{theorem1.1}]
First we prove that if $u_0\in L^\f(\R)$ and $u$ is the corresponding solution of (\ref{conlaw-equation}),
then $(T,R_1(T), R_2(T), y(\cdot,T))\in \mathcal{R}(T).$ From lemma \ref{lemma3.6}, if $R_1(T)=0$ or there exists a  $0\leq T_1 \leq T$ such that $f'(\bar{\T}_g)=\frac{R_1(T)}{T-T_1}$, then $y(R_1(T)+,T)\geq \xi_0$. Hence $(T, R_1(T), R_2(T), y(T))\in \mathcal{R}(T)$. Conversely, let $(T,R_1,R_2, y(\cdot))\in \mathcal{R}(T),$ define $y_0=y(0+),$ $y_-=y(R_2(T)-),$ $y_+=y(R_1+)$ and $t_0$ by 
\begin{eqnarray*}
h_+\left(\frac{R_2}{T-t_0}\right) &=&-\frac{y_-}{t_0}\mbox{ and define}\\
f'(u_+) &=& \frac{R_1-y_+}{T}, f'(\bar{w}_-)=\frac{R_2}{T-t_0}, g'(w_-)=-\frac{y_-}{t_0},\\
g'(u_-)&=& -\frac{y_0}{T},\\
\ga_1(t)&=& R_1+f'(u_+)(t-T),\\
\ga_2(t) &=& R_2+ f'(\bar{w}_-)(t-T),\\
\eta_2(t) &=& g'(w_-) (t-t_0),\\
\eta_3(t) &=& g'(u_-)(t-T).
\end{eqnarray*}
Let for $0<t<T$, define 
%\begin{eqnarray*}
%D_1&=&\{(x,t):\ \min\{\eta_2(t),0\}<x<**\}\cup \{(x,t):\ \max\{\ga_2(t),0\}<x<\ga_1(t)\},\\
%D_2 &=& \{(x,t):\ x<\eta_3(t)\}\cup \{(x,t):\ x>\ga_1(t)\},\\
%D_3&=& \{(x,t):\ \eta_3(t)<x<\min\{\eta_2(t),0\}\}\cup \{(x,t): \ 0<x<\max\{\ga_2(t),0\}\}.
%\end{eqnarray*}
%Let $\{u_{1,0}, u_{3,0}\}\subset L^\f(\R)$ and $\{u_1,u_3\}$ be the corresponding solution of (\ref{conlaw-equation}) constructed in lemma** and lemma ** associated with $D_1$ and $D_3$. From the backward construction \cite{Sco}, there exists $u_{2,0}\in L^\f((-\f,y_0)\cup(y_+,\f))$ and a solution $u$ of (\ref{conlaw-equation}) in $D_2$ such that 
%$$u_2(\eta_3(t)-,t)=u_-, u_2(\ga_1(t)+,t)=u_+.$$
%Now define 
%\begin{eqnarray*}
%u_0(x)=\left\{\begin{array}{lll}
%u_{1,0}(x) &\mbox{if}& y_1<x<y_+,\\
% u_{3,0} &\mbox{if}& y_+<x<y_0,\\
% u_{2,0} &\mbox{if}& x<y_0 \mbox{ or } x>y_+
%\end{array}\right.
%\end{eqnarray*}
%and 
%\begin{eqnarray*}
%u(x,t)=\left\{\begin{array}{lll}
%u_{1}(x,t) &\mbox{in}& D_1,\\
% u_{3}(x,t) &\mbox{in}& D_3,\\
% u_{2}(x,t) &\mbox{in}& D_2.
%\end{array}\right.
%\end{eqnarray*}
%By the construction $u$ satisfies the R-H condition across the boundaries of 
%$D_i$ in $\R\times (0,\f), \ i=1,\ 2,\ 3,$ hence $u$ is the solution of (\ref{conlaw-equation}) with initial data $u_0$. Furthermore, by construction, $u$ satisfies 
%$$(T_1, R_1(T), R_2(T), y(\cdot,T))=(T, R_1, R_2, y(\cdot)).$$
%This proves the theorem. 
\begin{eqnarray*}
D_1 &=&\{(x,t): \min\{\eta_2(t),0\}<x\leq 0\}\cup \{(x,t):\ \max\{\ga_2(t),0\}<x<\ga_1(t)\},\\
D_2&=&\{(x,t):\ \min\{\eta_3(t),0\}<x<\min\{\eta_2(t),0\}\}\cup\{(x,t):\ 0\leq x<\max\{\ga_2(t),0\}\},\\
D_3&=&\{(x,t):\ x<\eta_3(t)\}\cup \{(x,t):\ x>\ga_1(t)\},\\
I_i&=&\bar{D}_i\cap \R,\ i=1,2,3.
\end{eqnarray*}
From subsection \ref{newsection}, there exists a $u_{1,0}\in L^\f(\R)$ and the corresponding solution $u_1$ such that 
$$u_1(\eta_2(t)+,t)=w_-,\ u_1(\ga_2(t)+,t)=\bar{w}_-, u_1(\ga_1(t)-,t)=u_+.$$
From lemma \ref{lemma6.2}, there exists a $u_{2,0}\in L^\f(\R)$ and the corresponding solution $u_2$ of (\ref{conlaw-equation}) satisfies:
$$u_2(\eta_3(t)+,t)=u_-, u_2(\eta_2(t)-,t)=w_-, u(\ga_2(t)-,t)=\bar{w}_-.$$
From the backward construction \cite{Sco}, there exists a $u_{3,0}\in L^\f(\R)$ and the corresponding solution $u_3$ of (\ref{conlaw-equation}) such that 
$$u_3(\eta_3(t)-,t)=u_-, u_3(\ga_1(t)+,t)=u_+.$$
Therefore by R-H condition if we define 
\begin{eqnarray*}
u_0(x)=\left\{\begin{array}{llll}
u_{1,0}(x) &\mbox{if}& x\in \mbox{Interior of}\ I_1,\\
u_{2,0}(x) &\mbox{if}& x\in \mbox{Interior of}\ I_2,\\
u_{3,0}(x) &\mbox{if}& x\in \mbox{Interior of}\ I_3,
\end{array}\right.
\end{eqnarray*}
then $u$ is the solution of (\ref{conlaw-equation}) with initial data $u_0$ given by
\begin{eqnarray*}
u(x,t)=\left\{\begin{array}{llll}
u_{1}(x,t) &\mbox{if}& (x,t)\in D_1,\\
u_{2}(x,t) &\mbox{if}& (x,t)\in D_2,\\
u_{3}(x,t) &\mbox{if}& (x,t)\in D_3,
\end{array}\right.
\end{eqnarray*}
satisfying $R_i(T)=R_i$, $i=1,2$, $y(\cdot,T)=y(\cdot)$. This proves the theorem.
\end{proof}
\subsection{Proof of theorem \ref{theorem1.3}}\label{exactcontrol}
%\begin{theorem}
%	Let $(T, R_{1}, R_2, y(\cdot))\in R(T)$ and $C_{1}<0<R_{1}<C_{2}$, $B_{1}<0<B_{2}$ are given. Assume that 
%	\begin{eqnarray}
%	y(C_{1}+)&>&B_{1},\\
%	y(C_{2}-)&<&B_{2},
%	\end{eqnarray}
%	and $u_{1, 0}\in L^{\f}(\R\setminus(B_{1}, B_{2}))$. Then there exist a $\tilde{u}_{0}\in L^{\f}(B_{1}, B_{2})$ and the solution $u$ of \eqref{conlaw-equation} with initial data $u_{0}$ satisfying 
%	\begin{eqnarray}
%	u_{0}(x)=\left\{\begin{array}{lll}
%	u_{1, 0}(x)&\mbox{if}&x\notin(B_{1}, B_{2}),\\
%	\tilde{u}_{0}(x)&\mbox{if}&x\in(B_{1}, B_{2}).
%	\end{array}\right.
%	\end{eqnarray}
%	Let $(T, R_{1}, R_{2}, y(x, T))$ be the element in $R(T)$ corresponds to $u(\cdot, T)$. Then 
%	\begin{eqnarray}
%	R_{i}&=&R_{i}(T) \mbox{ for } i=1, 2,\\
%	y(x)&=&y(x, T) \mbox{ for all } x\in(C_{1}, R_{2})\cup(R_{1}, C_{2}).
%	\end{eqnarray}
%\end{theorem}

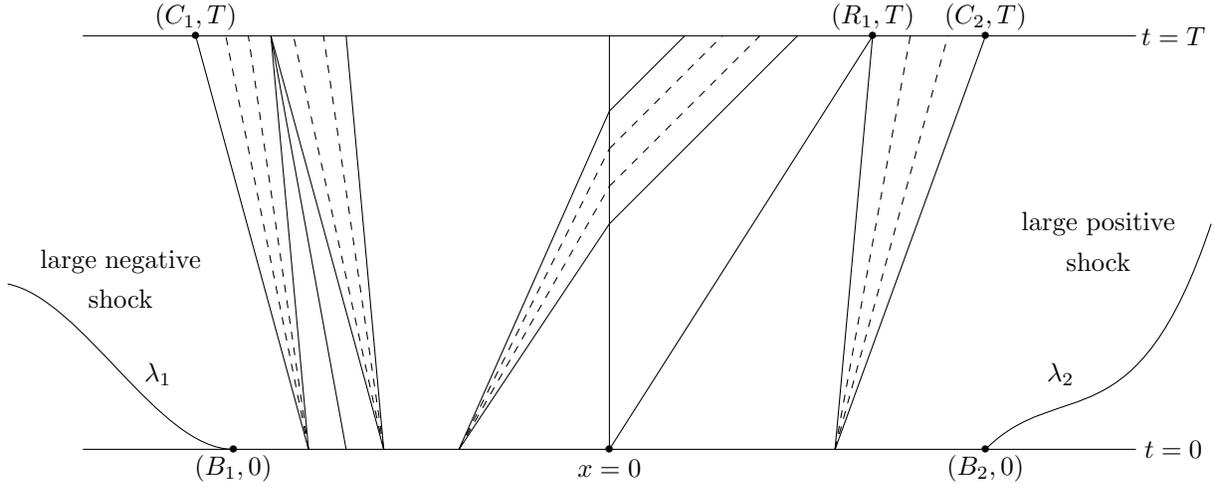
\begin{figure}
	\centering
	\begin{tikzpicture}
	\draw (-7,0)--(7,0);
	\draw (0,0)--(0,5.5);
	\draw (-7,5.5)--(7,5.5);
	\draw (-2,0)--(0,3);\draw (0,3)--(2.5,5.5);
	\draw(-2,0)--(0,4.5);\draw (0,4.5)--(1,5.5);
	\draw[dashed](-2,0)--(0,3.5);\draw[dashed](0,3.5)--(2,5.5);
	\draw[dashed](-2,0)--(0,4);\draw[dashed](0,4)--(1.5,5.5);
	\draw (0,0)--(3.5,5.5);
	\draw (3.5,5.5)--(3,0);
	\draw (3,0)--(5,5.5);
	\draw[dashed](3,0)--(4,5.5);
	\draw[dashed](3,0)--(4.5,5.5);
	\draw (-3,0)--(-3.5,5.5);
	\draw[dashed](-3,0)--(-3.8,5.5);
	\draw[dashed](-3,0)--(-4.2,5.5);
	\draw (-4.5,5.5)--(-3.5,0);
	\draw[dashed](-4,0)--(-4.8,5.5);
	\draw[dashed](-4,0)--(-5.1,5.5);
	\draw(-3,0)--(-4.5,5.5);
	\draw (-4.5,5.5)--(-4,0);
	\draw(-4,0)--(-5.5,5.5);
	\draw (5,0) .. controls (6,1) and (7,0) .. (8,3);
	\draw (-5,0) .. controls (-6,0) and (-7,2) .. (-8,2.2);
	%numbering
	\draw (7.5,0) node{\small$t=0$};
	\draw (7.5,5.5) node{\small$t=T$};
	\draw (0,-.25) node{\small$x=0$};
	\draw (-5,-.25) node{\small$(B_{1},0)$};
	\draw (5,-.25) node{\small$(B_{2},0)$};
	\draw (-5.5,5.75) node{\small$(C_{1},T)$};
	\draw (3.5,5.75) node{\small$(R_1,T)$};
	\draw (5,5.75) node{\small$(C_{2},T)$};
	\draw (-6,1) node{\small$\lambda_{1}$};
	\draw (6,1) node{\small$\lambda_{2}$};
	\draw (6.5,3) node{\small{large positive}};
	\draw (6.5,2.5) node{\small{shock}};
	\draw (-6.5,2.5) node{\small{large negative}};
	\draw (-6.5,2) node{\small{shock}};
	\draw (3.5,5.5) node{\tiny $\bullet$};
	\draw (5,5.5) node{\tiny $\bullet$};
	\draw (-5.5,5.5) node{\tiny $\bullet$};
	\draw (5,0) node{\tiny $\bullet$};
	\draw (-5,0) node{\tiny $\bullet$};
	\draw (0,0) node{\tiny $\bullet$};
	\end{tikzpicture}
		\caption{This figure illustrates the solution for exact control problem.}\label{fig-11}
	%	\captionof{figure}{An illustration of the Theorem.}
\end{figure}

%\begin{center}
%	\begin{tikzpicture}
%	\draw (-7,0)--(7,0);
%	\draw (0,0)--(0,5.5);
%	\draw (-7,5.5)--(7,5.5);
%	\draw (-2,0)--(0,3);\draw (0,3)--(2.5,5.5);
%	\draw(-2,0)--(0,4.5);\draw (0,4.5)--(1,5.5);
%	\draw[dashed](-2,0)--(0,3.5);\draw[dashed](0,3.5)--(2,5.5);
%	\draw[dashed](-2,0)--(0,4);\draw[dashed](0,4)--(1.5,5.5);
%	\draw (0,0)--(3.5,5.5);
%	\draw (3.5,5.5)--(3,0);
%	\draw (3,0)--(5,5.5);
%	\draw[dashed](3,0)--(4,5.5);
%	\draw[dashed](3,0)--(4.5,5.5);
%	\draw (-3,0)--(-3.5,5.5);
%	\draw[dashed](-3,0)--(-3.8,5.5);
%	\draw[dashed](-3,0)--(-4.2,5.5);
%	\draw (-4.5,5.5)--(-3.5,0);
%	\draw[dashed](-4,0)--(-4.8,5.5);
%	\draw[dashed](-4,0)--(-5.1,5.5);
%	\draw(-3,0)--(-4.5,5.5);
%	\draw (-4.5,5.5)--(-4,0);
%	\draw(-4,0)--(-5.5,5.5);
%	\draw (5,0) .. controls (6,1) and (7,0) .. (8,3);
%	\draw (-5,0) .. controls (-6,0) and (-7,2) .. (-8,2.2);
%	%numbering
%	\draw (7.5,0) node{\small$t=0$};
%	\draw (7.5,5.5) node{\small$t=T$};
%	\draw (0,-.25) node{\small$x=0$};
%	\draw (-5,-.25) node{\small$(B_{1},0)$};
%	\draw (5,-.25) node{\small$(B_{2},0)$};
%	\draw (-5.5,5.75) node{\small$(C_{1},T)$};
%	\draw (3.5,5.75) node{\small$(R,T)$};
%	\draw (5,5.75) node{\small$(C_{2},T)$};
%	\draw (-6,1) node{\small$\lambda_{1}$};
%	\draw (6,1) node{\small$\lambda_{2}$};
%	\draw (6.5,3) node{\small{large positive}};
%	\draw (6.5,2.5) node{\small{shock}};
%	\draw (-6.5,2.5) node{\small{large negative}};
%	\draw (-6.5,2) node{\small{shock}};
%	\end{tikzpicture}
%%	\captionof{figure}{An illustration of the Theorem.}
%\end{center}
\begin{proof}
	Define $\delta_{1}=y(C_{1}+)-B_{1}$, $\delta_{2}=B_{2}-y(C_{2}-)$
	\begin{eqnarray}
	\tilde{y}(x)=\left\{\begin{array}{lll}
	y(x)&\mbox{if}& x\in(C_{1}, R_{2})\cup(R_{1}, C_{2}),\\
	x&\mbox{if}& x<C_{1}, C_{1}<y(C_{1}+),\\
	x&\mbox{if}& x<y(C_{1}+)<C_{1},\\
	y(C_{1}+)&\mbox{if}& y(C_{1}+)<x<C_{1},\\
	x&\mbox{if}& x>C_{2},\\
	x&\mbox{if}& x>y(C_{2}-)>C_{2},\\
	y(C_{2}-)&\mbox{if}& C_{2}<x<y(C_{2}-).\\
	\end{array}\right.
	\end{eqnarray} 
	Let $\tilde{u}_{-}$, $\tilde{u}_{+}$ be defined by 
	\begin{eqnarray}
	g'(\tilde{u}_{-})&=&\frac{C_{1}-y(C_{1}+)}{T},\\
	f'(\tilde{u}_{+})&=&\frac{C_{2}-y(C_{2}-)}{T}
	\end{eqnarray}
	and 
	\begin{eqnarray}
	\gamma_{1}(t)&=&C_{1}+g'(\tilde{u}_{-})(t-T),\\
	\gamma_{2}(t)&=&C_{2}+f'(\tilde{u}_{+})(t-T).
	\end{eqnarray}
	Then from theorem \ref{theorem1.1}, there exists a $\tilde{u}_{0}\in L^{\f}(\R)$ and a solution $\tilde{u}$ with initial data $\tilde{u}_{0}$ such that 
	\begin{eqnarray}
	\tilde{u}(\gamma_{1}(t)+, t)&=&\tilde{u}_{-},\\
	\tilde{u}(\gamma_{2}(t)-, t)&=&\tilde{u}_{+}.
	\end{eqnarray}
	Then the free region lemmas 2.2, 2.3 and 2.4 as in \cite{Sco} (see figure \ref{fig-11} for illustration), one can find  $\lambda_{1}$ large negative number and $\lambda_{2}$ large positive number, such that there exist solutions $u_{2}$ and $u_{3}$ of \eqref{conlaw-equation} with respective initial data $u_{2,0}$ and $u_{3,0}$ given by 
	\begin{eqnarray}
	u_{2,0}=\left\{\begin{array}{lll}
	u_{1, 0}(x)&\mbox{if}&x<B_{1},\\
	\lambda_{1}&\mbox{if}&B_{1}<x<B_{1}+\delta_{1},\\
	\ti{u}_{-}&\mbox{if}&x>B_{1}+\delta_{1},
	\end{array}\right.
	\end{eqnarray}
	\begin{eqnarray}
	u_{3,0}=\left\{\begin{array}{lll}
	u_{1, 0}(x)&\mbox{if}&x>B_{2},\\
	\lambda_{2}&\mbox{if}&B_{2}-\delta_{2}<x<B_{2},\\
	\ti{u}_{+}&\mbox{if}&x<B_{2}-\delta_{2}
	\end{array}\right.
	\end{eqnarray}
	and satisfies
	\begin{eqnarray}
	\tilde{u}_{2}(\gamma_{1}(t)-, t)&=&\tilde{u}_{-},\\
	\tilde{u}_{3}(\gamma_{2}(t)+, t)&=&\tilde{u}_{+}.
	\end{eqnarray}
	Hence define
	\begin{eqnarray}
	u_{0}(x)=\left\{\begin{array}{lll}
	u_{1, 0}(x)&\mbox{if}&x<B_{1},\\
	\lambda_{1}&\mbox{if}&B_{1}<x<B_{1}+\delta_{1},\\
	\tilde{u}_{0}(x)&\mbox{if}&B_{1}+\delta_{1}<x<B_{2}-\delta_{2},\\
	\lambda_{2}&\mbox{if}&B_{2}-\delta_{2}<x<B_{2},\\
	u_{1, 0}(x)&\mbox{if}&B_{2}<x,
	\end{array}\right.
	\end{eqnarray}
	\begin{eqnarray}
	u(x, t)=\left\{\begin{array}{lll}
	u_{2}(x, t)&\mbox{if}& x<\gamma_{1}(t),\\
	\tilde{u}(x, t)&\mbox{if}&\gamma_{1}(t)<x<\gamma_{2}(t),\\
	u_{3}(x, t)&\mbox{if}&x>\gamma_{2}(t).
	\end{array}\right.
	\end{eqnarray}
	Then $(u_{0}, u)$ is the required solution satisfies the Theorem.
\end{proof}
\section{Optimal control}\label{optimalcontrol}
Let $K$ be given, the associated cost functional $J$ and admissible set $\mathcal{A}$ are as in (\ref{cost1}). Then we have the following:
\begin{lemma}
 For $u_0\in \mathcal{A}$, $J(u_0)$ is well defined.
 \begin{proof}
  Because of finite speed of propagation, it is immediate.
 \end{proof}
 \end{lemma}
\subsection{Proof of theorem \ref{theorem1.2}}
\begin{proof}[Proof of theorem \ref{theorem1.2}]
	Proof involves several steps,\\
	
	\begin{description}[font=\normalfont]
\descitem{Step 1:}{thm1-step-1} Let 	
	
$$\mathcal{\tilde{A}}=\{(T, R_{1}, R_{2}, y(\cdot))\in R(T): y(x)=x \mbox{ outside a compact set } \}.
$$
For $\alpha=(T, R_{1}, R_2, y(\cdot))\in\mathcal{\tilde{A}}$, define 
$$h_+\left(\frac{x}{T-t_+(x)}\right)=-\frac{y(x)}{t_+(x)},\ \mbox{for}\ x\in (0,R_2)\ \mbox{and}$$
\begin{eqnarray}
\mathcal{\tilde{J}}(\alpha)&=\int\limits_{-\f}^{0}\left|\frac{x-y(x)}{T}-K(x)\right|^2dx+\int\limits_{0}^{R_{2}}\left|\frac{y(x)}{t_+(x)}+K(x)\right|^2dx+\int\limits_{R_{2}}^{R_{1}}\left|\bar{\theta}_{g}-K(x)\right|^2dx\\\nonumber&+\int\limits_{R_{1}}^{\f}\left|\frac{x-y(x)}{T}-K(x)\right|^2dx.
\end{eqnarray}
	\begin{eqnarray*}
	R_{0}&=&\frac{f(\bar{\theta}_{g})-f(\theta_{f})}{\bar{\theta}_{g}-\theta_{f}}T,\\
	M_{1}&=&\int_{-\f}^{0}|K(x)|^2dx+\int_{R_{0}}^{\f}|K(x)|^2dx+\int_{0}^{R_{0}}|\bar{\theta}_{g}-K(x)|^2dx.
	\end{eqnarray*}
	Then
	\begin{equation*}
	\inf_{\alpha\in\mathcal{\tilde{A}}}\mathcal{\tilde{J}}(\alpha)\le \inf_{u_{0}\in\mathcal{A}}\mathcal{J}(u_{0})\le M_{1}.
	\end{equation*}
	\begin{proof}[Proof of \descref{thm1-step-1}{Step 1}:]
	Let
	\begin{eqnarray*}
	w_{0}(x)=\left\{\begin{array}{lll}
		\theta_{g}&\mbox{if}&x<0,\\
		\theta_{f}&\mbox{if}&x>0.
	\end{array}\right.
\end{eqnarray*}
Then $w$ is the solution to \eqref{conlaw-equation} with data  $w_{0}$, where 
	\begin{eqnarray*}
w(x, t)=\left\{\begin{array}{lll}
\theta_{g}&\mbox{if}&x<0,\\
\bar{\theta}_{g}&\mbox{ if }&0<x<\frac{f(\bar{\theta}_{g})-f(\theta_{g})}{\bar{\theta}_{g}-\theta_{g}}t,\\
\theta_{f}&\mbox{if}&x>\frac{f(\bar{\theta}_{g})-f(\theta_{g})}{\bar{\theta}_{g}-\theta_{g}}t,
\end{array}\right.
\end{eqnarray*}
here $y(x, t)=x$ for $x\in(-\f, 0)\cup(R_{0}, \f)$. Since from theorem \ref{theorem1.1} we have 
\begin{equation*}
      \inf_{\al \in\ti{\mathcal{A}}}\mathcal{\ti{J}}(\al)\le \inf_{u_0 \in\mathcal{A}}J(u_0)
       \ \mbox{and}\ 
  \inf_{u_{0}\in\mathcal{A}}J(u_{0})\le J(w_{0})=M_{1}.
\end{equation*}
This proves \descref{thm1-step-1}{Step 1}.
\end{proof}
\descitem{Step 2:}{thm1-step-2} Let $\mathcal{\tilde{A}}_{1}=\{\alpha\in\mathcal{\ti{A}}: \mathcal{\tilde{J}}(\alpha)\le 2M_{1}\}$, then there exists a constant $M_{2}>0$, $c_{1}=\max\{M_{2}, c\}$ such that  for all $\alpha=(T, R_{1}, R_2, y(\cdot))$ with $\mathcal{\tilde{J}}(\alpha)\le 2M_{1}$,
\begin{description}
	\descitem{(i)}{step2i} $R_{1}\le M_{2}$,\\
	\descitem {(ii)}{step2ii} $|y(0+)|\le (18T^2M_{1})^{1/3}$,\\
	\descitem{(iii)}{step2iii} $y(-c_{1})\ge -(c_{1}+(6T^2M_{1})^{1/3})$,\\
	\descitem{(iv)}{step2iv} $y(c_{1})\le (c_{1}+(12M_{1}T^2)^{1/3})$.
\end{description}
\begin{proof}[Proof of \descref{thm1-step-2}{Step 2}]
	Suppose $R_{1}>c+f'(\bar{\theta}_{g})T$, then the line $R_{1}+f'(\bar{\theta}_{g})(t-T)$ does not intersect the $t$ axis for $t>0$. Subsequently, we have $R_{2}=R_{1}$ and $t(R_1-)\le t_+(x) \leq t(c+)<T$, for all $x\in (c, R_{1})$. Since $K(x)=0$ for $x>c$, we get
	\begin{eqnarray*}
	2M_{1}\ge \mathcal{\tilde{J}}(\alpha)&\ge& \int\limits_{c}^{R_{1}}\left|\frac{y(x)}{t_+(x)}\right|^2dx,\\
	&=&\int\limits_{c}^{R_{1}}\left|h_{+}\left(\frac{x}{T-t_+(x)}\right)\right|^2dx,\\
	&\ge&\int\limits_{c}^{R_{1}}\left|h_{+}\left(\frac{x}{T-t_+(R_1-)}\right)\right|^2dx\rr\f, \mbox{ as $R_{1}$ }\rr\f,
	\end{eqnarray*}
	which is a contradiction. Therefore, there exists $M_{2}>0$ such that $R_{1}\le M_{2}$.
	
	Denote $y(0+)=y(0)$ and for $y(0)<x<0$, then we have $y(x)\le y(0)<x<0$ and $0\leq x-y(0)\le x-y(x)$. This gives
	\begin{eqnarray*}
	2M_{1}\ge \mathcal{\tilde{J}}(\alpha)&\ge& \int_{y(0)}^{0}\left|\frac{x-y(x)}{T}-K(x)\right|^2dx,\\
	&=& \frac{1}{2}\int_{y(0)}^{0}\left|\frac{x-y(x)}{T}\right|^2dx- \int_{y(0)}^{0}\left|K(x)\right|^2dx,\\
	&\ge&\frac{1}{2}\int_{y(0)}^{0}\left|\frac{x-y(0)}{T}\right|^2dx- M_{1}.\\
	\end{eqnarray*} 
	Hence 
	\begin{equation*}
	18M_{1}T^{2}\ge 3\int_{y(0)}^{0}(x-y(0))^{2}dx=-y(0)^{3}.
	\end{equation*}
	This proves \descref{step2i}{(i)}.
	
	Now we write $y(-c_{1})=y(-c_{1}+)$, $y(c_{1})=y(c_{1}-)$. If $y(-c_{1})\ge -c_{1}$ then \descref{step2iii}{(iii)} is obvious. Thus we assume that $y(-c_{1})<-c_{1}$, then for $y(-c_{1})<x<-c_{1}$, we have $y(x)\le y(-c_{1})$ and $x-y(x)\ge x-y(-c_{1})\ge0$. Since $K(x)=0$ for $x<-c_{1}$, we obtain 
	\begin{eqnarray}
	2M_{1}\ge \mathcal{\tilde{J}}(\alpha)&\ge& \int_{y(-c_{1})}^{-c_{1}}\left|\frac{x-y(x)}{T}\right|^2dx,\\
	&\ge& \frac{1}{T^2}\int_{y(-c_{1})}^{-c_{1}}\left| x-y(-c_{1})\right|^2dx,\\
	&=&\frac{1}{3T^2}\left(-c_{1}-y(-c_{1})\right)^3.
	\end{eqnarray}  
	That is,
	\begin{equation*}
	y(-c_{1})\ge -(c_{1}+(16M_{1}T^2)^{1/3}).
	\end{equation*}
	This proves \descref{step2iii}{(iii)}.
\end{proof}

Similarly if $y(c_{1})\le c_{1}$, there is nothing to prove. Hence we assume that $c_{1}<y(c_{1})$. By the choice of $c_{1}$, we have $c_{1}>R_{1}$ and subsequently we get $y(c_{1})\le y(x)$ for $c_{1}<x<y(c_{1})$. Now it follows that $x-y(x)\le x-y(c_{1})\le 0$. Due to $c_{1}>c$ we have $K(x)=0$ for $x\in(c_{1}, y(c_{1}))$, therefore 
	\begin{eqnarray*}
2M_{1}\ge \mathcal{\tilde{J}}(\alpha)&\ge& \int\limits_{c_{1}}^{y(c_{1})}\left|\frac{x-y(x)}{T}\right|^2dx,\\
&=& \frac{1}{T^2}\int\limits_{c_{1}}^{y(c_{1})}\left| x-y(x)\right|^2dx,\\
&\geq & \frac{1}{T^2}\int\limits_{c_{1}}^{y(c_{1})}\left| x-y(c_{1})\right|^2dx,\\
&=&\frac{1}{3T^2}\left(y(c_{1})-c_{1}\right)^3.
\end{eqnarray*} 
Thus $y(c_{1})\le (c_{1}+(6T^2M_{1})^{1/3})$ and it proves \descref{step2iv}{(iv)}. This completes the proof of \descref{thm1-step-2}{Step 2}.

\descitem{Step 3:}{thm1-step-3} Define $M_{3}=c_{1}+(18T^2M_{1})^{1/3}$ and 
\begin{equation*}
\mathcal{\tilde{A}}_{2}=\{\alpha=(T, R_{1}, R_2, y(\cdot))\in\mathcal{\tilde{A}}: y(x)=x \mbox{ if }x\notin[-M_{3}, M_{3}], R_{1}\le M_{2}, |y(0+)|\le(18M_{1}T^2)^{1/3}\}.
\end{equation*} Then
\begin{equation*}
\inf_{\alpha\in\mathcal{\tilde{A}}_{2}}\mathcal{\tilde{J}}(\alpha)\le \inf_{\alpha\in\mathcal{\ti{A}}}\mathcal{\ti{J}}(\alpha).
\end{equation*}
\begin{proof}[Proof of \descref{thm1-step-3}{Step 3}]
	From \descref{thm1-step-1}{Step 1}, we have 
	\begin{equation*}
	\inf_{\alpha\in\mathcal{\tilde{A}}_{1}}\mathcal{\tilde{J}}(\alpha)\le \inf_{\alpha\in\mathcal{\tilde{A}}}\mathcal{\tilde{J}}(\alpha).
	\end{equation*}
	Let $M_{1}, M_{2}, M_{3}$ and $c_{1}$ defined as above. Let $\alpha=(T, R_{1}, R_2, y(\cdot))\in \mathcal{\tilde{A}}_{1}$. Then from \descref{thm1-step-2}{Step 2}, $R_{1}\le M_{2}$, $|y(0+)|\le(18M_{1}T^2)^{1/3}$ and, $y(-c_{1})\ge -M_{3}$, $y(c_{1})\le M_{3}$. Let $c_1>R_1$ and define
	\begin{eqnarray*}
	\tilde{y}(x)=\left\{\begin{array}{llll}
	y(-c_{1})&\mbox{if}& x\in(\min(-c_{1}, y(-c_{1})), -c_{1}),\\
	y(x)&\mbox{if}& x\in(-c_{1}, R_{2})\cup(R_{1}, c_{1}),\\
	y(c_{1})&\mbox{if}& x\in(c_{1}, \max(c_{1}, y(c_{1}))),\\
	x&&\mbox{otherwise}.
	\end{array}\right.
	\end{eqnarray*}
	Then $\tilde{\alpha}=(T, R_{1}, R_{2}, \tilde{y}(\cdot))\in \mathcal{\tilde{A}}_{2}$ and 
	\begin{eqnarray*}
	\mathcal{\tilde{J}}(\tilde{\alpha})-\mathcal{\tilde{J}}(\alpha)&\le&\int\limits_{\min(-c_{1}, y(-c_{1}))}^{-c_{1}}\left|\frac{x-y(-c_{1})}{T}\right|^2dx-\int\limits_{\min(-c_{1}, y(-c_{1}))}^{-c_{1}}\left|\frac{x-y(x)}{T}\right|^2dx\\&+&\int\limits_{c_{2}}^{\max(c_{2}, y(c_{2}))}\left|\frac{x-y(c_{2})}{T}\right|^2dx-\int\limits_{c_{1}}^{\max(c_{1}, y(c_{1}))}\left|\frac{x-y(x)}{T}\right|^2dx\\
	&\le&0.
	\end{eqnarray*}
	Since $y(x)\le y(-c_{1})$, for $x\in(\min(-c_{1}, y(-c_{1})),-c_{1})$ and  $y(x)\ge y(c_{1})$ for $x\in(c_{1}, \max(c_{1}, y(c_{1})))$, we get
	\begin{equation*}
	\mathcal{\tilde{J}}(\tilde{\alpha})\le \mathcal{\tilde{J}}(\alpha).
	\end{equation*}
	Due to $\mathcal{\tilde{A}}_{2}\subset\mathcal{\tilde{A}}$, we have 
	\begin{equation*}
	\inf_{\alpha\in\mathcal{\tilde{A}}_{2}}\mathcal{\tilde{J}}(\alpha)\le \inf_{\alpha\in\mathcal{\tilde{A}}}\mathcal{\tilde{J}}(\alpha)\le \inf_{\alpha\in\mathcal{\tilde{A}}}\mathcal{\tilde{J}}(\alpha)\le \inf_{\alpha\in\mathcal{\tilde{A}}_{2}}\mathcal{\ti{J}}(\alpha).
	\end{equation*}
	This proves 
	\begin{equation*}
	\inf_{\alpha\in\mathcal{\tilde{A}}_{2}}\mathcal{\tilde{J}}(\alpha)= \inf_{\alpha\in\mathcal{\tilde{A}}}\mathcal{\ti{J}}(\alpha).
	\end{equation*}
\end{proof}
\descitem{Step 4:}{thm1-step-4} Let $\{\alpha_{k}\}\in\mathcal{\tilde{A}}_{2}$ be a sequence such that
\begin{equation*}
\lim\limits_{k\rr\f}\mathcal{\tilde{J}}(\alpha_{k})= \inf_{\alpha\in\mathcal{\tilde{A}}}\mathcal{\ti{J}}(\alpha).
\end{equation*} 
Let $\alpha_{k}=(T, R_{1,k}, R_{2,k}, y_{k}(\cdot))$, as $\alpha_{k}\in\mathcal{\tilde{A}}_{3}$, we have $\{R_{1,k}\}, \{R_{2,k}\}$ are bounded and $y_{k}|_{[-M_{3}, M_{3}]}$ is a bounded non-decreasing function. Hence for a subsequence still denoted by $\alpha_{k}$ such that $\alpha_{k}\rr\alpha_{0}=(T, R_{1}, R_2, y(\cdot))\in\mathcal{\tilde{A}}_{2}\subset\mathcal{\tilde{A}}$ and 
\begin{equation*}
\mathcal{\tilde{J}}(\alpha_{0})=\inf_{\alpha\in\mathcal{\tilde{A}}}\mathcal{\tilde{J}}(\alpha).
\end{equation*}
Since $\alpha_{0}\in\mathcal{\tilde{A}}$ we get $\alpha_{0}=(T, R_{1}, R_2, y(\cdot))\in R(T)$, therefore from Theorem \ref{theorem1.1} there exists a $u_{0}\in L^{\f}(\R)$ and the corresponding solution $u$ of \eqref{conlaw-equation} satisfying $R_{1}=R_{1}(T)$, $R_{2}=R_{2}(T)$ and $y(x)=y(x, T)$. As $y(x)=x$ for $x\in(-M_{3}, M_{3})$ we obtain 
\begin{eqnarray}
u_{0}(x)=\left\{\begin{array}{lll}
\theta_{g}&\mbox{if}&x<-M_{3},\\
\theta_{f}&\mbox{if}&x>M_{3},
\end{array}\right. 
\end{eqnarray}
then $u_{0}\in\mathcal{A}$. Hence
\begin{equation*}
\mathcal{J}(u_{0})=\inf_{w_{0}\in\mathcal{A}}\mathcal{J}(w_{0})
\end{equation*}
has a solution. This proves the theorem.
	\end{description}
\end{proof}
%Let $u(x, t)$ be the solution of \eqref{conlaw-equation} with initial data $u_{0}$. Suppose $R_{1}(T)=0$ and $\gamma(t)=f'(\bar{\bar{\theta}}_{g})(t-T)$, $\xi_{0}=\gamma(0)=-Tf'(\bar{\bar{\theta}}_{g})$.\\
%{\bf Claim}: $y(0+, T)\ge\xi_{0}$.\\
%{\bf Case(i)}: Suppose there exists a $T_{k}<t_{k}<T$ such that $R_{1}(t_{k})>0$, $\frac{R_{1}(t_{k})}{T-T_{k}}=f'(\bar{\theta}_{g})$. Let $(\xi_{k}, \tau_{k})$ be as in lemma *. Hence from the non-intersecting of characteristics, we have 
%\begin{equation*}
%\xi_{k}\le y(R_{1}(t_{k})+, t_{k}).
%\end{equation*}
%Since $\xi_{k}\rr\xi_{0}$, $t_{k}\rr T$ as $k\rr\f$ and by stability 
%\begin{eqnarray}
%\xi_{0}\le\lim\limits_{k\rr\f}\xi_{k}\le \lim\limits_{k\rr\f}y(R_{1}(t_{k})+, t_{k})\le \tilde{y}(0, T_{k})\le y(0+, T).
%\end{eqnarray}
%{\bf Case (ii)}: Suppose there exists an $\e_{0}>0$ such that case (i) does not hold for $t\in(T-\e_{0}, T)$. Hence $R_{1}(T)=0$ for all $t\in(T-\e_{0}, T)$ and $\gamma(\theta)=-\frac{y(0+, t)}{t}(\theta-t)\in ch(0, t)$ and for a.e., $t\in(T-\e_{0}, T)$,
%\begin{equation*}
%f'(u(0+,t))=-\frac{y(0+, t)}{t}.
%\end{equation*}
%Since $g(u(0-, t))=f(u(0+, t))$ and $f'(u(0+, t))\le0$, hence $u(0+, t)\le \bar{\bar{\theta}}_{g}$. Hence $-\frac{y(0+, t)}{t}=f'(u(0+, t))\le f'(\bar{\bar{\theta}}_{g})$ implies that $y(0+, t)\ge-tf'(\bar{\bar{\theta}}_{g})$. Letting $t\rr T$ to obtain $y(0+, T)\ge -Tf'(\bar{\bar{\theta}}_{g})=\xi_{0}$. This proves the claim.
%% Hence $y(0+, T)\ge -Tf'(\bar{\bar{\theta}}_{g})=\xi_{0}$. This proves the claim. 
%Hence $(T, R_{1}(T), R_{2}(T), y(x, T))\in R(T)$. This proves the lemma. 
\section{Reachable set for $(A,B)$ connection}\label{connection} 
\begin{definition}
Let $A\geq \T_f$, $B\leq \T_g$ is called a $(A,B)$ connection if $f(A)=g(B)$.
\end{definition}
So far in this article we considered the case $A=\T_f$ or $B=\T_g$. Therefore from now onwards  we assume that $f'(A)>0$, $g'(B)<0$. 
\begin{definition}
$u$ is called a $(A,B)$ entropy solution of (\ref{conlaw-equation}) with initial data $u_0$ if $u$ is the solution obtained from the Hamilton-Jacobi method as in \cite{Jde}, associated to given $(A,B)$ connection.
\end{definition}
\par Let $L_1\leq R_1$ and $0\leq T_1 , T_2 \leq T$ be such that $f'(B)=\frac{R_1}{T-T_1}$, $g'(A)=\frac{L_1}{T-T_2}.$ Let $\bar{B} \leq \T_f\leq  B$, $A\leq \T_g \leq \bar{A}$ be such that $f(B)=f(\bar{B})$, $g(A)=g(\bar{A})$. Let $(\tau_0^+, \xi_0^+, s_{\xi_0}^+)$,  $(\tau_0^-, \xi_0^-, s_{\xi_0}^-)$ be constructed as in lemma \ref{3.6} for $(R_1, T_1)$ with $\bar{\al}=B$ (for the flux $f$) and $(L_1, T_2)$ with  $\bar{\bar{\al}}=A$ (for the flux $g$) respectively. 
\begin{definition}(Reachable set) Let $(T, L_1, R_1, y(\cdot))$ is called an element in the reachable set $\mathcal{R}^{A,B}(T)$ if they satisfy one of the following conditions:
\begin{itemize}
\item[1.] $y:(-\f, L_1)\cup (R_1,\f)\rr \R$ be a non decreasing function such that 
\begin{eqnarray*}
\left\{\begin{array}{lll}
y(x)\leq 0 &\mbox{if}& x<L_1,\\
y(x)\geq 0 &\mbox{if}& x>R_1.\end{array}\right.
\end{eqnarray*}
Suppose there exist $0\leq T_1, T_2 \leq T$ such that 
\begin{eqnarray*}
f'(B)=\frac{R_1}{T-T_1}, g'(A)=\frac{L_1}{T-T_2}, 
\end{eqnarray*}
then 
\begin{eqnarray}\label{condition}
y(L_1-,T)\leq \xi_0^-, \ y(R_1+,T)\geq \xi_0^+.
\end{eqnarray}
\item[2.] If $R_1\geq 0, L_1=0$, then $y:(-\f,R_1)\cup(R_1,\f)\rr \R$ be a non decreasing function with 
\begin{eqnarray*}
\left\{\begin{array}{lll}
y(x)\leq 0 &\mbox{if}& x<R_1,\\
y(x)\geq 0 &\mbox{if}& x>R_1,\\
y(L_1-, T) \geq \xi_0^-.
\end{array}\right.
\end{eqnarray*}
\item[3.] If $R_1=0$, $L_1\leq 0$, then $y:(-\f, L_1)\cup (L_1,\f)\rr \R$ be a non decreasing function with 
\begin{eqnarray*}
\left\{\begin{array}{lll}
y(x)\leq 0 &\mbox{if}& x<L_1,\\
y(x)\geq 0 &\mbox{if}& x>L_1,\\
y(R_1+, T) \geq \xi_0^+.
\end{array}\right.
\end{eqnarray*}
\item[4.] In all the cases, the following must hold:
\begin{eqnarray*}
\sup_x |x-y(x)|<\f.
\end{eqnarray*}
\end{itemize}
\end{definition}
Then we have the following:
\begin{theorem}[characterization of $\mathcal{R}^{A,B}(T)$] $(T,L_1, R_1,y(\cdot))\in $$\mathcal{R}^{A,B}(T)$ if and only id there exist a $u_0\in L^\f(\R)$ and the corresponding $(A,B)$ entropy solution $u$ of (\ref{conlaw-equation}) satisfy 
\begin{eqnarray*}
(T, L_1, R_1,y(\cdot)) = (T, L_1(T), R_1(T), y(\cdot,T)), 
\end{eqnarray*}
where $(L_1(T), R_1(T), y(\cdot, T))$ are defined by $u$.
\end{theorem}
As earlier we can decompose the domain $\R\times (0,T)$ into three disjoints regions $D_1, D_2$ and $D_3$. Here we only sketch the proof of backward construction and the rest follows as earlier. 
\begin{itemize}
\item[I.] Backward construction (continuous and shock solutions): Define 
\begin{itemize}
\item[(i)] $h_+:[f'(B),\f)\rr [g'(\bar{A},\f)$ by $h_+= g'\circ g_+^{-1}\circ f \circ f'.$
\item[(ii)]  $h_-:(-\f,g'(A)]\rr (-\f,f'(\bar{A})]$ by $h_-= f' \circ f_-^{-1} \circ g \circ g'.$
\end{itemize}
Then $h_\pm$ are isomorphisms and by R-H condition across the interface, using $h_\pm$ it follows as in earlier case 
\begin{itemize}
\item[(i)] There are no forward rarefaction from the interface.
\item[(ii)] Continuous and shock solutions are constructed.
\item[(iii)] Using this and $L^1$-contractivity for solutions with discrete set of discontinuities of the solution, one can glue them to obtain a solution in $D_2\cup D_3$, where $D_2$ and $D_3$ are described earlier. 
\end{itemize}

\item[II.] Backward construction in $D_1$. This is the case where the $(A, B)$ entropy exist. Assume that $(T,L_1,R_1,y(\cdot))\in \mathcal{R}^{A,B}(T)$ satisfies (\ref{condition}) with 
\begin{eqnarray*}
y(x)=\left\{\begin{array}{lll}
y_+ &\mbox{if}& x>R_1,\\
y_- &\mbox{if}& x<L_1.
\end{array}\right.
\end{eqnarray*}
$f'(B)=\frac{R_1}{T-T_1}$, $g'(A)=\frac{L_1}{T-T_2}.$
Let $(\tau_0^\pm, \xi_0^\pm, s_{\xi_0}^\pm)$ be as defined earlier. Define 
\begin{eqnarray*}
f'(\beta_-)&=&\frac{R_1-\xi_0^+}{T},\ g'(\beta_-)=\frac{L_1-\xi_0^-}{T},\\
f'(u_+) &=& \frac{R_1-y_+}{T}, \ g'(u_-)=\frac{L_1-y_-}{T},\\
\ga_1(t)&=&R_1+\frac{f(u_+)-f(\beta_+)}{u_+-\beta_+}(t-T),\\
\ga_2(t)&=& R_1 + f'(\beta_+)(t-T),\\
\ga_3(t) &=& -f'(\bar{\beta})(t-\tau_0^-),\\
\eta_3(t) &=& -g'(\bar{A})(t-\tau_0^+),\\
\eta_2(t) &=& L_1 + g'(\beta_-)(t-T),\\
\eta_1(t) &=& L_1+ \frac{g(\beta_)--g(u_-)}{\beta_--u_-}(t-T).
\end{eqnarray*}
Since $(T, L_1, R_1, y(\cdot))\in \mathcal{R}^{A,B}(T)$, hence by condition (\ref{condition}) and convexity of $f$
 and $g$, we have $y_-\leq \xi_1^- \leq \xi_0^-\leq 0 \leq \xi_0^+ \leq \xi_1^+ \leq y_+$ where 
 $\xi_1^-=\eta_2(0)$, $\xi_1^+=\xi_2(0)$. 
 Define 
 \begin{eqnarray*}
 u_{1,0} (x)=\left\{\begin{array}{llll}
 u_- &\mbox{if}& x<\xi_1^-,\\
 \beta_- &\mbox{if}& \xi_1^-<x\xi_0^-,\\
 \bar{A} &\mbox{if}& \xi_0^-<x<0,\\
 \bar{B} &\mbox{if}& 0<x<\xi_0^+,\\
 \beta_+ &\mbox{if}& \xi_0^+<x<\xi_1^+,\\
 u_+ &\mbox{if}& x>\xi_1^+
 \end{array}\right.
 \end{eqnarray*}
and for $0<t\leq T,$ (see figure \ref{fig-12} for illustration)
 \begin{eqnarray*}
 u_{1} (x,t)=\left\{\begin{array}{lllllllllllllllll}
u_- &\mbox{if}& x<\eta_1(t),\\
\beta_- &\mbox{if}& \eta_1(t)<x<\eta_2(t),\\
g'^{-1}\left(\frac{x-\xi_0^-}{t}\right) &\mbox{if}& \eta_2(t)<x<\xi_0^-(t), \tau_0^- \leq t <T,\\
&\mbox{if}& \eta_2(t)<x<\eta_3(t), 0<t<\tau_0^-,\\
\bar{A} &\mbox{if}& \eta_3(t)<x<0, 0<t<\tau_0^- ,\\
A &\mbox{if}& s_{\xi_0^-}(t)<x<0, \tau_0^-<x<0, \\
B &\mbox{if}& 0<x<s_{\xi_0^+}(t), \\
\bar{B}&\mbox{if}& 0<x<\ga_3(t), 0<t<\tau_0^+ \\
 f'^{-1}\left(\frac{x-\xi_0^+}{t}\right)  &\mbox{if}& s_{\xi_0^+}(t)<x<\ga_2(t), \tau_0^+\leq t <T\\
&\mbox{if}& \ga_3(t)<x<\ga_2(t), 0<t<\tau_0^+\\
\beta_+&\mbox{if}& \ga_2(t)<x<\ga_3(t)\\
u_+ &\mbox{if}& x>\ga_1(t).
 \end{array}\right.
 \end{eqnarray*}
Then $u_1$ is the $(A,B)$ entropy solution of (\ref{conlaw-equation}) with $u_{1,0}$ as the  initial data. 
\end{itemize}

\begin{figure}
	\centering
	\begin{tikzpicture}
	\draw (-7,0)--(7,0);
	\draw (-7,6)--(7,6);
	\draw (0,0)--(0,6);
	%left hand side
	\draw [red](0,3)--(-4.5,6);
	\draw [red](0,3)--(0,1);
	\draw [red](-1,0)--(0,1);
	\draw (-1,0)--(-4.5,6);
	\draw (-1,0)--(-3.6,5.15);
	\draw (-1,0)--(-3,4.6);
	\draw (-1,0)--(-2.5,4.15);
	\draw (-1,0)--(-2,3.7);
	\draw (-1,0)--(-1.5,3.2);
	\draw (-1,0)--(-.95,2.6);
	\draw (-1,0)--(-.5,2.05);
	\draw (-1,0)--(-.2,1.55);
	\draw (0,1) .. controls (-.3,2.5) and (-3,4.5) .. (-4.5,6);
	\draw (4,0)--(4.5,6);
	\draw (-5.5,0)--(-4.5,6);
	\draw (-3.5,0)--(-4.5,6);
	\fill[white!40!white,rotate around={30:(-1,0.5)}] (-1,.6) rectangle (-0.2,3);
	\coordinate (P) at (-1.2,1.7);
	\draw (P) node[rotate=-60] (N) {\tiny $(g')^{-1}\left(\frac{x-\xi_{0}^{-}}{t}\right)$};
	%right hand side
	\draw[red] (0,4)--(4.5,6);
	\draw[red](0,4)--(0,2);
	\draw [red](1.5,0)--(0,2);
	\draw (1.5,0)--(4.5,6);
	\draw (1.5,0)--(3.6,5.45);
	\draw (1.5,0)--(3,5.11);
	\draw (1.5,0)--(2.5,4.83);
	\draw (1.5,0)--(2,4.5);
	\draw (1.5,0)--(1.5,4.15);
	\draw (1.5,0)--(.95,3.7);
	\draw (1.5,0)--(.5,3.17);
	\draw (1.5,0)--(.2,2.65);
	\draw (0,2) .. controls (.3,4) and (3,5) .. (4.5,6);
	\draw (6,0)--(4.5,6);
	\fill[white!40!white] (.6,2) rectangle (2.5,3);
	\draw (1.5,2.5)node{\tiny $(f')^{-1}\left(\frac{x-\xi_{0}^{+}}{t}\right)$};
	%numbering
	%x-axis
	\draw (0,-.2)node{\scriptsize $x=0$};
	\draw (-1,-.2)node{\scriptsize $\xi_{0}^{-}$};
	\draw (-3.5,-.2)node{\scriptsize $\xi_{1}^{-}$};
	\draw (1.5,-.2)node{\scriptsize $\xi_{0}^{+}$};
	\draw (4,-.2)node{\scriptsize $\xi_{1}^{+}$};
	\draw (6,-.2)node{\scriptsize $y_{+}$};
	\draw (-5.5,-.2)node{\scriptsize $y_{-}$};
	\draw (6,0) node{\tiny $\bullet$};
	\draw (4,0) node{\tiny $\bullet$};
	\draw (1.5,0) node{\tiny $\bullet$};
	\draw (-1,0) node{\tiny $\bullet$};
	\draw (-3.5,0) node{\tiny $\bullet$};
	\draw (-5.5,0) node{\tiny $\bullet$};
	\draw (0,0) node{\tiny $\bullet$};
	%y-axis
	\draw (.25,1)node{\scriptsize $\tau^{-}_{0}$};
	\draw (-.2,2)node{\scriptsize $\tau_{0}^{+}$};
	\draw (.2,3)node{\scriptsize $T_{2}$};
	\draw (-.15,4)node{\scriptsize $T_{1}$};
	%left hand side
	\draw (-1.8,1)node{\scriptsize $\eta_{2}$};
	\draw (-3.4,1)node{\scriptsize $\eta_{1}$};
	\draw (-.4,.3)node{\scriptsize $\eta_{3}$};
	\draw (-1,3.2)node{\scriptsize $A$};
	\draw (-1,5.5)node{\scriptsize $A$};
	\draw (-4.5,6.25)node{\scriptsize $L_{1}$};
	%right hand side
	\draw (4.5,6.25)node{\scriptsize $R_{1}$};
	\draw (3.8,1)node{\scriptsize $\gamma_{1}$};
	\draw (2.3,1)node{\scriptsize $\gamma_{2}$};
	\draw (.8,.65)node{\scriptsize $\gamma_{3}$};
	\draw (1,4)node{\scriptsize $B$};
	\draw (1,5)node{\scriptsize $B$};
	\draw[<-](1.8,4.5)--(-1,4.5);
	\draw (-1.2, 4.5)node{\scriptsize $s_{\xi_{0}}^{+}$};
	\draw[<-](-1.8,3.7)--(.1,3.7);
	\draw (.35, 3.7)node{\scriptsize $s_{\xi_{0}}^{-}$};
	\draw (4.5,6) node{\tiny $\bullet$};
	\draw (-4.5,6) node{\tiny $\bullet$};
	\end{tikzpicture}
		\caption{This figure illustrates backward construction with reflected characteristics.}\label{fig-12}
%	\captionof{figure}{}
\end{figure}
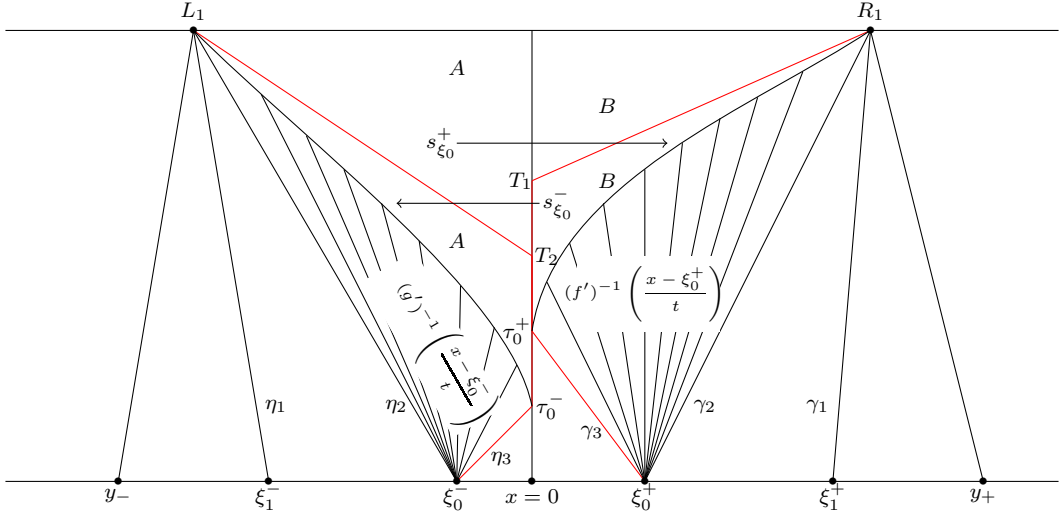

\section{Appendix}\label{appendix}
\begin{proof}[Proof of lemma \ref{stability}]

From the hypothesis on $\{f_k\}$ and $\{g_k\}$, it follows that $\lim\limits_{k\rr \f} (f_k^*, g_k^*)\rr (f^*,g^*)$ in $C^1_{\mbox{loc}}(\R)$. 
Since $\{u_{0,k}\}$ is uniformly bounded in $L^{\f}(\R)$ and converges to $u_{0}$ in $weak^{*}- L^{\f}(\R)$. Hence $\{v_{0,k}\}$ converges to $v_{0}$ uniformly on compact subsets of $\R$ and having uniformly  Lipschitz constant. Hence $\{v_{k}\}$ are having uniformly Lipschitz constant. Hence by Arzela-Ascoli theorem, there exists a subsequence still denoted by $\{v_{k}\}$ converges to $w$ in $C^{0}_{loc }(\R\times [0,\f))$.\\
{\bf Claim}: $\lim\limits_{k\rr\f}ch_{k}(x,t)\subset ch(x,t)$, $v=w$.\\
For $\gamma_{k}\in ch_{k}(x,t)$, then from lemma 4.2 of \cite{Kyoto} (page 38), $\left\{\frac{d \gamma_{k}}{d \theta}\right\}$ is uniformly bounded and hence for subsequence $\{\gamma_{k_{i}}\}$ converges to $\tilde{\gamma}\in ch (x,t)$. In order to prove the claim we need to show that $\tilde{\gamma}\in ch (x,t)$. If $\gamma\in c (x,t)$ then $v_{k}(x,t)=\Gamma_{v_{0,k},\gamma_{k}}(x,t)\le \Gamma_{v_{0,k,\gamma}}(x,t)$. Letting $k=k_{i}$ and $k_{i}\rr\f$  to obtain 
\begin{equation*}
w(x,t)=\lim\limits_{k_{i}\rr\f}v_{k_{i}}(x,t)=\Gamma_{v_{0},\tilde{\gamma}}(x,t)\le \Gamma_{v_{0},\gamma}(x,t).
\end{equation*}
Hence $\ti{\gamma}\in ch(x,t)$  and 
\begin{equation*}
w(x,t)=\inf_{\gamma\in c(x,t)} \Gamma_{v_{0},\gamma}(x,t)=v(x,t).
\end{equation*}
This proves the claim. Hence by uniqueness of the limit, it follows that $\lim\limits_{k\rr\f}v_{k}=v$ in $C^{0}_{loc}(\R)$ and $\lim\limits_{k\rr\f}ch_{k}(x,t)\subset ch(x,t)$. Since Lipschitz constant of $\{v_{k}\}$ are uniformly bounded, hence for any $\varphi\in C^{1}_{0}(\R\times(0,\f))$, we have for $\Omega=\R\times(0,\f)$
\begin{eqnarray}
\lim\limits_{k\rr\f}\int\limits_{\Omega}^{}\frac{\partial v_{k}}{\partial x}\varphi dxdt&=&-\lim\limits_{k\rr\f}\int\limits_{\Omega}^{}v_{k}\frac{\partial \varphi}{\partial x} dxdt\\
&=&-\int\limits_{\Omega}^{}\lim\limits_{k\rr\f}v_{k}\frac{\partial \varphi}{\partial x} dxdt\\
&=&-\int\limits_{\Omega}^{}v\frac{\partial \varphi}{\partial x} dxdt.
\end{eqnarray}
Hence $\frac{\partial v_{k}}{\partial x}\rr \frac{\partial v}{\partial x}$ in $D'(\Omega)$. This proves the lemma.
\end{proof}

\section*{Acknowledgments} The author A is supported by Raja Ramanna fellowship.  The author  SSG  acknowledge the support of the Department of Atomic Energy, Government of India, 
under project no. 
$12$-$R \&$  D-TFR-5.01-0520. 
 SSG would like to thank
Inspire faculty-research grant DST/INSPIRE/04/2016/000237. 
The authors 
would like to thank Gran Sasso Science Institute, L'Aquila, Italy for the support.  Authors would like to thank Pierangelo Marcati for enlightening discussions and helping while writing the initial version of the article.

\end{document}